%% file: AffineCurvature.tex
\documentclass{amsart}

\include{Preamble}
\graphicspath{{figures/}}
\usepackage{eqnarray}

\begin{document}

\title
{Numerical methods for motion of level sets by affine curvature}

\author{Adam M. Oberman}
\thanks{Department of Mathematics and Statistics, McGill University, 805 Sherbrooke Street West, Montreal, Quebec, H3A 0G4, Canada ({\tt adam.oberman@mcgill.ca})}

\author{Tiago Salvador}
\thanks{Department of Mathematics and Statistics, McGill University, 805 Sherbrooke Street West, Montreal, Quebec, H3A 0G4, Canada ({\tt tiago.saldanhasalvador@mail.mcgill.ca}). This author partially supported by FCT doctoral grant SFRH / BD / 84041 /2012.}

\date{\today}

\begin{abstract}
We study numerical methods for the nonlinear partial differential equation that governs the motion of level sets by affine curvature.   We show that standard finite difference schemes are nonlinearly unstable.  We build convergent finite difference schemes, using the theory of viscosity solutions.   We demonstrate that our approximate solutions capture the affine invariance and morphological properties of the evolution.  Numerical experiments demonstrate the accuracy and stability of the discretization. 
\end{abstract}

\maketitle

\begin{figure}
\includegraphics[width=0.65\textwidth]{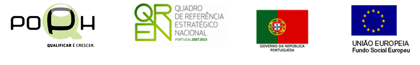}
\end{figure}

\tableofcontents

\section{Introduction}

\subsection{The affine curvature PDE}
The affine curvature evolution is one of the most fundamental geometric evolution equations, after the mean curvature evolution.
It was introduced by Sapiro and Tannenbaum in \cite{sapiro1994affine} and \cite{angenent1998affine} and has applications in mathematical morphology, edge detection, image smoothing, and image enhancement, see~\cite{SapiroBook}. 

We are motivated by recent work of Jeff Calder \cite{caldersmart2016}, which provides an application of the affine curvature PDE to the statistics of large data sets.  The convex hull peeling algorithm  \cite{chazelle1985convex} provides an affine invariant notion of the median and the quantiles of multidimensional probability distributions.  While there is more than one way measure data depth  \cite{barnett1976ordering}, affine invariance is an important property for such measures \cite{liu1999multivariate}.  The level sets of the solution of the affine curvature PDE, with right hand side given by the probability density $\rho$, also gives an affine invariant notion of the depth of $\rho$.   According to \cite{caldersmart2016}, these two notions of depth are equivalent: the rescaled data depth layers of $N$ data points sampled from the density, $\rho$,  given by convex hull peeling algorithm converge, in the limit $N \to \infty$, to the levels given by the solution of the PDE .  Compared to convex hull peeling, the PDE characterization is efficient in terms of the number of data points $N$: an efficient density estimation method can be used to approximate $\rho$, and afterwards the PDE solve does not depend on $N$.  This kind of limiting PDE approach has already been shown to be effective for non-dominated sorting~\cite{calder2014hamilton}.

 The planar motion of level sets by affine curvature is governed by the nonlinear partial differential equation (PDE)
\begin{equation}\label{ACPDE}\tag{AC}
u_t = \Aff[u] := \Abs{\nabla u} \left (k[u]\right)^{1/3}.
\end{equation}
Here $u = u(x,y):\R^2 \to \R$, $\nabla u = (u_x,u_y)$ denotes the gradient of $u$, and $k[u]$ denotes the curvature of the level set of $u$ 
\begin{equation}
\label{curvature}
k[u] = \divergence\left(\frac{\nabla u}{\Abs{\nabla u}}\right) = \frac{u_{xx}u_y^2-2u_x u_y u_{xy}+u_{yy}u_x^2}{(u_x^2+u_y^2)^{3/2}}.
\end{equation}
The affine curvature PDE is closely related to the well known PDE for motion of level sets by mean curvature
\begin{equation}\tag{MC}\label{MCevolution}
u_t =  \Delta_1 u := \Abs{\nabla u}k[u]
\end{equation}
studied in the seminal article \cite{Osher88frontspropagating}.   However, the PDE \eqref{ACPDE}  exhibits instabilities not found in the mean curvature PDE \eqref{MCevolution}, as demonstrated below.  In order to resolve these instabilities, we propose a Lipschitz regularization of the operator. 

The regularized operator is also a geometric PDE, and viscosity solutions converge to solutions of the affine curvature PDE in the limit as the regularization parameter goes to zero~\cite[Theorem 4.6.1]{giga2006surface}. The advantage of the regularized operator is that it allows us to build stable, convergent explicit solvers which are not otherwise available. Moreover, the resulting discretization can be combined  into a filtered scheme that achieves the higher  accuracy of the otherwise unstable standard finite difference scheme.  In addition, the numerical solutions exhibit the affine invariance and morphology properties of the evolution.

\subsubsection*{Our approach to the convergent discretization}
In this paragraph we present an overview of our approach to building an elliptic discretization for clarity.  The details and supporting theory can be found in the sections which follow.  
Elliptic discretizations are available for $\Delta_1 u$, rather than for $k[u]$, 
so we rewrite 
\begin{align}\label{AffPQ}
\Aff[u]	& = A(\Abs{\nabla u},\Delta_1 u) = (\Abs{\nabla u}^2 \Delta_1 u)^{1/3},
&& \text{ where } A(p,q) = \left(p^2 q \right)^{1/3}.
\end{align}
The goal is to make use of available elliptic discretizations of $\pm \Abs{\nabla u}$ and $\Delta_1 u$ to build an elliptic discretization of the full operator $\Aff[u]$.  However, simply inserting these operators into the function $A(p,q)$ is not sufficient: 
 elliptic schemes are built by composing non-decreasing maps with elliptic operators, and $A(p,q)$ fails to be non-decreasing.  Furthermore,  explicit time discretizations require Lipschitz continuous operators, and $A(p,q)$ also fails to be Lipschitz continuous.  

Since the properties of the nonlinear and singular function $A(p,q)$ are so important, we  first study a  model equation in one-dimension.   Define
\bq\label{AffPQ1D}
\Aff^{1D}[u] := A(u_x, u_{xx}) = \left(\Abs{u_x}^2 \:u_{xx}\right)^{1/3}
\eq
so that the $\Aff^{1D}[u]$ operator has the same homogeneity in first and second derivatives as the $\Aff[u]$ operator.  
Like the higher dimensional PDE, the model equation exhibits the instability of standard finite differences.

To build an elliptic scheme for the one dimensional equation, what is needed is a non-decreasing representation of the function $A(p,q)$, which is consistent with $\Aff^{1D}[u]$.   Furthermore, we need a Lipschitz continuous approximation of $A(p,q)$, with Lipchitz constant $K^h$, to build a monotone discretization of the time dependent PDE,  using a time step $dt \le 1/K^h$.  
Once this modified function is available, we proceed by inserting the discretization of the two dimensional operators into the modified function, which results in a convergent scheme.

%
%
%
%
%
%
%
%
%
%
%
%

\subsubsection*{Related numerical work.}

There are by now a large number of numerical methods for the level set mean curvature PDE~\eqref{MCevolution}, see the review papers \cite{surveyMC} and \cite{ciomaga2011image}.  In particular, Catt\'{e} and Dibos proposed a morphological scheme that satisfies the comparison principle~\cite{MorphologicalMC} and the first author presented a convergent wide stencil finite difference scheme \cite{ObermanMC} using a median formula.  

For the affine curvature evolution, the recent article \cite{EsedogluCurvature} gives a Bence-Merriman-Osher~\cite{MBOScheme} thresholding scheme. It introduced a different regularization of the cube root, which was needed for theoretical purposes, but not in practice.   Thresholding methods are effective for moving a given curve by the evolution, and allow for large time steps to be taken.  Using the level set representation  \eqref{ACPDE} moves every level set by the evolution, but requires a much smaller time step. 

Alvarez and Guichard proposed a local scheme which lacks the affine invariance property \cite{guichard1994axiomatisation}. A morphological scheme which generalized \cite{MorphologicalMC} was proposed by Guichard and Morel for affine curvature~\cite{guichard1997partial}. This inf-sup scheme, although morphologically invariant, has some limitations on the speed at which the level set curves move.  In \cite{MoisanScheme}, a nonlocal geometric morphological scheme is presented.

\subsection{Background mathematics}

 \subsubsection*{Euclidean and Affine Curvature}
We begin with a parametric description of the affine curvature evolution, and make a comparison with the more familiar mean curvature evolution.  We refer to \cite[Chapter 2]{SapiroBook} for more details. 
Consider a curve described  parametrically $\curve(s):[a,b] \in \R \to \R^2$ where $s$ parameterizes the curve. 
If the curve is parameterized by the Euclidean arc length, $\Abs{\frac{\td \curve}{\td s}} = 1,$
then the curvature, $k$, is defined, up to a sign, by $\Abs{k} := \Abs{\curve_{ss}}$. 
Letting $\T$ and $\N$ denote respectively the unit Euclidean tangent and the Euclidean normal of the curve, 
\[
\frac{\td \curve}{\td s} = \T \quad \text{and} \quad \frac{\td^2 \curve}{\td s^2} = k\N.
\]
The affine curvature arises from a different parameterization of the curve.  Define the  parameter, $\st$ by the condition that the vectors $C_\st$ and $C_{\st \st}$ form a parallelogram of area $1$, 
\[
[\curve_\st,\curve_{\st \st}] = 1.
\]
Here the brackets denote the determinant of the matrix whose columns are given by those vectors.  By differentiating the last equation, we obtain $[C_\st, C_{\st \st \st}] = 0$, which implies that $C_{\st \st \st} = \mu C_\st$ for some constant, $\mu$.  Using the defining condition again, we obtain
 \[
\mu = [\curve_{ \st \st },\curve_{\st \st \st}],
 \]
 which we define to be the affine curvature of the curve.  The affine curvature is the simplest nontrivial affine invariant of the curve \cite{su1983affine}.
 Ellipses have constant affine curvature.  
 
Under the affine curvature evolution, any convex curve remains convex; any convex smooth curve evolves to an ellipse until it collapses to a single point; any smooth curve becomes convex after a certain time.  Moreover, the affine curvature evolution is invariant under the class of special affine transformations, which are defined by matrices with determinant $1$ (see  Theorem~\ref{MoisanTheorem} below).   Compare this to the mean curvature evolution, which shrinks curves to circles \cite{Gage} and is invariant under the smaller class of orthogonal transformations.  

Affine differential geometry is not defined for non-convex curves.   However, we can define still define the affine curvature evolution using the Euclidean curvature, by taking the velocity to be $k^{1/3}\N$.  

\subsubsection*{Level Set PDE formulation}

The Level Set Method \cite{sethian1999level,OsherFedkiw} for the affine curvature evolution results in the PDE~\eqref{ACPDE}. The level set method has the following advantages compared to parameterized curve evolution: (i) it provides a natural generalization of the flows when the curve becomes singular and notions such as normals are not well defined; (ii) there is no need to track topological changes since they are discovered when the corresponding level set is computed; (iii) it can be discretized on a uniform grid, which is convenient for many applications.  

In the Level Set Method, a curve is represented implicitly as the level set of the auxiliary function $u(x,y,t):\R^2\times \R \to\R$, that is,
\[
\curve(t) = \left\{(x,y)\in \R^2 \mid u(x,y,t) = c\right\}
\]
for some arbitrary constant $c\in \R$. If $u$ satisfies $u_t = \Abs{\nabla u} \beta[u(\cdot,t)]$, for some function $\beta$ which depends on the level set of $u$, then all its level sets move in the normal direction with speed $\beta$. 
For example, choosing $\beta = 1$, we obtain the time-dependent eikonal equation, $u_t =  \Abs{\nabla u}$.
Taking $\beta =  k[u]$ or $\left(k[u]\right)^{1/3}$ we obtain 
 \eqref{MCevolution} and \eqref{ACPDE}, respectively.

\section{The affine curvature PDE}

\subsection{Definition of viscosity solutions}
Viscosity solutions \cite{CIL} provide the correct notion of weak solution to a class of degenerate elliptic PDEs, which includes \eqref{ACPDE} and \eqref{MCevolution}. The article \cite{EvansSpruckMCI} focuses on the mean curvature equation. The book \cite{giga2006surface} established existence and uniqueness of viscosity solutions in a bounded domain, with Neumann boundary conditions, see section 3.6 for the Comparison Principle, and Theorem 4.6.1 for the convergence of approximations.  It is also shown that the Lipschitz constant of the solution is preserved by the evolution (Section 3.5).     See also the book \cite{CaoBook} for viscosity solutions for geometric evolution equations, and numerical methods. 



\begin{definition}\label{defn:degenerateelliptic}
	The function $F:  \mathcal{S}^{n} \times \R^n \times \R \times \Omega : \to \R$ is proper and degenerate elliptic (in the sense of \cite{CIL}) if 
\[
F(M,p,r,x) \le F(N,p,s,x), \quad \text { for all } M \succeq N, r \le s, \text{ and all } x \in \Omega, p \in \R^n
\]	
where $Y \preceq  X$  if $d^\intercal Yd \le d^\intercal X d$ for all $d\in \Rn$.
\end{definition}

\begin{remark}
	By this convention, the Laplacian operator is elliptic when written  $F(M) = -\tr(M)$.  Some authors use the other convention, without the minus sign. 
\end{remark}

\begin{lemma}\label{lem:degell}
The operator $-\Aff[u]$ is degenerate elliptic.
\end{lemma}
\begin{proof}
We start with the observation that using \eqref{curvature}, we can write
\[
\Delta_1 u = u_{\T\T},
\qquad \T = \frac{(-u_y,u_x)}{(u_x^2+u_y^2)^{1/2}},
\]
where $\T$ is the (Euclidean) unit tangent. Then it is clear that $-\Delta_1$ is degenerate elliptic. 
Using the representation \eqref{AffPQ} the degenerate ellipticity of $-\Aff$ follows.
\end{proof}

%
%

We now give the definition of viscosity solutions on a domain $\Omega \subset \R^2$. We follow \cite{giga2006surface}. Let $T > 0$. We are interested in the Cauchy problem
\bq\label{cauchyproblem}
\begin{cases}
u_t = \Aff[u]						& \text{in } (x,t) \in \Omega \times (0,T),\\
B(x,p) := \nu(x)^\intercal p = 0		& \text{on } (x,t) \in \partial \Omega \times (0,T),\\
u(x,0) = u_0(x)						& \text{in } x \in \overline{\Omega}.
\end{cases}
\eq
where $\nu$ is the outward unit normal of $\partial \Omega$.

\begin{definition}
A function $u \in USC(\overline{\Omega} \times [0,T])$ is called a viscosity subsolution of \eqref{cauchyproblem} if $u(x,0) \leq u_0(x)$ for $x \in \overline{\Omega}$ and for any $\varphi \in C^2(\overline{\Omega}\times[0,T])$ such that $u-\varphi$ has a local maximum at $(x,t) \in \overline{\Omega} \times (0,T)$ then
\[\begin{cases}
\varphi_t(x,t) - \Aff[\varphi](x,t) \leq 0								& \text{if } x \in \Omega,\\
\min\{\varphi_t(x,t) - \Aff[\varphi](x,t), B(x,\nabla \phi(x,t))\}	 \leq 0	& \text{if } \in \partial \Omega.
\end{cases}\]
Similarly, $u \in LSC(\overline{\Omega} \times [0,T])$ is called a viscosity supersolution of \eqref{cauchyproblem} if $u(x,0) \geq u_0(x)$ for $x \in \overline{\Omega}$ and for any $\varphi \in C^2(\overline{\Omega}\times[0,T])$ such that $u-\varphi$ has a local minimum at $(x,t) \in \overline{\Omega} \times (0,T)$ then
\[\begin{cases}
\varphi_t(x,t) - \Aff[\varphi](x,t) \geq 0								& \text{if } x \in \Omega,\\
\max\{\varphi_t(x,t) - \Aff[\varphi](x,t), B(x,\nabla \phi(x,t))\} \geq 0	& \text{if } \in \partial \Omega.
\end{cases}\]
Finally, we call $u$ a viscosity solution of \eqref{cauchyproblem} if $u^*$ is a viscosity subsolution and $u_*$ is a viscosity supersolution.
\end{definition}

We have the following uniqueness result, from \cite{giga2006surface}.

\begin{theorem}[Comparison Principle]
Suppose that $\Omega$ is convex with $C^2$ boundary $\partial \Omega$. Let $u$ and $v$ be a viscosity subsolution and supersolution of \eqref{cauchyproblem}. Then $u \leq v$ in $\Omega \times (0,T)$.
\end{theorem}

\begin{remark}
We will also be interested in the static equation $\Aff[u] = f$. The definition of viscosity solution and comparison principle presented here for the parabolic equation generalize to the static equation.
\end{remark}

\subsection{Invariance properties of the PDE}

We now study some properties of \eqref{ACPDE}, starting with the following Lemma.
\begin{lemma}\label{lemma:Affproperties}
Let $u \in C^2(\R^2)$. Then the operator $\Aff$ has the following properties:
\begin{enumerate}[i)]
	\item Rescaling: for $h > 0$, define $v(x,y) := u(x/h,y/h)$
	\[\Aff[v] = h^{-4/3}\Aff[u];\]
	\item Morphology: for $g \in C^1(\R)$,
	\[\Aff[g \circ u] = g^\prime(u) \Aff[u];\]
	\item Affine Invariance: for any affine map $\phi(\x) = A\x+\textbf{b}$,
	\[\Aff[u \circ \phi] = (\det A)^{2/3} \Aff[u] \circ \phi.\]
\end{enumerate}
\end{lemma}

\begin{remark}
For $\Delta_1 u$, rescaling gives $\Delta_1 v = h^{-2} \Delta_1 u$ and invariance only holds for orthogonal transformations.
\end{remark}

\begin{proof}
i) Writing, for $\Abs{\nabla u} \not= 0$,  
$k[u] = \frac{1}{\Abs{\nabla u}}\left(\tr(D^2u) - \frac{\left.\nabla u\right.^\intercal D^2u \: \nabla u}{\Abs{\nabla u}^2} \right)$,
allows us to write
$
\left(\Aff[u]\right)^3 = \Abs{\nabla u}^2 \tr(D^2u) - \left.\nabla u\right.^\intercal D^2u \: \nabla u.
$
Then for $v(x,y) := u(x/h,y/h)$ we have $\nabla v = \frac{1}{h} \nabla u$ and $D^2v = \frac{1}{h^2}D^2u$, and property $i)$ follows.

Property $ii)$ follows from the fact that the PDE has the structure of a level set PDE, \cite{giga2006surface}, so the level sets are invariant under relabelling.


Finally, we prove $iii)$. Setting $v(x,y) = u(\phi(x,y))$, we have
\bq\label{lemma:Affproperties_aux1}
\nabla v = A \: \nabla u \quad \text{and} \quad D^2v = A^\intercal D^2u \:A.
\eq
Moreover, 
\[
\Aff[u] = \left(\Abs{\nabla u}^2 \tr(D^2u) - \left.\nabla u\right.^\intercal D^2u \: \nabla u\right)^{1/3} = \left(u_{xx} u_y^2 - 2u_x u_y u_{xy} + u_{yy} u_x^2\right)^{1/3}
\]
Now, using this formula, we compute $\Aff[v]$, after which the derivatives of $v$ are replaced by derivatives of $u$ using \eqref{lemma:Affproperties_aux1}. The proof then follows from an elementary but lengthy computation.
\end{proof}

\begin{theorem}\label{MoisanTheorem}
Consider $\Phi_t:u_0\mapsto u(\cdot,t)$ the solution map of \eqref{ACPDE}. Then $\Phi_t$ satisfies the following properties:
\begin{enumerate}[i)]
	\item \textit{Monotonicity:} $u \leq v \Rightarrow \Phi_t(u) \leq \Phi_t(v)$;
	\item\label{morphology} \textit{Morphology/Relabelling:} for any monotone scalar function $g$,
	\[\Phi_t(g \circ u) = g \circ \Phi_t(u);\]
	\item\label{affineinvariance} \textit{Affine Invariance:} for any affine map $\phi(\x) = A\x+\textbf{b}$,
	\[\Phi_t(u\circ \phi) = \left(\Phi_{t(\det A)^{2/3}}(u)\right) \circ \phi.\]
\end{enumerate}
\end{theorem}
\begin{remark}
Note that our rescaling factor property $iii)$ differs from the one in \cite{MoisanScheme}, however both formulas agree (and give unity) for special affine transformations which have determinant $1$.
\end{remark}
\begin{proof}
We establish the three properties one by one.
\begin{enumerate}[i)]
	\item \textit{Monotonicity:}
	It follows easily from the fact that \eqref{ACPDE} is an elliptic PDE and thus satisfies a comparison principle.
	\item \textit{Morphology:}
	Let $g$ be a monotone scalar function and $u$ be the solution of \eqref{ACPDE} with initial condition $u(\cdot,0) = u_0$. We want to show that $v := g \circ u$ is the solution of \eqref{ACPDE} with initial condition $v(\cdot,0) = g \circ u_0$. Formally, it is enough to observe that
\[
v_t = g^\prime(u) u_t, \qquad \Aff[v] = g^\prime(u) \Aff[u],
\]
where the second equality follows by Lemma \ref{lemma:Affproperties}.
	\item \textit{Affine Invariance:}
	Let $\phi$ be an affine transformation with $\phi(\x) = A\x+\textbf{b}$. Let $u$ be the solution of \eqref{ACPDE} with initial condition $u(\cdot,0) = u_0$. We want to show that
\[v(x,y,t) := u(\phi(x,y),t (\det A)^{2/3})\]
is the solution of \eqref{ACPDE} with initial condition $v(\cdot,0) = u_0 \circ \phi$. We have
\[v_t  = (\det A)^{2/3} u_t,\]
which together with Lemma \ref{lemma:Affproperties} $iii)$, is enough to conclude the proof.
\end{enumerate}
\end{proof}

\subsection{An exact solution}
We now give an example of an exact solution for \eqref{ACPDE}. See also  \cite[Section 1.7.4]{giga2006surface}. 

\begin{lemma}\label{lemma:ellipse}
Given $a,b > 0$, let $u:\R^2\times [0,\infty)$ be given by the shifted ellipse
\[u(x,y,t) = t + \frac{3}{4}(ab)^{2/3}\left(\left(\frac{x}{a}\right)^2+\left(\frac{y}{b}\right)^2\right)^{2/3} = t + \frac{3}{4}\left(\frac{b}{a}x^2+\frac{a}{b}y^2\right)^{2/3}.\]
Then $u$ is a classical solution of \eqref{ACPDE}. Moreover, we conclude that ellipses remain ellipses of fixed eccentricity under the motion by affine curvature.
\end{lemma}
\begin{proof}
Define
\[\phi(s) = \frac{3(ab)^{2/3}}{4}s^{2/3} \quad \text{and} \quad S(x,y) = \left(\frac{x}{a}\right)^2+\left(\frac{y}{b}\right)^2.\]
We can then rewrite $u$ as $u(x,y,t) = t + (\phi \circ S)(x,y)$.

The proof then follows by showing that $u$ is a solution of \eqref{ACPDE} if and only if
\bq\label{lemma:ellipse-aux}
1 = \phi^\prime(S) \Abs{\nabla S}\: \divergence\left(\frac{\nabla S}{\Abs{\nabla S}}\right)^{1/3}.
\eq

Indeed, we have that
\[\nabla S(x,y) = 2\left(\frac{x}{a^2},\frac{y}{b^2}\right), \quad \Abs{\nabla S} = 2\rho \quad \text{and} \quad \rho = \sqrt{\left(\frac{x}{a^2}\right)^2+\left(\frac{y}{b^2}\right)^2}.\]
Moreover,
\[
\divergence\left(\frac{\nabla S}{\Abs{\nabla S}}\right) =
\frac{1}{a^2\rho}-\frac{\left(x/a^2\right)^2}{a^2\rho^3}+\frac{1}{b^2\rho}-\frac{\left(y/b^2\right)^2}{b^2\rho^3}
= \frac{\left(y/b^2\right)^2}{a^2\rho^3}+\frac{\left(x/a^2\right)^2}{b^2\rho^3} = \frac{S(x,y)}{a^2b^2\rho^3}.
\]
Then \eqref{lemma:ellipse-aux} is equivalent to
\[1 = \phi^\prime(S)2\rho\left(\frac{S}{a^2b^2\rho^3}\right)^{1/3}.\]
To conclude the proof it is now enough to observe that $\phi$ is the solution of the equation above together with $\phi(0) = 0$.
\end{proof}

\section{Finite difference equations and convergence} 

\subsection{Elliptic finite difference schemes}\label{sec:convergencetheory}
The framework developed in \cite{BSnum} provides conditions for when approximation schemes converge to the unique viscosity solution of a PDE.  The article \cite{ObermanSINUM} defines elliptic finite difference schemes, which are monotone and stable.

\begin{theorem}\label{BStheorem}
Consider an elliptic PDE that satisfies a comparison principle for viscosity solutions. A consistent, stable and monotone approximation scheme converges locally uniformly to the (unique) viscosity solution.
\end{theorem}

Consider the domain $\Omega \subset D$ where $D$ is a $n$-cube.  We use a uniform grid of spacing, $\dx$,  in $D$, and define the finite difference grid, 
$\Grd = \{ x \in h \Zb^n  ~\mid ~  x\in D \}$. 
	
\begin{definition}[Elliptic finite difference equation]
Let $C(\Grd)$ denote the set of grid functions, $u: \Grd \to \R$. 
A finite difference operator is a map 
 $F^\dx: C(\Grd) \to C(\Grd)$, which has the following form,
\[
F^\dx[u](x) = F^\dx(x, u(x), u(x) - u(\cdot) ), 	
\]
where $u(\cdot)$ indicates the values of the grid function $u$.
It has \emph{stencil width} $W$ if 
 $F^\dx(x, u(x), u(x) - u(\cdot) )$ depends only on values $u(y)$ for $\norm{y-x}_\infty/\dx  \le W$.  A \emph{solution} of the finite difference scheme is a grid function which satisfies the equation $F^\dx[u](x) = 0$ for all $x \in \Grd$.
The finite difference operator is elliptic if 
\[
	r \le s, ~ v(\cdot) \le w(\cdot) \implies 
\Sk\left(x,r,v(\cdot)\right) \leq \Sk(x,s,w(\cdot)).
\]
\end{definition}

\begin{definition}[Consistent]\label{def:consistentSimpler}
The scheme~$\Sk$ is \emph{consistent} with the equation~$F$ if for any smooth function $\phi$ and $x\in{\Omega}$,
\[
\lim_{\dx \to 0,y\to x} \Sk[\phi](y) = F(D^2 \phi(x), \grad \phi(x),\phi(x),x).
\]
The scheme is accurate to order $k$ if 
\[
\lim_{y\to x} \Sk[\phi](y) = F(D^2 \phi(x), \grad \phi(x),\phi(x),x) + \bO(\dx^k).
\]
\end{definition}
\begin{remark}
	Consistency is usually verified by a Taylor series argument. 
\end{remark}

\begin{definition}[Discrete Comparison Principle]
Given the finite difference operator $\Sk: C(\Grd) \to C(\Grd)$, the \emph{comparison principle} holds for $\Sk$ if
\begin{equation}
\label{comparison}\tag{Comp}
\Sk[u](x)  \le \Sk[v](x) \text{ for all } x  \implies u(x) \le v(x)  \text{ for all } x.
\end{equation}
\end{definition}

\begin{remark}
In the Discrete Comparison Principle, the boundary conditions are encoded in $\Sk$, the assumption $\Sk[u] \leq \Sk[v]$ means $u \le v$ at Dirichlet boundary points.  Uniqueness of solutions follows from the Discrete Comparison Principle, since if $u,v$ are solutions, then  $\Sk(u) = \Sk(v) = 0$, so $u \le v$ and $u\ge v$, and thus  $u=v$.
\end{remark}

\begin{definition}[Lipschitz constant of the scheme] 
The finite difference operator $\Sk: C(\Grd) \to C(\Grd)$ is Lipschitz continuous with constant $K^\dx$ if $K^\dx$ is the smallest constant such that
\[
\Abs{\Sk\left(x,r,v(\cdot)\right) - \Sk\left(x,s,w(\cdot)\right)} \le K^\dx \max( \abs{r-s}, \norm{v - w}_\infty ), \text { for all } x \in \Grd
\]	
\end{definition}

\begin{remark}
	In the definition above, the maximum on the right hand side can be replaced with a maximum over the neighbouring grid values without changing the Lipschitz constant.  
\end{remark}

In \cite{ObermanSINUM}, it is shown that the Euler map $u \to u - \rho F[u]$  with constant $\rho \le 1/K^\dx$ is monotone, and a (non-strict) contraction.  By adding an arbitrarily small multiple of $u$ to the scheme, the comparison principle holds, and the Euler map is a strict contraction. The resulting scheme provides a convergent discretization of the solution of the parabolic PDE $u_t + F[u]=0$, since the discretization is monotone and stable. 

In this context, monotone means that the discrete comparison principle holds: if  $u(x,n), v(x,n)$ are solutions of the scheme, $u(x,n) \leq v(x,n)$ for all $x$ implies $u(x,n+1) \leq v(x,n+1)$ for all $x$.

Filtered schemes were first proposed in \cite{FroeseObermanFiltered} in the context of the Monge-Amp\`{e}re equation to overcome the limited accuracy of the wide-stencil elliptic schemes. The idea is to blend a stable elliptic convergent scheme with an accurate scheme and retain the advantages of both: stability and convergence of the former, and higher accuracy of the latter. The only requirement on the accurate scheme is consistency. Filtered schemes have also been used for  Hamilton-Jacobi equations in \cite{FilteredHJ}.  We require that the difference between the filtered scheme and the elliptic scheme is uniformly bounded.  Under this assumption, the proof of the Barles-Souganidis theorem can be modified to obtain convergence results for filtered schemes \cite{FroeseObermanFiltered}.

\subsection{Example finite difference schemes}

Define the standard finite difference operators for $u_x, u_y$ and $u_{xy}$.
\begin{definition}[Standard finite differences for $u_x$, $u_y$,  and $u_{xy}$]
\begin{equation}
\label{standard_fd}
\begin{aligned}
& u_x^h := \frac{u(x+h,y)-u(x-h,y)}{2h} = u_x(x,y) + \bO(h^2),\\
& u_{xx}^h := \frac{u(x+h,y)-2u(x,y)+u(x-h,y)}{h^2} = u_{xx}(x,y) + \bO(h^2),\\
& u_{xy}^h := \frac{u(x+h,y+h)+u(x-h,y-h)-u(x+h,y-h)-u(x-h,y+h)}{4h^2} = u_{xy}(x,y) + \bO(h^2),
\end{aligned}
\end{equation}
and, similarly for $u_y^h$, and $u_{yy}^h$ in the $y$ coordinate. 
\end{definition}

The operator $-u_{xx}^h$ is elliptic, the others are not. 

\begin{definition}
Define the  backward and forward first order derivatives
\begin{align*}
D^-_x[u](x,y) &:= \frac{u(x,y)-u(x-h,y)}{h} = u_x(x,y) + \bO(h),
\\
-D^+_x[u](x,y) &:= \frac{u(x,y)-u(x+h,y)}{h} = -u_x(x,y) + \bO(h),
\end{align*}
and, similarly for $D^-_y[u]$ and $D^-_y[u]$.
\end{definition}
The operators $D^-_x$ and $-D^+_x$ are elliptic.

\subsection{Non-decreasing functions and elliptic discretizations}\label{sec:schemeEikonal}

In order to build elliptic difference schemes, we study the properties of non-decreasing maps. This is required since in general elliptic schemes are built composing non-decreasing maps with elliptic terms. Moreover, for some nonlinear elliptic PDEs the domain of ellipticity is restricted and thus we need to build non-decreasing extensions of the underlying functions.

In this section we define elliptic schemes for $\Abs{\nabla u}$ and $-\Abs{\nabla u}$. This is done both as an illustration of the general principle, and because these schemes will be needed for our discretization of the two dimensional affine curvature operator. 

\begin{definition}[Non-decreasing functions]
For $x,y \in \R^N$ we say $x \leq y$ if $x_i \leq y_i$ for all $i =1,\dots, N$. The function $F:\R^N \to \R$ is non-decreasing, $F \in ND(\R^N)$, if 
\[
x \leq y \implies F(x) \leq F(y).
\]
Write $\R^+ = \{ x \in \R \mid x \geq 0 \}$ and $\R^- = \{ x \in \R \mid x \leq 0 \}$.
Furthermore, if $F \in ND(\R^N)$ and $F: \R^N \to \R^+$, (resp. $F: \R^N \to \R^-)$ we write $F \in ND^+(\R^N)$ (resp. $F \in ND^-(\R^N)$). 
\end{definition}

\begin{remark}
When $f \in C^1(\R^N)$, if $f$ is nondecreasing in each variable,  i.e., $f_{x_i} \geq 0$ for all $i =1,\dots, N$, then $f \in ND(\R^N)$. 
\end{remark}

\begin{remark}
The set of nondecreasing functions is closed under the composition of functions. In particular, sums of nondecreasing functions are nondecreasing. 
\end{remark}

\begin{example}
The function $x^+ = \max(x,0) \in ND^+(\R)$ and  $x^- = \min(x,0)$ is in $ND^-(\R)$.
Write $N(x,y) = \sqrt{x^2+y^2}$, and define $N^+(x,y) := N(x^+,y^+)$ and $N^-(x,y) := -N(x^-,y^-)$. Then $N^+ \in ND^+(\R^2)$, and  $N^- \in ND^-(\R^2)$. Furthermore, $N^+ = N$ on  $\{ x, y \geq 0 \}$, $N^- = -N$ on  $\{ x, y \leq 0 \}$.
\end{example}

\begin{example}[Upwinding]
	More generally, if we consider the operator $a(x)u_x$, then the upwind discretization 
	\[
	a^+ D^-_x[u] + a^-D^+_x[u]
	\]
is first order accurate  and elliptic. 
\end{example}

The first example of using elliptic discretizations as building blocks is given by the following.

\begin{example} \label{defn:elliptic_1}\label{lemma:basic_elliptic}
Define 
\[
\Abs{u_x^h}^+ = \max\left\{-D^+_xu,D^-_xu,0\right\}, 
\quad 
-\Abs{u_x^h}^- = \min\left\{-D^+_xu,D^-_xu,0\right\}
\]
that approximate $\abs{u_x}$ and $-\abs{u_x}$to first order.  The operators $\Abs{u_x^h}^+$ and $-\Abs{u_x^h}^-$ are elliptic, the former is nonnegative, and the latter is nonpositive. 
\end{example}

\begin{proof}
Since $D^-_x$ and $-D^+_x$ are elliptic and $\max,\min \in ND(\R^2)$, the composed functions $\Abs{u_x^h}^+ $ and $-\Abs{u_x^h}^-$ are elliptic.
\end{proof}

%
%

\begin{example}

Define 
\[
\abs{\nabla u^h}^+ = N \left(|u_x^h|^+,|u_y^h|^+\right),
\qquad
-\abs{\nabla u^h}^- = -N\left(-|u_x^h|^-,-|u_y^h|^-\right)
\]
which are elliptic,  consistent with $\Abs{\nabla u}$, and $-\Abs{\nabla u}$, respectively,  and first order accurate.  
\end{example}

\begin{proof}
Since $|u_x^h|^+$ and $|u_y^h|^+$ are nonnegative, 
\[
\abs{\nabla u^h}^+ = N \left(|u_x^h|^+,|u_y^h|^+\right) = N^+ \left(|u_x^h|^+,|u_y^h|^+\right).
\]
Thus, since $N^+ \in ND(\R^2)$ and $|u_x^h|^+$ and $|u_y^h|^+$ are elliptic, the composed function $|\nabla u^h|^+$ is elliptic.

Similarly, $-|u_x^h|^-$ and $-|u_y^h|^-$ are elliptic and nonpositive and 
\[
-|\nabla u^h|^- = N^-\left(-|u_x^h|^-,-|u_y^h|^-\right)
\]
with $N^- \in ND(\R^2)$. Thus $-|\nabla u^h|^-$ is elliptic. 

Consistency and first order accuracy follow from the generalized chain rule: the discretization of each term is consistent and first order accurate, and $N(\cdot)$ is a 1-Lipschitz function.
\end{proof} 

\section{Numerical methods for the model equation}\label{sec:model1D}

In this section we build convergent discretizations to the one dimensional model of our equation defined by
\bq\label{1Dmodel}\tag{AC-1D}
\Aff^{1D}[u] := A(u_x,u_{xx}) = \left(u_x^2 \:u_{xx}\right)^{1/3} = f ,\quad x \in (-1,1),
\eq
as well as the parabolic PDE,
\[
u_t = \Aff^{1D}[u]-f,
\quad \text{ for } (x,t) \in (-1,1) \times (0,\infty).
\]
along with initial and boundary conditions. In order to do so we need to build elliptic discretizations for $\Aff^{1D}$, which is achieved by studying the nonlinear function $A(p,q)$. A naive approach would suggest to simply substitute the elliptic discretizations for $\Abs{u_x}$ and $u_{xx}$ into $A(p,q)$. However, this is not possible since $A(p,q) \not\in ND(\R^2)$: $dA/dp  = 2/3(q/p)^{1/3}$ and so $A$ fails to be non-decreasing when $pq < 0$. 

Our approach is the following. First, in \autoref{subsec:elliptic1D}, we write $A(p,q)$ as a sum of two nondecreasing functions in terms of $\Abs{p}$, $-\Abs{p}$, $q$. These terms are replaced by $\abs{u^h_x}$, $-\abs{u^h_x}$, $u_{xx}^h$ which are elliptic and consistent and thus a consistent and elliptic discretization for $\Aff^{1D}$ is build. Then, in \autoref{subsec:Lipschitz1D}, we present a Lipschitz regularization of the function $A(p,q)$ and proceed similarly. The Lipschitz regularization is needed to build a provably convergent explicit scheme for the parabolic PDE as discussed in \autoref{sec:convergencetheory}. Finally, in \autoref{subsec:convergence1D}, we present the convergence results.

\subsection{An elliptic discretization of the one dimensional operator}\label{subsec:elliptic1D}


In the next lemma, we decompose $A(p,q)$ into the sum of two nondecreasing functions.
\begin{lemma}\label{lemma:aux_elliptic}
Define $A^+(p,q) = A(p^+,q^+)$ and $A^-(p,q) = A(p^-,q^-)$. Then $A^+ \in ND^+(\R^2)$, $A^- \in ND^-(\R^2)$ and 
\[
-A(p,q) = A(p,-q)  = A^+(\abs{p},-q) + A^-(-\abs{p},-q),
\quad 
\text{ for all } p,q
\]
\end{lemma}

\begin{proof}
The fact that $A$ can be decomposed follows from checking two cases, depending on the sign of $q$.

Next we establish that $A^+ \in ND^+(\R^2)$ and  $A^- \in ND^-(\R^2)$.  First it is clear that $A^+ \geq 0$ and that $A^- \leq 0$. Compute the partial derivatives, 
\[
A^+_p = \frac{2}{3}\left(\frac{q^+}{p^+}\right)^{1/3} \geq 0,
\qquad 
A^+_q = \frac{1}{3}\left(\frac{p^+}{q^+}\right)^{2/3} \geq 0,
\]
So $A^+$ is nondecreasing in each variable, which is enough to show $A^+ \in ND$.  Similarly, consider $A^-(p,q)$. Again taking partial derivatives, we see that $A^-$ is nondecreasing in each variable, so $A^- \in ND$.
\end{proof}

\begin{lemma}\label{lemma:consistencyellipticAC1D}
Define the finite difference operator 
\begin{align}\tag*{$(AC)^{1D,e}$}\label{ellipticAff1D}
-\Aff^{1D,e}[u] 
= 
A^+\left( \Abs{u_x^h}^+,  -u_{xx}^h  \right) + A^-\left( -\Abs{u_x^h}^-,-u_{xx}^h  \right )
\end{align}
where the finite difference operators involved are defined above. Then  $-\Aff^{1D,e}$ is elliptic and consistent with the $-\Aff^{1D}$.
\end{lemma}

\begin{proof}
The operators $|u_x^h|^+$, $-|u_x^h|^-$, $-u_{xx}^h$ are elliptic. 
Then, due to Lemma \ref{lemma:aux_elliptic}, the operator $-\Aff^{1D,e}[u]$ consists of the composition of the nondecreasing functions $A^+$ and $A^-$ with the three preceding operators, with the first two taking the place of $\abs{p}$ and $-\abs{p}$ and the last one taking the place of $-q$. Since the operators are elliptic and the functions are nondecreasing, $-\Aff^{1D,e}[u]$ is elliptic.
Consistency follows from the consistency of each of the schemes used. 
\end{proof}

\begin{remark}[Accuracy]\label{rmk:accuracy1D}
One challenge in forming a discretization of \eqref{1Dmodel} is that the accuracy breaks down near $u_{xx} = 0$.  
For example, if we use second order accurate approximations of $u_{xx}$, the accuracy of finite differences for the operator is only $\bO(h^{2/3})$. 
Consider the expression
\[
A(p+h^l,q+h^2) = ((p + h^k)^2(q+ h^2))^{1/3}
\]
where $k = 1$ or $2$.  
If $p\not=0, q=0$ we get
\[
A(p+h^k,q+h^2) = h^{2/3} (p + h^l)^{2/3} = \bO(h^{2/3}).
\]
So regardless of the accuracy of the discretization of the $u_x$ term, the overall accuracy of the scheme cannot be better than $\bO(h^{2/3})$.  In fact, a similar argument shows that the order of accuracy is $2k/3$ near $p=0, q\not=0$.  Our elliptic discretization, which uses $k=1$, will have order of accuracy $2/3$.  
\end{remark}

\subsection{Lipschitz regularization of the one dimensional operator}\label{subsec:Lipschitz1D}

The function $A(p,q)$ fails to be Lipschitz continuous near the axis $p=0$, $q=0$. Hence the elliptic scheme presented in the previous section is not Lipschitz, a property that is required in order to build a monotone convergent scheme for the time dependent problem. Thus, using the notation for the sign function  $\sgn(q) = q/\abs{q}$ for $q\not=0$ and $\sgn(0) = 0$ otherwise, we regularize $A(p,q)$ as follows.

\begin{definition}Define for $K= K(\delta), L = L(\delta) > 0$, the regularized function 
\bq\label{Adelta}
A^\delta(p,q) =	\sgn(q)\min(\abs{A(p,q)},K\abs{p},L\abs{q}).
\eq
\end{definition}

Naturally, we can then define the regularized PDE operator as
\[
\Aff^{1D,\delta}[u] = A^\delta(u_x,u_{xx}).
\]

Defining an elliptic and consistent scheme for $\Aff^{1D,\delta}[u]$ is accomplished by noticing that the discretization \ref{ellipticAff1D} generalizes when we replace $A$ with the regularized version $A^\delta$ in each term: just like for $A(p,q)$ in Lemma \ref{lemma:aux_elliptic}, we can decompose $A^\delta(p,q)$ into the sum of two nondecreasing functions. This is achieved in the following lemma.

\begin{lemma}
Define $A^{\delta,+}(p,q) = A^\delta(p^+,q^+)$ and $A^{\delta,-}(p,q) = A^\delta(p^-,q^-)$. Then
\[-A^\delta(p,q)  = A^{\delta,+}(\abs{p},-q)+A^{\delta,-}(-\abs{p},-q),\]
where $A^{\delta,+} \in ND^+(\R^2)$ and $A^{\delta,-} \in ND^-(\R^2)$.
\end{lemma}
\begin{proof}
To prove the decomposition of $A^\delta$ we have to consider two cases: $q \geq 0$ and $q < 0$.

Suppose first that $q \geq 0$. Then $A^{\delta,+}(\abs{p},-q) = 0$ since $(-q)^+ = 0$. Therefore
\begin{align*}
A^{\delta,+}(\abs{p},-q) + A^{\delta,-}(-\abs{p},-q)		& = A^\delta(-\abs{p},-q)\\
	& = -\sgn(q)\min\left(\abs{A(p,q)},K\abs{p},L\abs{q}\right)\\
	& = -A^\delta(p,q).
\end{align*}

Assume now that $q < 0$. Then $A^{\delta,-}(-\abs{p},-q) = 0$ since $(-q)^- = 0$. Hence
\begin{align*}
A^{\delta,+}(\abs{p},-q) + A^{\delta,-}(-\abs{p},-q)		& = A^\delta(\abs{p},-q)\\
	& = \sgn(-q)\min\left(\abs{A(p,q)},K\abs{p},L\abs{q}\right)\\
	& = -A^\delta(p,q).
\end{align*}

To show that $A^{\delta,+}$ is nondecreasing, we start by rewriting it as
\[
A^{\delta,+}(p,q) = \min\left(A^+(p,q),Kp^+,Lq^+\right).
\]
The result then follows by noticing that $\min$, $A^+$, $Kp^+$, $Lq^+$ are all nondecreasing. Similarly, $A^{\delta,-}$ is also nondecreasing, which follows from rewriting it as
\[
A^{\delta,-}(p,q) = \max\left(A^-(p,q),Kp^-,Lq^-\right).
\]
and noticing that $\max$, $A^-$, $Kp^-$, $Lq^-$ are all nondecreasing.
\end{proof}

Ellipticity and consistency of the regularized scheme with respect to $-\Aff^{1D,e,\delta}[u]$ follows now easily, just like in Lemma \ref{lemma:consistencyellipticAC1D}. For this reason we omit the proof.

\begin{lemma}\label{lemma:1Delliptic}
For $K = K(\delta)$, $L=L(\delta) > 0$, define the finite difference scheme
\bq\label{elliptic_scheme_regular1D}\tag*{$(AC)^{1D,e,\delta}$}
-\Aff^{1D,e,\delta}[u] = A^{\delta,+}\left( \Abs{u_x^h}^+, -u_{xx}^h \right) + A^{\delta,-}\left( -\Abs{u_x^h}^-, -u_{xx}^h \right ).
\eq
Then $-\Aff^{1D,e,\delta}[u]$ is elliptic and consistent with $-\Aff^{1D,\delta}[u]$.
\end{lemma}

Proving consistency of the regularized scheme with respect to $-\Aff^{1D}[u]$ requires extra work as the parameters $K$, $L$ need to be chosen carefully. The next theorem summarizes the results proven in the lemmas that follow.

\begin{theorem}
Asssume $K = h^{-1/3}$ and $L = h^{-4/3}$. Let $x \in \Omega$ be a reference point on the grid and $\phi$ be a smooth                  function that is defined in a neighborhood of the grid. Then the scheme $\Aff^{1D,e,\delta}$ defined by \ref{elliptic_scheme_regular1D} is consistent with $\Aff^{1D}$ and has accuracy
\[
\Aff^{1D,e,\delta}[\phi](x) - \Aff^{1D}[\phi](x) = \bO(h^{2/3}).
\]
Moreover, $\Aff^{1D,e,\delta}$ is Lipschitz continuous with constant $C^h$ given by
\bq\label{Lipconst1D}
C^h = h^{-4/3} + 2h^{-10/3}.
\eq
\end{theorem}

We start by showing that our regularization of $A$ is indeed Lipschitz continuous

\begin{lemma}\label{lemma:aux_regularization}
Suppose  $K\sqrt{L} \geq 1$. Then 
\[
\Abs{\frac{d}{dp} A^\delta(p,q)}  \leq K
\qquad
\Abs{\frac{d}{dq} A^\delta(p,q)}  \leq L
\]
where the derivatives exist, and $A^\delta(p,q)$ is Lipschitz continuous. Moreover,
\[
\Abs{A^\delta(p,q) - A(p,q)} \leq \max\left(\frac{4}{27K^2} \abs{q},\frac{2}{3\sqrt{3L}}\abs{p}\right), 
\quad 
\text{ for all } p,q.
\]

\end{lemma}

\begin{proof}
First we determine the sets where each term in the minimum is active. We claim that
\[
A^\delta(p,q) = 
\begin{cases}
Lq,				&  \abs{q} \leq L^{-3/2} \abs{p},\\
\sgn(q)K\abs{p},	&  \abs{q} \geq K^3\abs{p},\\
(p^2q)^{1/3}		& \text{ otherwise.}
\end{cases}
\]
Indeed, since $K\sqrt{L} \geq 1$, the claim follows from 
\[
\abs{q} \leq L^{-3/2} \abs{p} \Rightarrow K\abs{p} \geq L \abs{q}
\quad \text{and} \quad
\abs{q} \geq K^3\abs{p} \Rightarrow L\abs{q} \geq K \abs{p}.
\]
and
\[
L\abs{q} \leq \abs{A(p,q)} \Leftrightarrow \abs{q} \leq L^{-3/2} \abs{p}
\quad \text{and} \quad
K\abs{p} \leq \abs{A(p,q)} \Leftrightarrow K^3 \abs{p} \leq \abs{q}.
\]

We continue the proof by proving the derivate estimates. Computing the derivative with respect to $p$ gives
\[
\frac{d}{dp} A^\delta(p,q) =
\begin{cases}
0, &  \abs{q} \leq L^{-3/2} \abs{p}\\
\sgn(q)K\sgn(p), &  \abs{q} \geq K^3\abs{p}  \\
\frac 2 3 (qp^{-1})^{1/3} & \text{ otherwise } 
\end{cases}
\]
In the third case, since $\abs{q/p} \leq K^3$, the partial derivative is bounded by $\frac{2}{3} K$, and so we can conclude that $\abs{dA^\delta/dp} \leq K$.

Similarly, computing 
\[
\frac{d}{dq} A^\delta(p,q) =
\begin{cases}
	L, &  q \leq L^{-3/2} p \\
	0, &  q \geq K^3p  \\
	\frac 1 3 (pq^{-1})^{2/3} & \text{ otherwise } 
\end{cases}
\]
In the third case, since $\abs{p/q} \leq L^{3/2}$ the value is bounded by $\frac 1 3 L$, and the global bound holds.

Finally, we prove the estimate on the error introduced by the regularization. Since both $A$ and $A^\delta$ are even functions with respect to $p$ and odd with respect to $q$, we can assume without loss of generality that $p$ and $q$ are both positive. We have two cases to consider: $\abs{q} \geq K^3 \abs{p}$ and $\abs{q} \leq L^{-3/2} \abs{p}$ which is where $A$ and $A^\delta$ differ.

If $\abs{q} \geq K^3 \abs{p}$ then
\[
\Abs{A^\delta(p,q) - A(p,q)} \leq \frac{4}{27K^2} \abs{q}.
\]

In this case, $A^\delta(p,q) = Kp$ and so
\[
\frac{d}{dp}\left(A(p,q)-A^\delta(p,q)\right) = K-\frac{2}{3}\left(\frac{q}{p}\right)^{1/3} = 0 \Leftrightarrow \frac{8}{27} q = K^3 p \Leftrightarrow p = \frac{8}{27K^3} q
\]
where $\alpha = 2/3$. Hence
\[
\max_p \Abs{A^\delta(p,q) - A(p,q)} \leq \Abs{A^\delta\left(\frac{8}{27K^3} q,q\right)-A\left(\frac{8}{27K^3}q,q\right)} = \frac{4}{27K^2}\abs{q}.
\]

If $\abs{q} \leq L^{-3/2} \abs{p}$ then
\[
\Abs{A^\delta(p,q) - A(p,q)} \leq \frac{2}{3\sqrt{3L}}\abs{p}.
\]
	
In this case, $A^\delta(p,q) = Lq$ and so
\[
\frac{d}{dq}\left(A(p,q)-A^\delta(p,q)\right) = L-\frac{1}{3}\left(\frac{p}{q}\right)^{2/3} = 0 \Leftrightarrow q = (3L)^{-3/2} p.
\]
Hence
\[
\max_q \Abs{A^\delta(p,q) - A(p,q)} \leq \Abs{A^\delta\left(p,(3L)^{-3/2} p\right)-A\left(p,(3L)^{-3/2} p\right)} = \frac{2}{3\sqrt{3L}}\abs{p}.
\]

The result now follows straightforwardly.
\end{proof}

Now, we show that $\Aff^{1D,e,\delta}[u]$ is Lipschitz continuous.

\begin{lemma}\label{lemma:Lipconst1D}
$\Aff^{1D,e,\delta}[u]$ is Lipschitz continuous with constant
\[C^h = \frac{K}{h}+\frac{2L}{h^2}.\]
\end{lemma}
\begin{proof}
Using the derivative estimates proved in Lemma \ref{lemma:aux_regularization} and the triangle inequality, we get
\begin{align*}
\Abs{A^\delta(p_1,q_1)-A^\delta(p_2,q_2)}	& \leq \Abs{A^\delta(p_1,q_1)-A^\delta(p_2,q_1)}+\Abs{A^\delta(p_2,q_1)-A^\delta(p_2,q_2)}\\
									& \leq \Abs{\frac{d}{dp}A^\delta(\tilde{p},q_1)}\abs{p_1-p_2} + \Abs{\frac{d}{dq}A^\delta(p_2,\tilde{q})}\abs{q_1-q_2}\\
									& \leq K\abs{p_1-p_2} + L\abs{q_1-q_2}.
\end{align*}
for some $\tilde{p},\tilde{q} \in \R$. Here, the schemes $\abs{u^h_x}^\pm$ take the place of $p_1,p_2$ and have Lipschitz constant of $1/h$. On the other hand, $q_1,q_2$ are replaced by $(-u_{xx}^h)^\pm$ which has Lipschitz constant $2/h^2$. The result follows easily then.
\end{proof}

Finally, we study how $K$ and $L$ can be chosen to ensure that $\Aff^{1D,e,\delta}[u]$ is consistent with $\Aff^{1D}[u]$.

\begin{lemma}\label{lemma:accuracyregular1D}Let $x \in \Omega$ be a reference point on the grid and $\phi$ be a twice differentiable function that is defined in a neighborhood of the grid. Assume $K = \bO(h^{-\alpha})$ and $L = \bO(h^{-\beta})$ such that $K\sqrt{L} \geq 1$, $\alpha \in (0,1)$ and $\beta \in (0,2)$. Then the scheme $\Aff^{1D,e,\delta}$ defined by \ref{elliptic_scheme_regular1D} is consistent with $\Aff^{1D}$ and has accuracy
\[
\Aff^{1D,e,\delta}[\phi](x) - \Aff^{1D}[\phi](x) = \bO\left(h^{1-\alpha} + h^{2-\beta}+h^{\min(2\alpha,\beta/2)}\right).
\]
Moreover, the optimal choice of $\alpha$ and $\beta$ is given by 
$\alpha = 1/3$ and $\beta = 4/3$, in which case the accuracy is $\bO(h^{2/3})$.
\end{lemma}

\begin{proof}
We showed in the previous lemma that
\[\Abs{A^\delta(p_1,q_1)-A^\delta(p_2,q_2)} \leq K\abs{p_1-p_2} + L\abs{q_1-q_2}.\]
Here, $\Abs{u_x^h}^\pm$ take the place of $p_1$, while $p_2$ is replaced by $\Abs{u_x}$. On the other hand, $q_1$ is replaced by $-u_{xx}^h$, while $q_2$ is replaced by $-u_{xx}$. The result follows from the consistent of the finite difference operators
\[\Abs{u_x^h}^\pm = \Abs{u_x} + \bO(h), \quad \left(-u_{xx}^h \right)^\pm = (-u_{xx})^\pm + \bO(h^2).\]
Hence
\[
\Aff^{1D,e,\delta}[\phi](x) - \Aff^{1D,\delta}[\phi](x) = \bO\left(h^{1-\alpha} + h^{2-\beta}\right).
\]

A direct application of Lemma \ref{lemma:aux_regularization}, where $u_x$ and $u_{xx}$ take the place of $p$ and $q$ respectively, leads to the estimate
\[\Abs{\Aff^{1D,\delta}[\phi] - \Aff^{1D}[\phi]} \leq \max\left(\frac{4}{27K^2} \abs{u_{xx}},\frac{2}{3\sqrt{3L}}\abs{u_x}\right).\]
Therefore, since $K = \bO(h^{-\alpha})$ and $L = \bO(h^{-\beta})$
\[
\Aff^{1D,e}[\phi](x) - \Aff^{1D}[\phi](x) = \bO\left(h^{\min(2\alpha,\beta/2)}\right).
\]
The accuracy of $\Aff^{1D,e,\delta}$ then follows from the equality
\[
\Aff^{1D,e,\delta}[\phi](x) - \Aff^{1D}[\phi](x) = \Aff^{1D,e,\delta}[\phi](x) - \Aff^{1D,\delta}[\phi] + \Aff^{1D,e}[\phi](x) - \Aff^{1D}[\phi](x)
\]

Finally, we observe that
\[\max(\min(1-\alpha,2\alpha)) = \max(\min(2-\beta,\beta/2)) = \frac{2}{3},\]
with the maximums being attained at $\alpha = 1/3$ and $\beta = 4/3$, thus justifying the optimal choice of $\alpha$ and $\beta$.
\end{proof}
\begin{remark}
As a result of the proof, we show as well that $\Aff^{1D,e,\delta}$ is consistent with $\Aff^{1D,\delta}$ with accuracy
\[
\Aff^{1D,e,\delta}[\phi](x) - \Aff^{1D,\delta}[\phi](x) = \bO\left(h^{1-\alpha} + h^{2-\beta}\right).
\]	
\end{remark}

\begin{remark}
With the optimal choice of $\alpha$ and $\beta$, the overall accuracy of $\Aff^{1D,e,\delta}$ and $\Aff^{1D,e}$ with respect to $\Aff^{1D}$ is the same (see Remark \ref{rmk:accuracy1D}), meaning that no accuracy is lost due to the regularization.
\end{remark}

\subsection{Convergence theorems for the one-dimensional model}\label{subsec:convergence1D}

Having proved the ellipticity and consistency of the schemes, the uniform convergence follows as discussed in \autoref{sec:convergencetheory}.

The first convergence result is for the elliptic problem, where there is no need for the regularized scheme.
\begin{theorem} Let $u(x)$ be the unique viscosity solution of $\Aff^{1D}[u] = f$ in $\Omega$, along with suitable boundary conditions. For each $h>0$,  let $u^{1D,e,h}$ be the uniformly bounded solution of $\Aff^{1D,e}[u] = f$. Then $u^{1D,e,h} \to u$ locally uniformly, as $h \to 0$.
\end{theorem}
\begin{proof}
Convergence for the elliptic discretization $-\Aff^{1D,e}$ follows from the Barles-Souganidis theorem.
\end{proof}

Unlike in the elliptic problem, in the parabolic problem we need to use the regularized scheme, with the time discretization being given by a forward Euler step. This leads us to
\bq\label{discretizationtime1D}
u(\cdot,t+dt) = u(\cdot,t) + dt \:\Aff^{1D,e,\delta}[u(\cdot,t)].
\eq

\begin{theorem}\label{theorem:convergence_regular_1D}
Let  $u(x,t)$ be the unique viscosity solution of  $u_t = {\Aff}^{1D}[u]$  in $\Omega \times [0,\infty)$, along with $u(x,0) = u_0(x)$ and suitable boundary conditions. Assume as well that $K = h^{-1/3}$ and $L = h^{-4/3}$. For each $h >0$,  let $u^{1D,e,h}$ be the uniformly bounded solution of the monotone  time discretization \eqref{discretizationtime1D} with $dt \leq 1/{C^h}$ given by \eqref{Lipconst1D}. Then $u^{1D,e,h} \to u$ locally uniformly, as $dt,h  \to 0$.
\end{theorem}
\begin{proof}
The elliptic scheme $\Aff^{1D,e,\delta}$ leads to a monotone time discretization, in other words, the solution map satisfies a comparison principle if $dt \le 1/{C^h}$ where $C^h$ is the Lipschitz constant of the elliptic scheme by \cite{ObermanSINUM}. Then the Barles-Souganidis theorem \cite{BSnum} applies.
\end{proof}

\begin{remark}
It is also true that, for fixed values of $K,L$, there is a unique viscosity solution, $u^\delta$, of the regularized PDE.  Then, fixing $K,L$ and using the discretization above with $dt = \bO(h^2)$, the forward Euler method converges uniformly to $u^\delta$ as $h\to 0$. 
\end{remark}

\section{Nonconvergence of standard finite differences}\label{sec:standard}\label{sec:breakaccurate}

In this section we show that standard finite differences are unstable for both the one dimensional and the two dimensional operators that we study. 
This instability can be understood from the one dimensional model, which, if we take $|u_x| = 1$, reduces to the operator $u_{xx}^{1/3}$.  It is certainly plausible that the singularity near $u_{xx} =0$ could lead to instabilities. 
This motivates the regularization introduced in the previous section, which replaces the singularity of the cube root with a linear function with large slope.  It also motives the convergent finite difference schemes, which have an explicit time step that ensures the convergence of the time dependent schemes.

We begin with an example where the standard finite scheme does not converge for the one-dimensional model. Next we consider the time dependent equation and the associated scheme obtained with the (unregularized) elliptic scheme and a forward Euler step in time. A scaling argument suggest that the choice $dt = \bO(h^{4/3})$ should provide a stable scheme. In fact, this scaling argument can be improved to a proof of the maximum principle for the homogeneous equation with the same time step. However, the maximum principle is not enough for convergence (we need the comparison principle) and we demonstrate divergence with that time step. Using a smaller time step $dt = \bO(h^2)$ appears to fix the problem.  (The standard finite difference scheme diverges for the example we present no matter how small the time step).  The example is then generalized to the two-dimensional operator.

\subsection{Nonconvergence of standard finite differences in one dimension}

Consider the discretization given by inserting the standard centered differences, given by \eqref{standard_fd},
\[
\Aff^{1D,a}[u] = \left(\left(u_x^h\right)^2\left(u_{xx}^h\right)\right)^{1/3}.
\]

We considered an exact solution  $u_0(x) = \sin(2\pi x)$ of \eqref{1Dmodel}.  Then we solved the time dependent problem, using the forward Euler time discretization,  with initial data given by the solution $u(x,0) = u_0(x)$   We found that this solution was unstable for the standard finite difference scheme.
The results are displayed in Figure \ref{fig:1Dbreak}, which illustrates that the numerical solution diverges from the exact solution.  This effect persists over different grid sizes, and over smaller time steps. 
\begin{figure}[htp]
\centering
\begin{tabular}{cccc}
\includegraphics[width=0.2\textwidth]{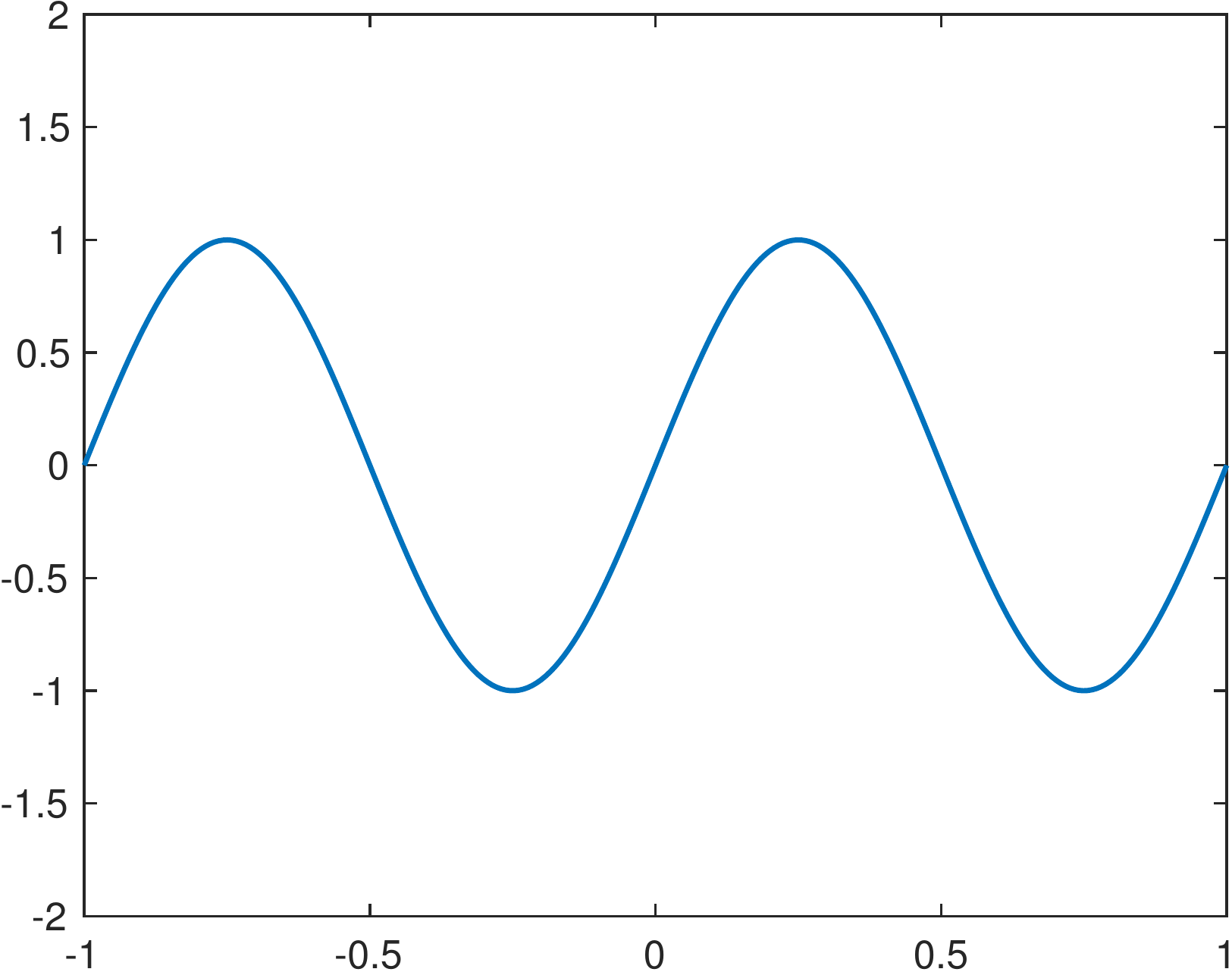} & \includegraphics[width=0.2\textwidth]{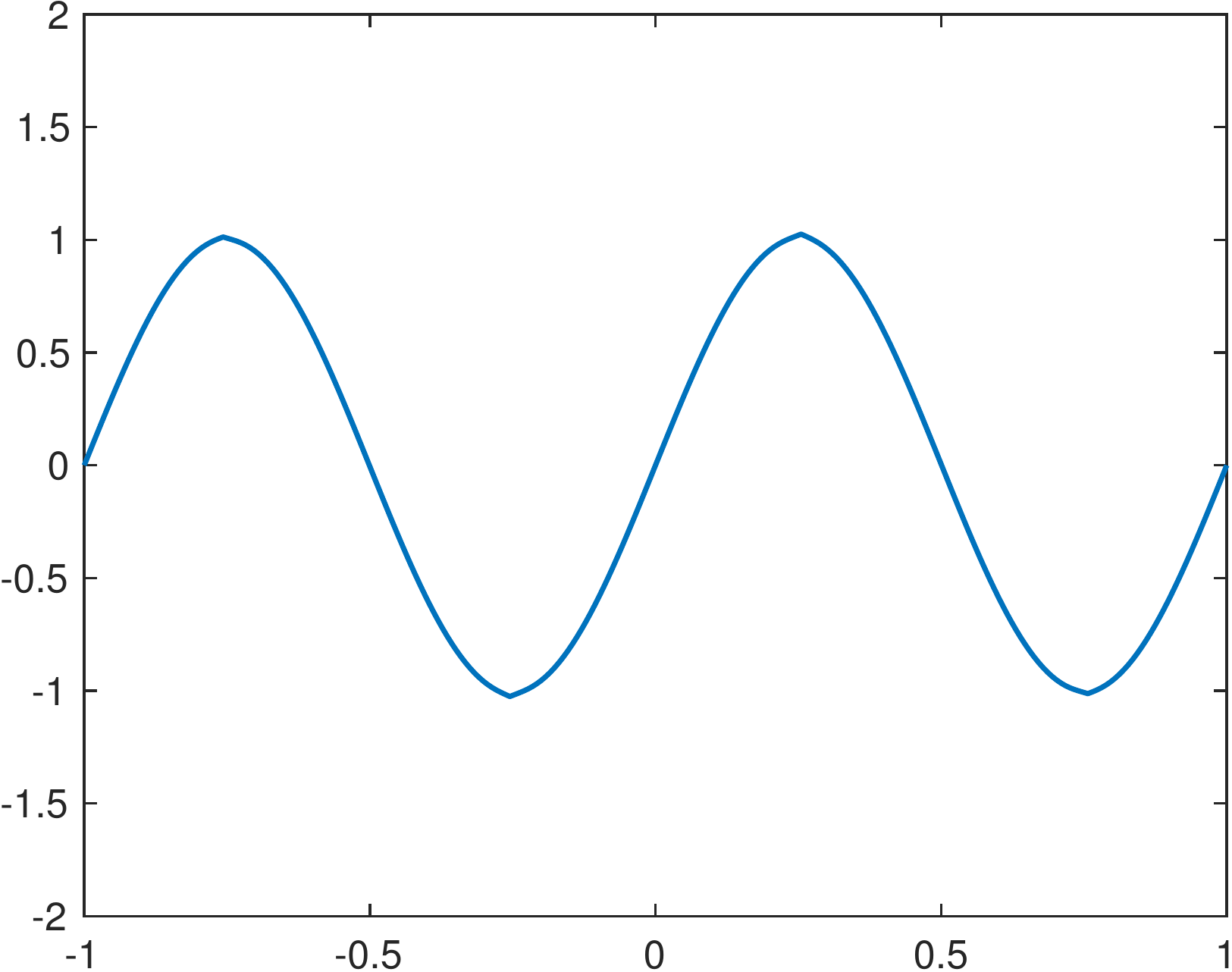} & \includegraphics[width=0.2\textwidth]{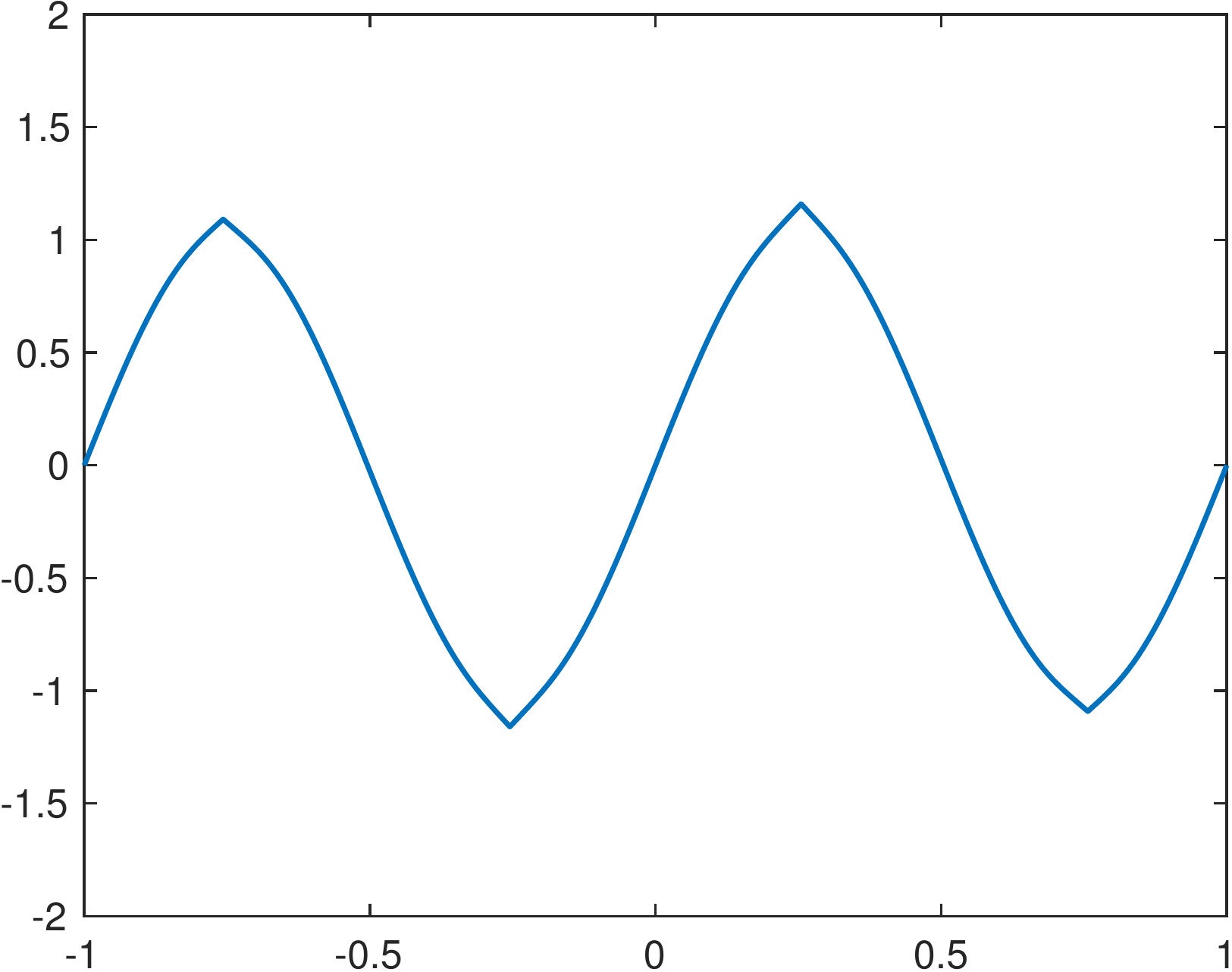} & \includegraphics[width=0.2\textwidth]{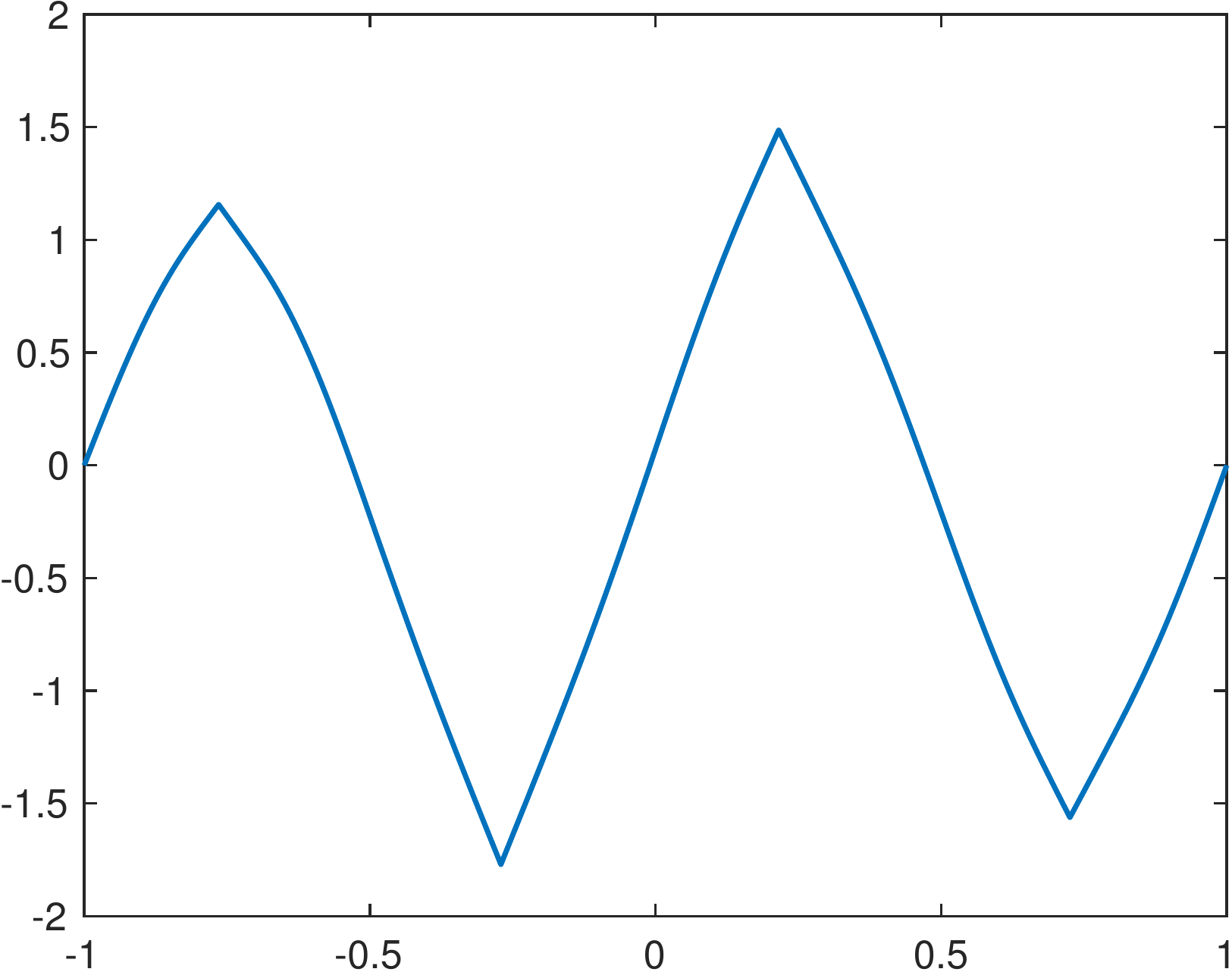}
\end{tabular}
\caption{Solution of the one-dimensional model equation using standard finite differences at times $t \in \{0,1,2,5\}$. Set $dt = h^2/2$ on a  $256$-point grid.}
\label{fig:1Dbreak}
\end{figure}

Next we consider the discretization of the time-dependent equation 
\begin{equation}\label{aff1dt}
u_t = \Aff^{1D}[u] - f.
\end{equation}
Using a forward Euler method in time, and the elliptic method in space leads to 
\begin{equation}\label{aff1dth}
u(\cdot,t+dt) = \Phi_t^{1D,e}(u) := u(\cdot,t) + dt (\Aff^{1D,e}[u(\cdot,t)]-f).
\end{equation}

A scaling argument suggests $dt = \bO(h^{\frac{4}{3}})$ might be stable, since the operator $\Aff^{1D}$ is homogeneous to this order in $h$.   In fact, we are able to prove the following.   
\begin{lemma} When $f=0$ in \eqref{aff1dt}, the solution map $\Phi^{1D,e}$ satisfies the maximum principle
\[\min \Phi^{1D,e}_t(u) \leq \Phi^{1D,e}_{t+dt}(u) \leq \max \Phi^{1D,e}_t(u),\]
provided $dt \leq (h^4/2)^{1/3}$.
\end{lemma}

\begin{proof}
We have
\[
-\frac{2}{h} \Abs{u_x^h}^- \leq -u_{xx}^h \leq \frac{2}{h}\Abs{u_x^h}^+.
\]
Thus, since $A^+ \in ND^+(\R^2)$ and $A^- \in ND^-(\R^2)$,
\[
0 \leq A^+\left( \Abs{u_x^h}^+,  -u_{xx}^h  \right) \leq A^+\left( \Abs{u_x^h}^+,  \frac{2}{h}\Abs{u_x^h}^+  \right) = \left(\frac{2}{h}\right)^{1/3}\Abs{u_x^h}^+
\]
and
\[
0 \geq A^+\left( \Abs{u_x^h}^-,  -u_{xx}^h  \right) \geq A^+\left( \Abs{u_x^h}^-,  -\frac{2}{h}\Abs{u_x^h}^-  \right) = -\left(\frac{2}{h}\right)^{1/3}\Abs{u_x^h}^-
\]
and so
\[
-\left(\frac{2}{h}\right)^{1/3}\Abs{u_x^h}^- \leq -\Aff^{1D,e}[u] \leq \left(\frac{2}{h}\right)^{1/3}\Abs{u_x^h}^+
\]
Hence
\[
u(x) - dt\left(\frac{2}{h}\right)^{1/3}\Abs{u_x^h}^+
\leq
u(x) + dt \Aff^{1D,e}[u]
\leq
u(x) + dt\left(\frac{2}{h}\right)^{1/3}\Abs{u_x^h}^-
\]

The proof now follows due to the choice of $dt$ and from observing that
\[
u(x) + h \Abs{u_x^h}^- = \max\{u(x+h),u(x-h),u(x)\}
\]
and
\[
u(x) - h \Abs{u_x^h}^+ = \min\{u(x+h),u(x-h),u(x)\}.
\]
\end{proof}

However, as the following example shows, the maximum principle by itself is not enough to guarantee convergence and so the above choice for the time step is not guaranteed to produce a convergent scheme.   In practice, the above choice for the time step leads to a divergent but bounded scheme with nonlinear instabilities.
\begin{example}\label{Example1D}
We solved \eqref{aff1dt} using \eqref{aff1dth}.  We took $u(x,0) = u_0$ where $u_0(x) = C x^{4/3}$ and $f =\Aff^{1D}[u_0]$. In Figure \ref{fig:Example1dMax}, we plot the numerical solution obtained at different times with $dt = (h^4/4)^{1/3}$ (The conservative choice of the time step is to account for the fact the equation is not homogeneous). The exact solution is clearly unstable.  For larger times, the solution has the same shape but with small high frequency oscillations in time.  On the other hand, choosing $dt = h^2/2$ leads to convergence. (The data is not presented to save space.)

\begin{figure}[htp]
\centering
\begin{tabular}{cccc}
\includegraphics[width=0.2\textwidth]{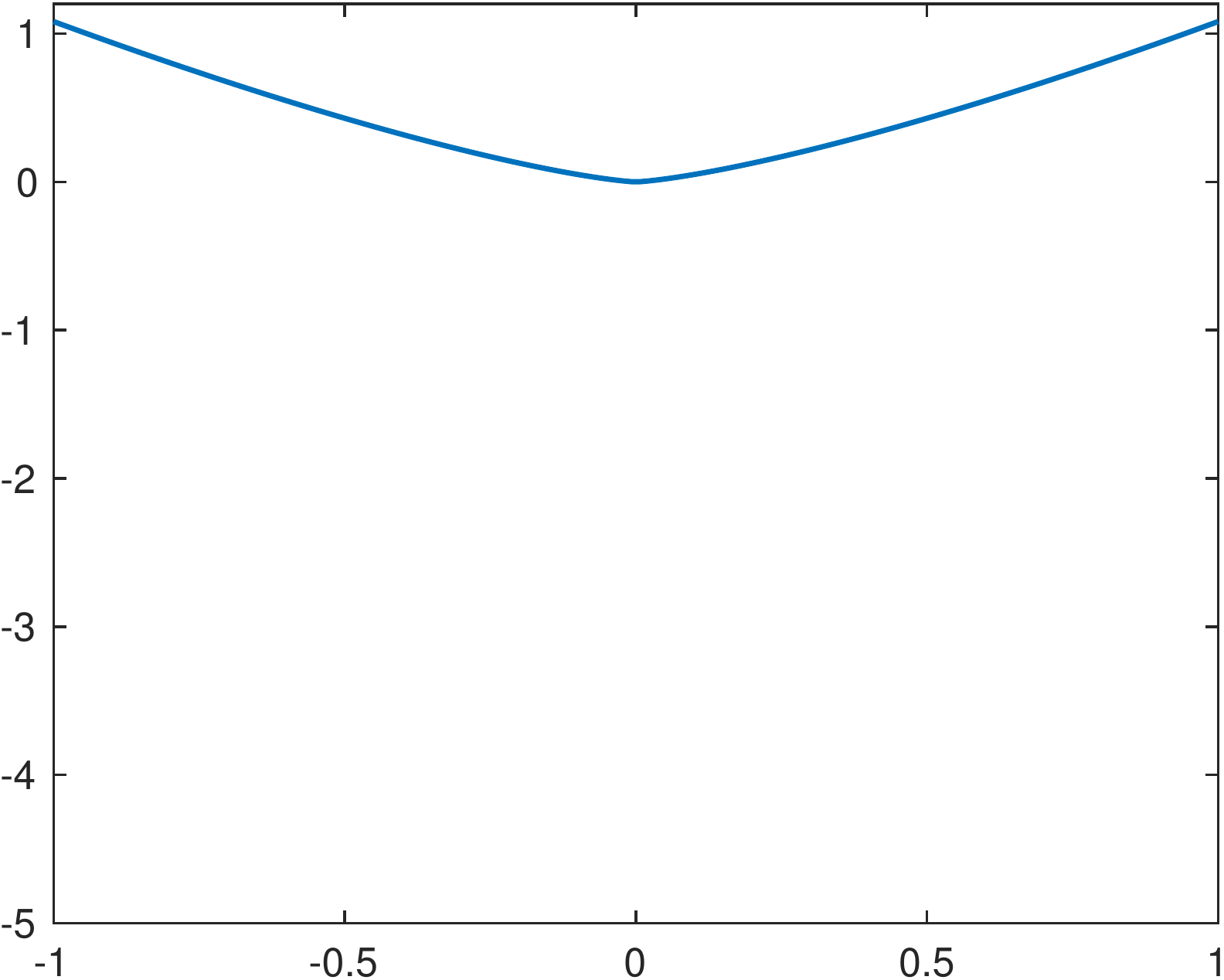} & \includegraphics[width=0.2\textwidth]{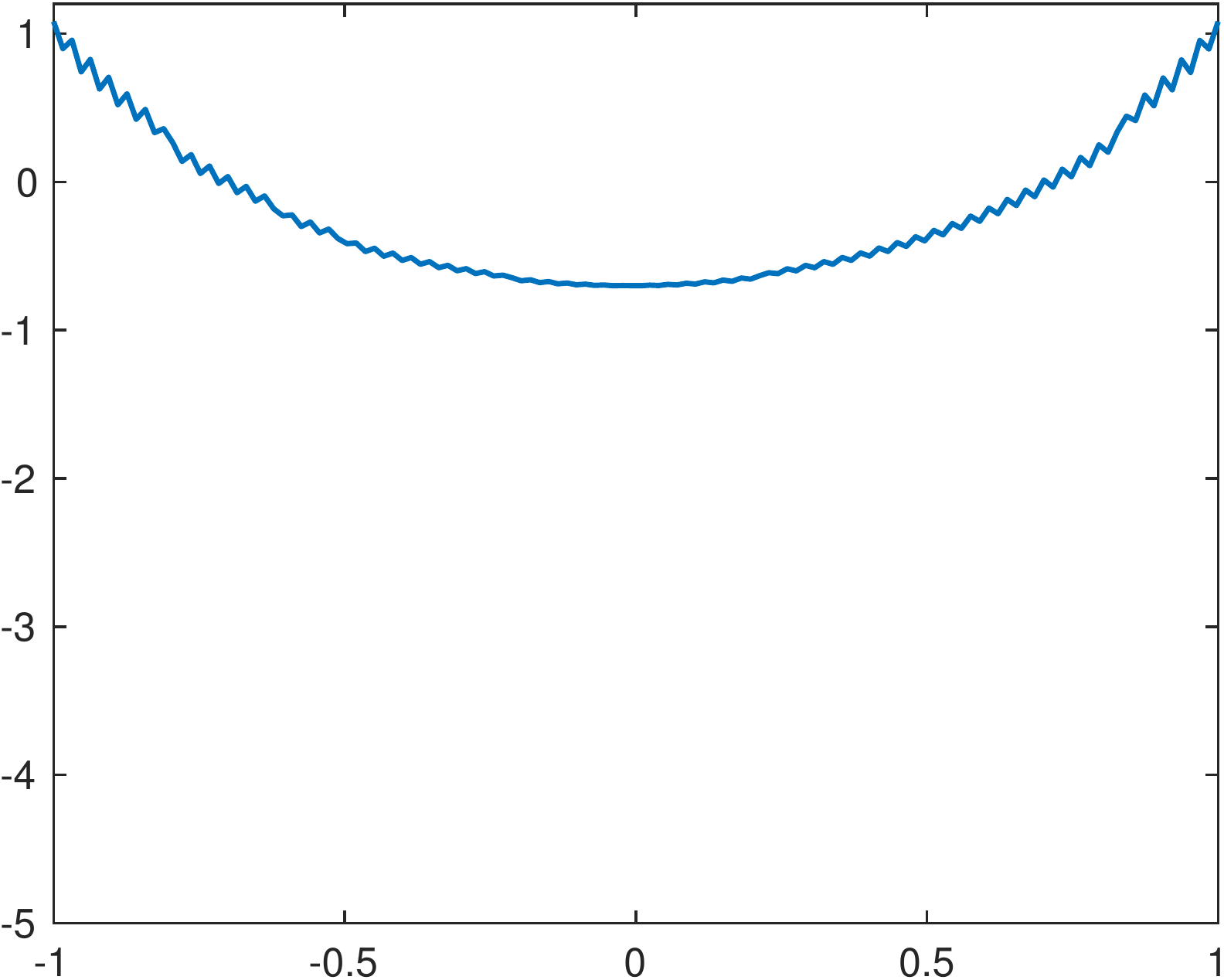} & \includegraphics[width=0.2\textwidth]{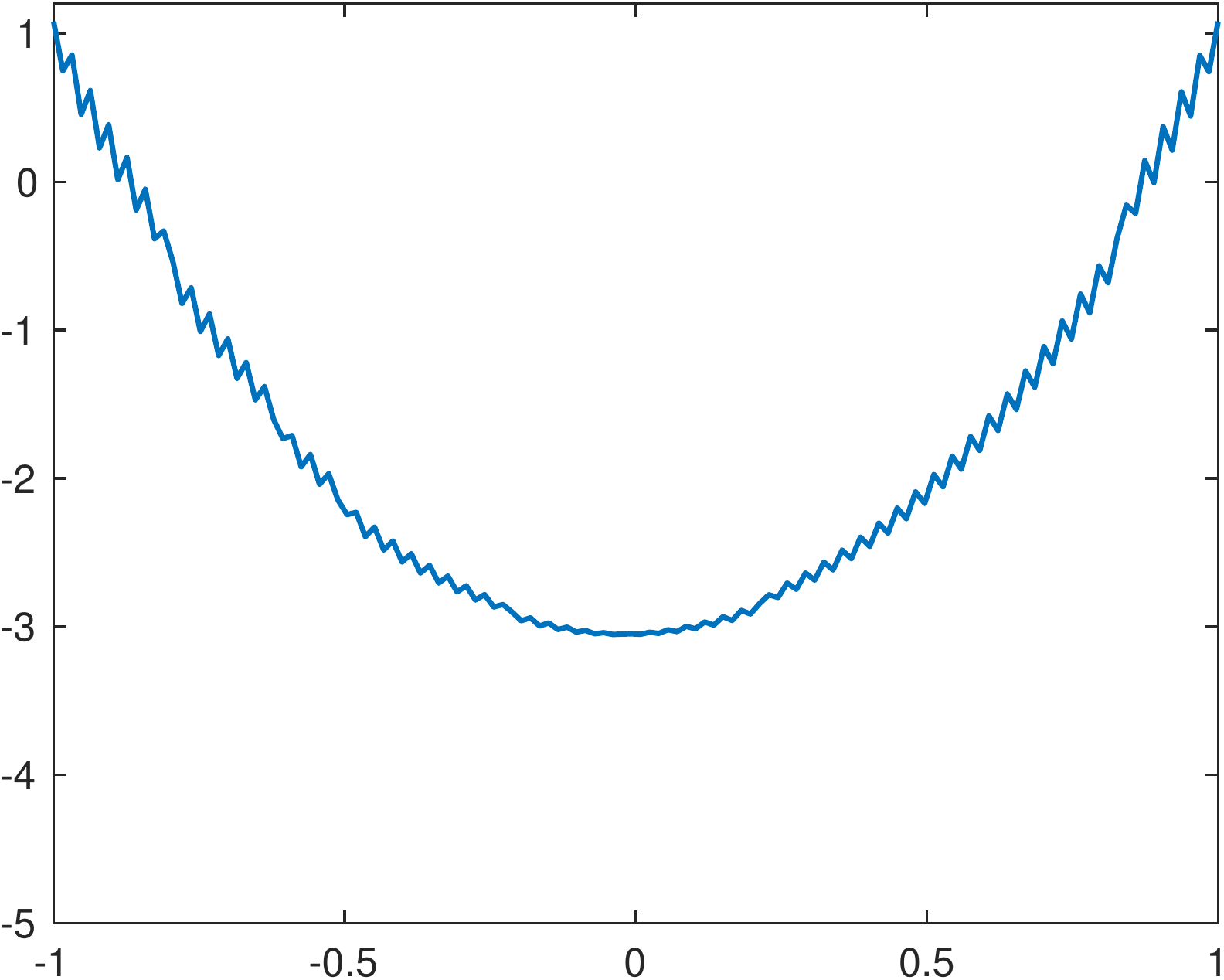} & \includegraphics[width=0.2\textwidth]{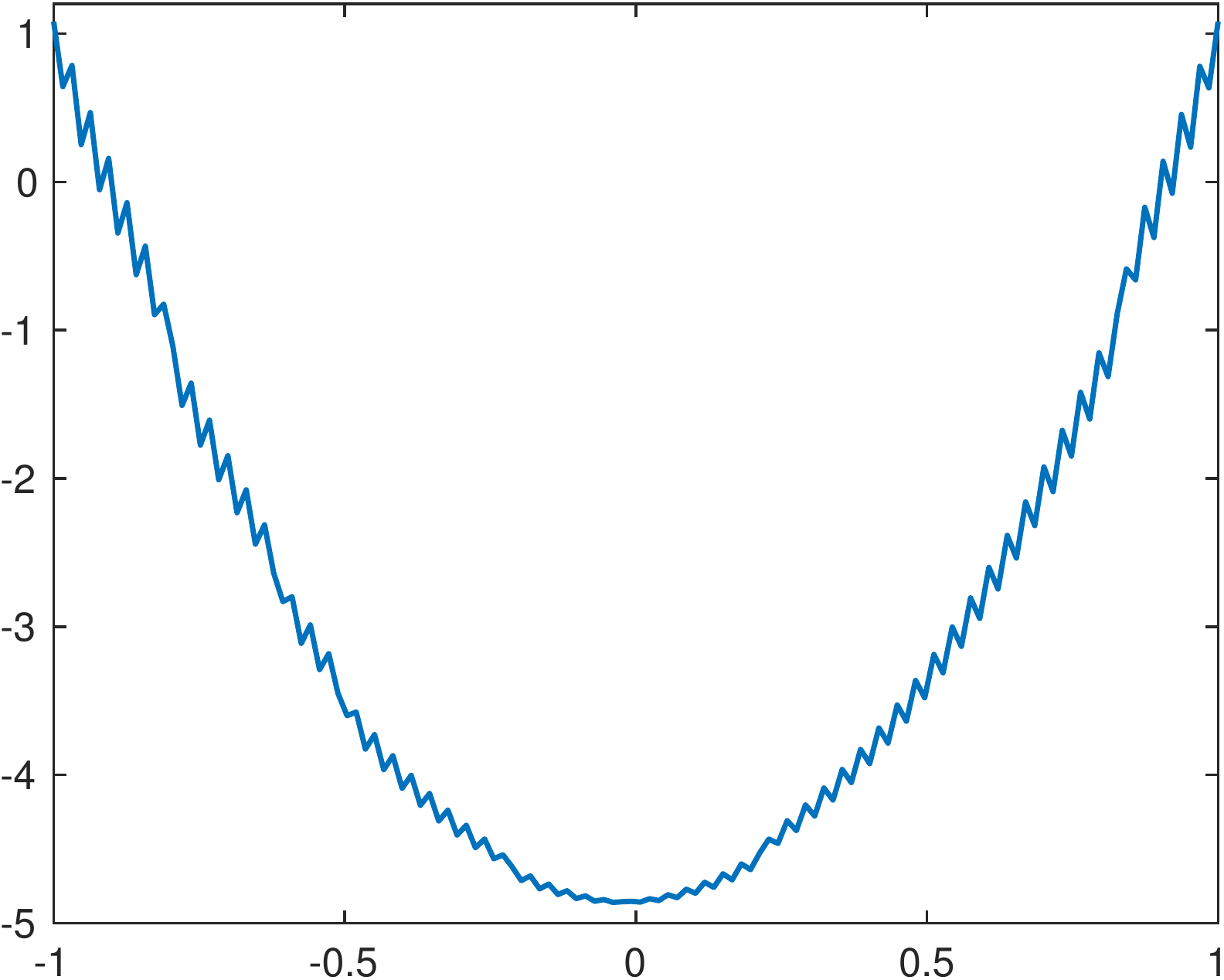}
\end{tabular}
\caption{Plot of the solution obtained described in Example \ref{Example1D} at time $t \in \{0,1,5,20\}$ on a $128$-point grid.}
\label{fig:Example1dMax}
\end{figure}
\end{example}

\subsection{Two dimensions}

In this section we define and study the standard finite difference scheme for $\Aff[u]$. We show that the accuracy degenerates near points where the affine curvature is zero. We give an example of a steady solution where the standard finite difference scheme breaks down.

Using \eqref{curvature}, we can write
\begin{align*}
\Aff[u]	
		& = \left(u_{xx} u_y^2 - 2u_x u_y u_{xy} + u_{yy} u_x^2\right)^{1/3}.
\end{align*}

\begin{definition}
Let $u:\R^2\to\R$. We define the scheme
\begin{align}\tag*{$(AC)^a$}\label{accurateAC}
\Aff^a[u] = \left(u_{xx}^h (u_y^h)^2 - 2u_x^h u_y^h u_{xy}^h + u_{yy}^h (u_x^h)^2\right)^{1/3}
\end{align}
that approximates $\Aff[u]$.
\end{definition}

\begin{remark}
The $u_{xy}^h$ term is not elliptic, and consequently $-\Aff^a$ is not elliptic.
\end{remark}

\begin{lemma}\label{lemma:standardaccuracy}  
Let $(x,y) \in \Omega$ be a reference point on the grid and $\phi$ be a smooth function that is defined in a neighborhood of the grid. Then the scheme $\Aff^a[\phi]$ defined by \ref{accurateAC} approximates $\Aff[\phi]$  with accuracy
\[
\Aff^a[\phi](x,y) - \Aff[\phi](x,y) = 
\begin{cases}
\bO(h^2), & \Aff[\phi](x,y) \neq 0,\\
\bO(h^{\frac{2}{3}}), &\Aff[\phi](x,y) = 0.
\end{cases}
\]
\end{lemma}

\begin{proof}
All the standard finite differences used are second order accurate. Therefore
\[\phi_{xx}^h \left(\phi_y^h\right)^2 - 2\phi_x^h \phi_y^h \phi_{xy}^h + \phi_{yy}^h \left(\phi_x^h\right)^2 = \phi_{xx} \phi_y^2 - 2\phi_x \phi_y \phi_{xy} + \phi_{yy} \phi_x^2 + \bO(h^2),\]
in other words, $(\Aff^a[\phi])^3 = (\Aff[\phi])^3 + \bO(h^2)$. In addition, the Taylor expansion tells us that
\[(r+t)^{1/3} = r^{1/3} + \frac{t}{3r^{2/3}} + \bO(t^2)\]
when $r \neq 0$. Hence, when $\Aff[\phi](x,y) \neq 0$, it follows that $\Aff^a[\phi] = \Aff[\phi]+\bO(h^2)$.
Otherwise, when $\Aff[\phi](x,y) = 0$, we can only conclude that $\Aff^a[\phi] = \Aff[\phi]+\bO(h^{\frac{2}{3}})$.
\end{proof}

Now we present an example that shows that the scheme $\Aff^a$ does not converge.  We choose a level set function and a right hand side so that the equation is a steady state (see Example~\ref{ex:static}(d) for more details). 
Starting from initial data corresponding to the exact solution, and evolving in time with a forward Euler step, the solution changes to order one.  Indeed, the solution does not appear to reach a steady state, even after running for a long time. See Figure~\ref{fig:exBreak} for snapshots in time of the solution.  Similar behaviour was observed on finer grids (although it took a longer time to reach a comparable change in the solution).

\begin{figure}[htp]
\centering
\begin{tabular}{ccc}
\includegraphics[width=0.3\textwidth]{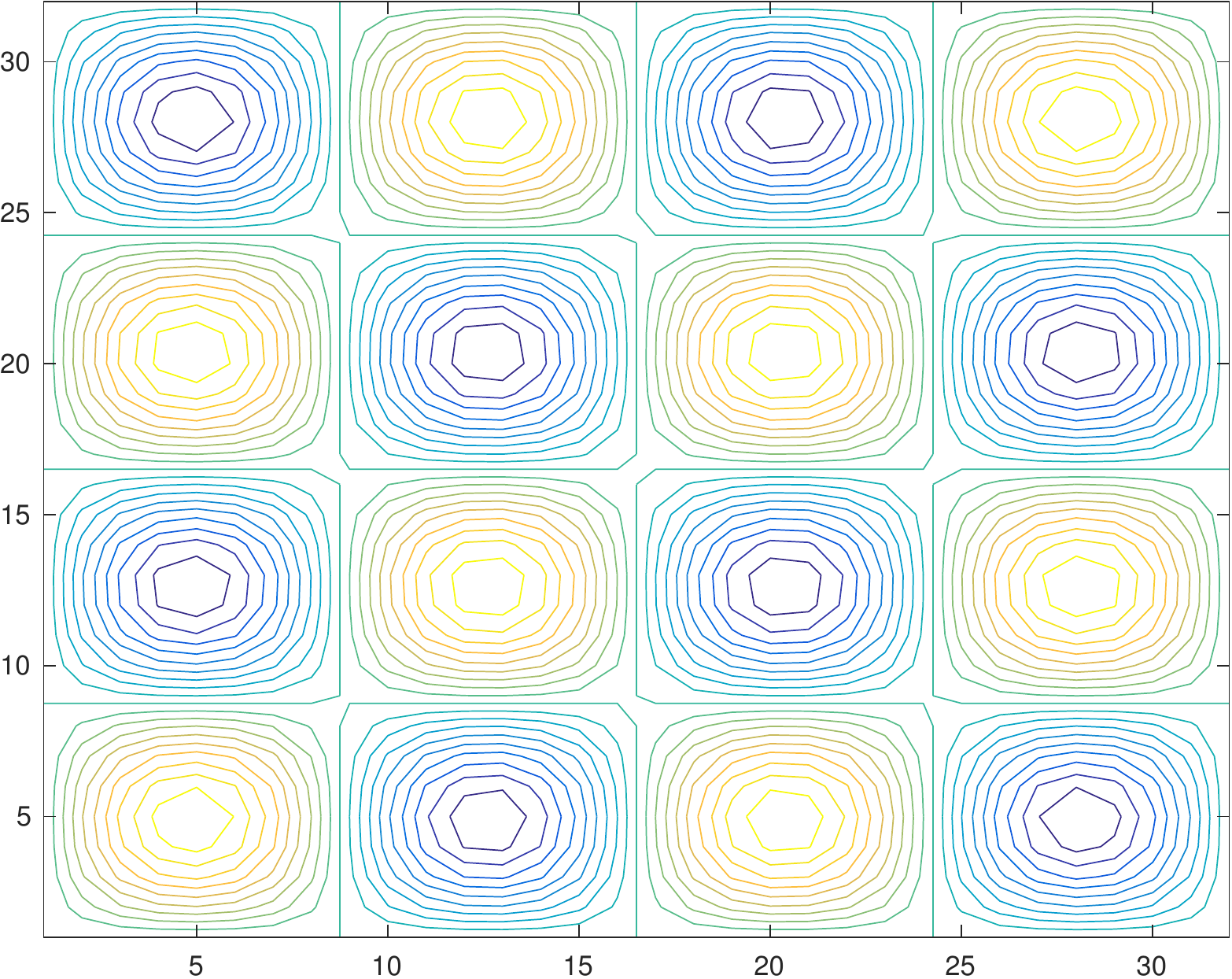} & \includegraphics[width=0.3\textwidth]{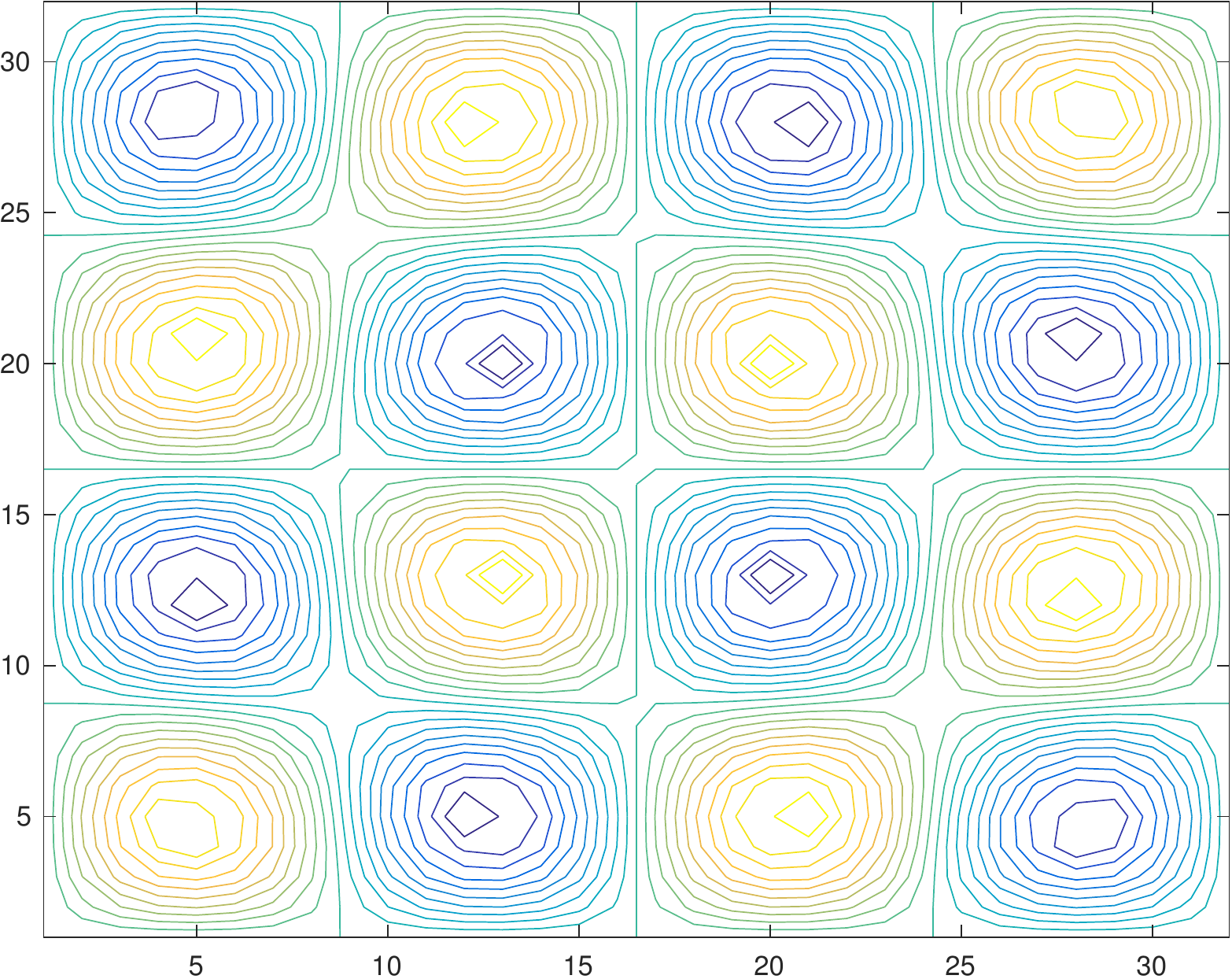} & \includegraphics[width=0.3\textwidth]{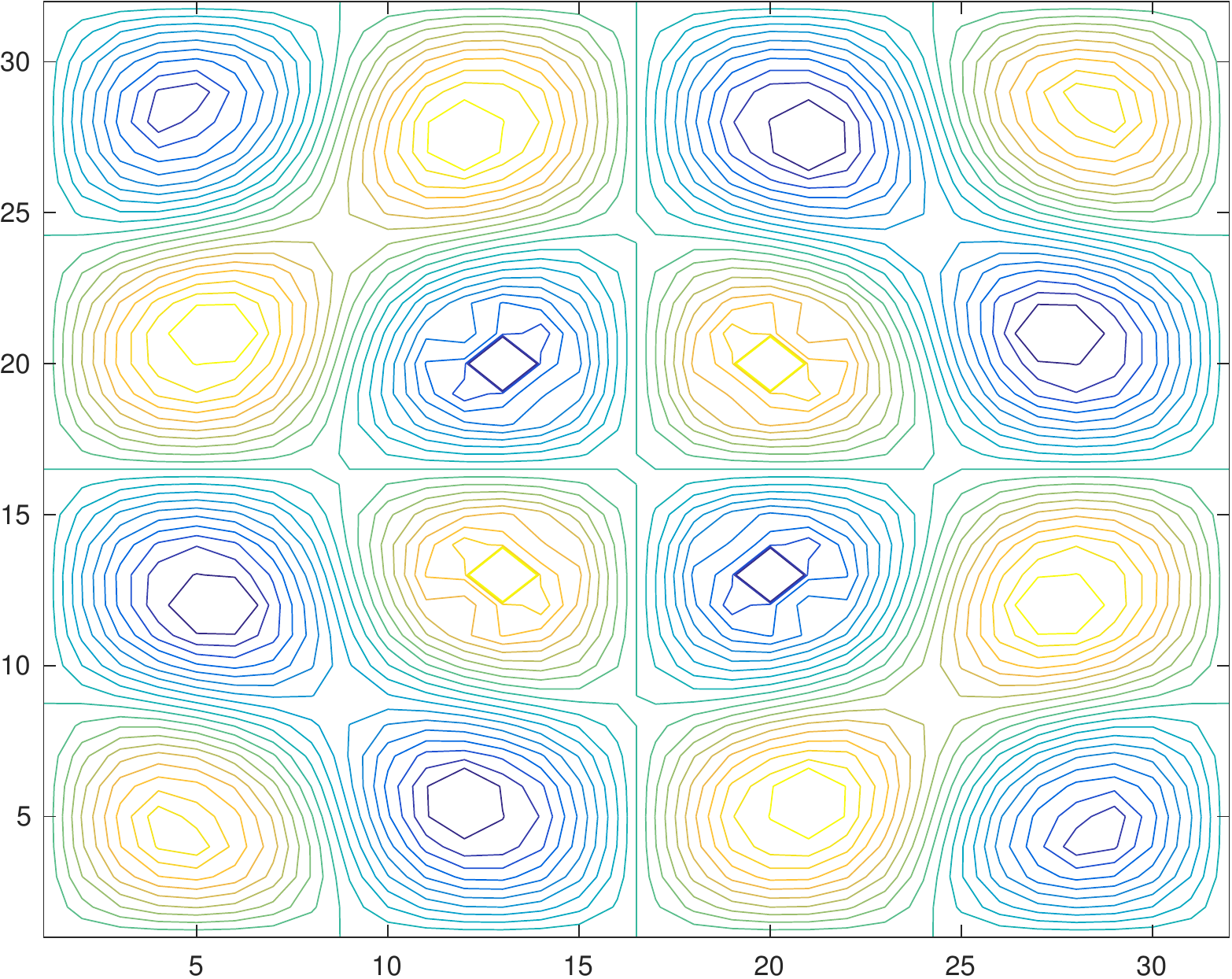}\\
\includegraphics[width=0.3\textwidth]{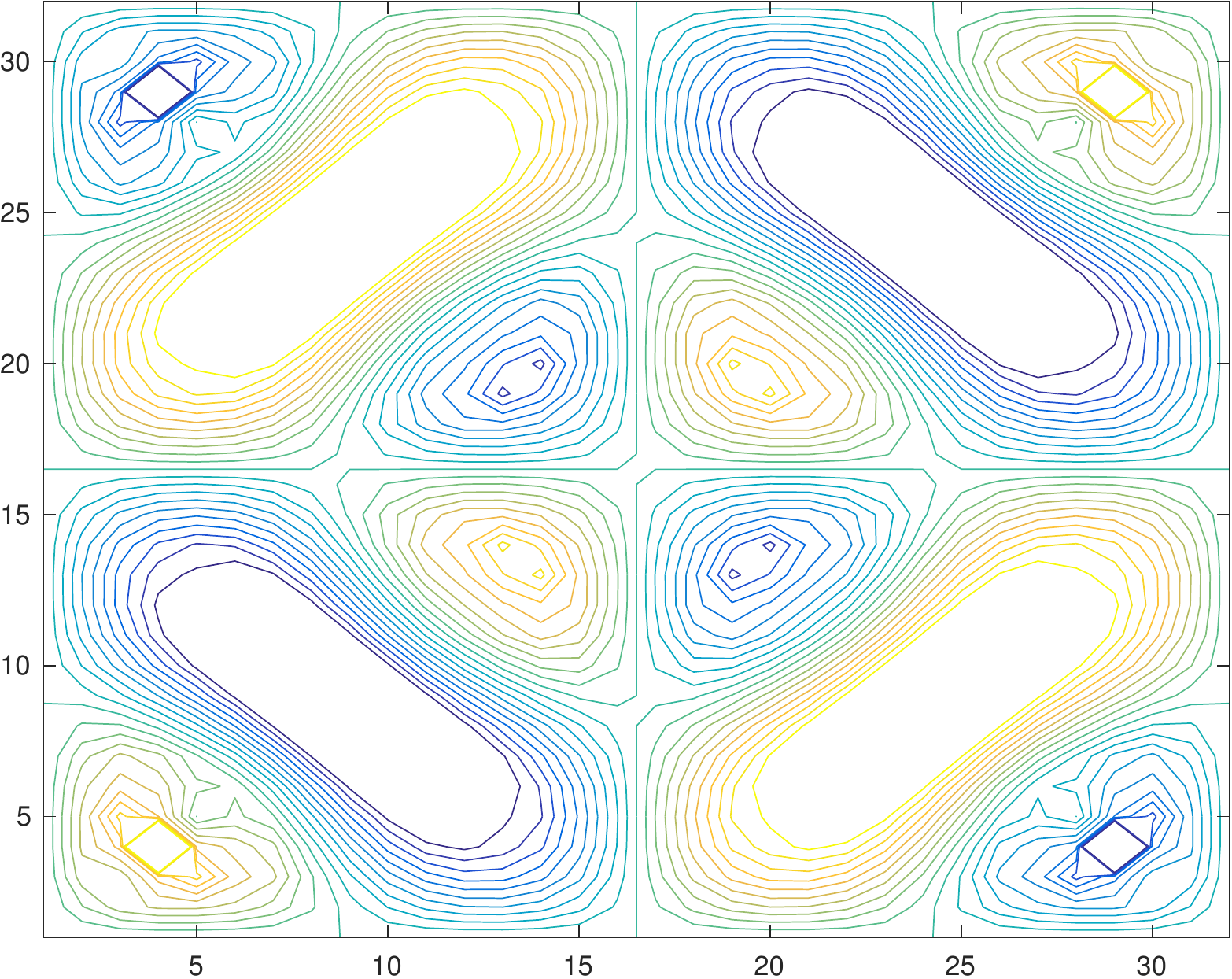} & \includegraphics[width=0.3\textwidth]{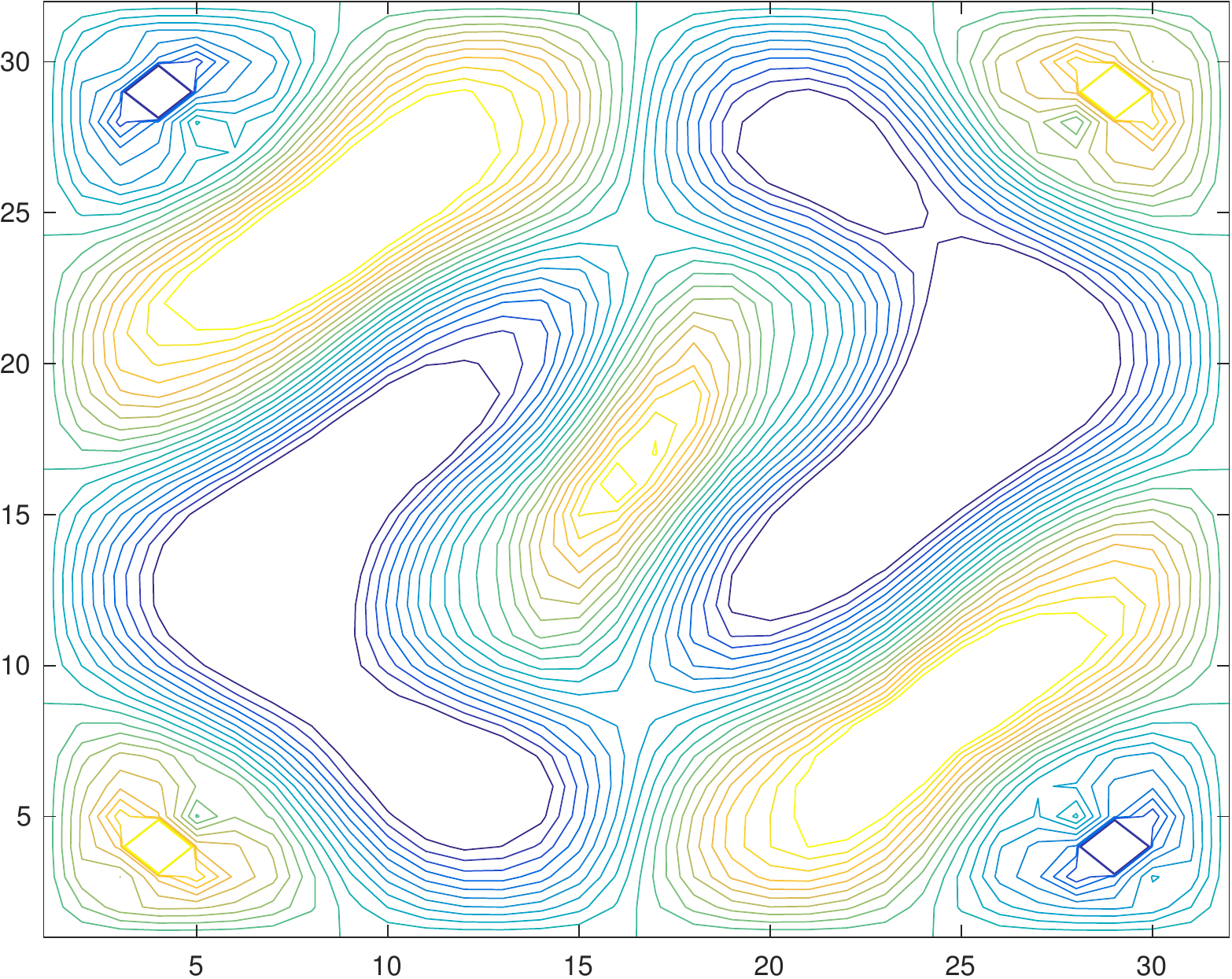} & \includegraphics[width=0.3\textwidth]{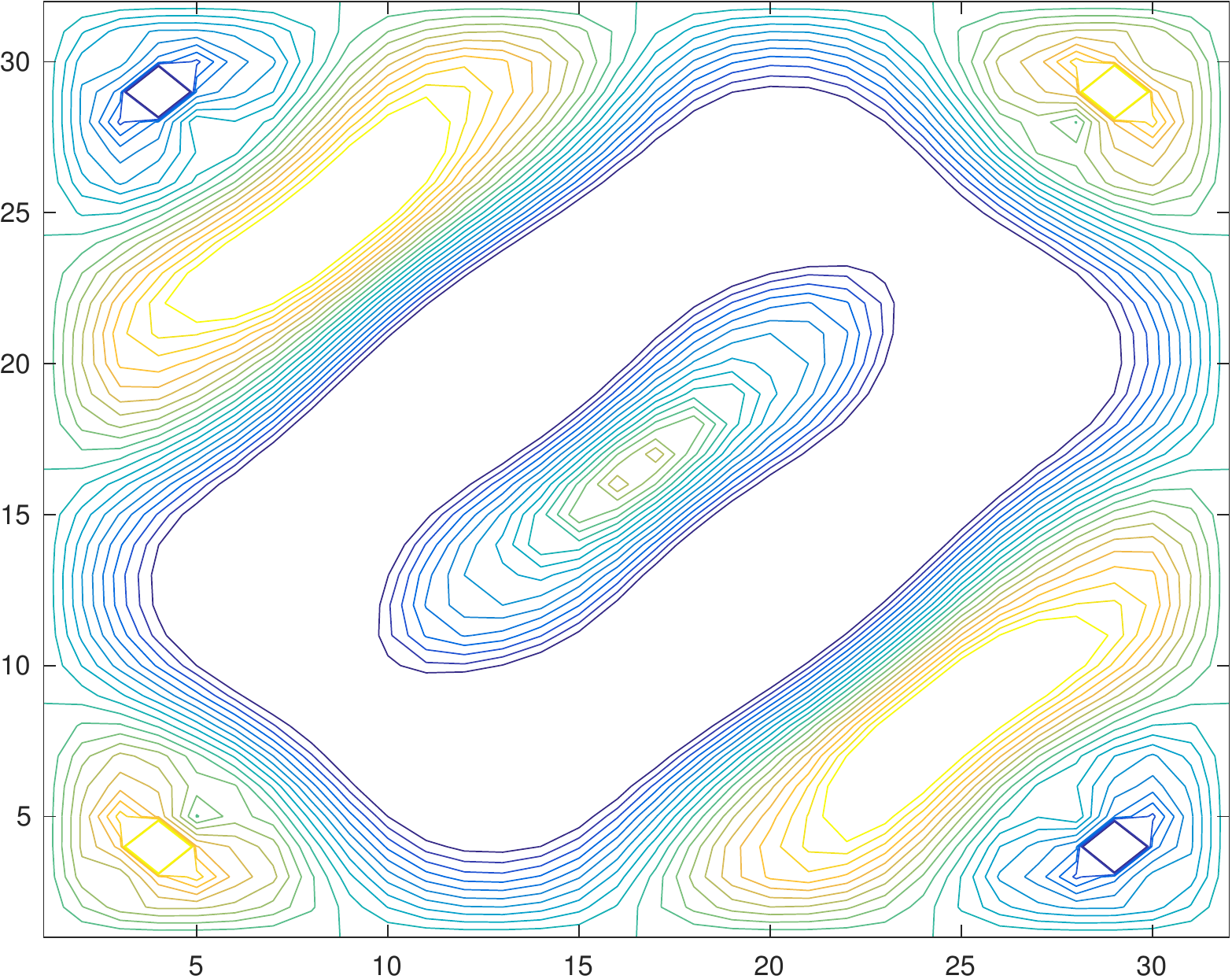}
\end{tabular}
\caption{Example \ref{ex:static} $(d)$ with the standard finite difference solver: level sets of the solution at times $t = 0$ (upper left), $t = 15$ (upper center), $t = 17$ (upper right), $t = 20$ (lower left), $t = 40$ (lower center) and $t = 50$ (lower right) with $dt = h^2/2$ on a $32 \times 32$ grid.}
\label{fig:exBreak}
\end{figure}

%

\section{Convergent finite difference methods}\label{sec:numerics}

Based on the convergence theory discussed above, and on the elliptic discretization of the one dimensional model equation, we now build a discretization of the affine curvature operator.   This discretization is based on the median scheme for the mean curvature operator from \cite{ObermanMC}.  We could also use the morphological operator \cite{MorphologicalMC}, which results in a very similar discretization of the operator $\Aff[u]$.  
We establish the accuracy of the discretization, and show that it is elliptic.  The scheme is augmented to a more accurate, but still convergent,  filtered scheme which interpolates between the standard scheme and the elliptic scheme.   We regularize the operator, which allows us to build a convergent monotone time discretization. 

\subsection{Median for $\Delta_1 u$}\label{sec:schemeMC}

In \cite{ObermanMC} an elliptic scheme for $\Delta_1 u(x)$ is presented based on taking the median of the values of $u$ sampled in a small approximately circular neighborhood of $x$.  The motivation follows from observing, using \eqref{curvature}, that  
\[
\Delta_1 u = u_{\T\T},
\qquad \T = \frac{(-u_y,u_x)}{(u_x^2+u_y^2)^{1/2}},
\]
where $\T$ is the (Euclidean) unit tangent. The median captures an approximation to $u_{\T \T}$, the second tangential derivative of $u$, since the larger values point in the direction of the gradient and the smaller values point in the opposite directions. 

We outline the scheme here.  We will define it at the reference point $(x,y)$. Let $(x_1,y_1),\ldots,(x_{n_S},y_{n_S})$ be the $n_S$ neighbours and set $d\theta = \frac{2\pi}{n_S}$. We refer to $d\theta$ as the directional resolution. Denote the value of the solution at the point $(x_i,y_i)$ by $u_i$. Set $\vi = (x_i,y_i)-(x,y)$ and $(x_{-i},y_{-i}) = (x,y)-\vi$. We choose the neighbours such that
\[(x_{i+1},y_{i+1}) = h \: n_\theta\left(\cos(i d\theta),\sin(i d\theta)\right) + (e_i,f_i)\]
where $|e_i|,|f_i| \leq h$. Thus $h$ is the spatial resolution and $n_\theta$ denotes the width of the stencil. In fact, for $n_\theta \leq 5$, the neighbours in the first quadrant are given by
\[(x_{i+1},y_{i+1}) = h\left(\lfloor n_\theta \cos(i d\theta) \rceil, \lfloor n_\theta \sin(i d\theta) \rceil\right)\]
where $i = 0,\ldots,\frac{n_S}{4}-1$, with the points on the remaining three quadrants being obtained by $\frac{\pi}{2}$, $\pi$ and $\frac{3\pi}{2}$ rotations. Here $\lfloor x \rceil$ denotes the integer closest to $x$.

%

\begin{definition}
Let $u:\R^2 \to \R$. Define the scheme
\bq\label{monotoneMC}\tag*{$(MC)^e$}
\Delta_1^e u(x,y)  = 2\frac{\underset{i=1,\ldots,n_S}\median u_i-u(x,y)}{(h \:n_\theta)^2}
\eq
that approximates $\Delta_1 u$.
\end{definition}

In general, consistency of finite difference schemes follows by Taylor expansions. However, additional care is needed when the PDE is singular, as when $\nabla u(x) = 0$ in \eqref{MCevolution}. We then recall here the definition of consistency for the mean curvature operator.

\begin{definition}
The numerical scheme $F^{h,d\theta}$ is consistent with $\Delta_1$ if for every smooth function $\phi$ and every $(x,y) \in \R^2$
\[\lim_{h,d\theta \to 0} F^{h,d\theta}[\phi] = \Delta_1\phi\]
at $(x,t)$ if $\nabla \phi(x,y) \neq 0$ and
\[\lambda \leq \liminf_{h,d\theta \to 0} F^{h,d\theta}[\phi] \leq \limsup_{h,d\theta \to 0} F^{h,d\theta}[\phi] \leq \Lambda\]
at $(x,t)$ where $\lambda,\Lambda$ are the least and greatest eigenvalues of $D^2\phi(x,y)$, otherwise.
\end{definition}

\begin{lemma}\label{degenerateellipticMC}The finite difference scheme $-\Delta_1^e u$, given by \ref{monotoneMC}, is elliptic.
\end{lemma}
\begin{proof}
We can rewrite the scheme as
\[-\Delta_1^e u = 2\frac{\underset{i=1,\ldots,n_S}\median (u(x,y)-u_i)}{(h \:n_\theta)^2}.\]
Since $\median$ is a nondecreasing function, we can immediately conclude that $\Delta_1^e u$ is elliptic.
\end{proof}

\begin{lemma}\label{consistencyMC}
Let $(x,y) \in \Omega$ be a reference point on the grid and $\phi$ be a smooth function that is defined in a neighborhood of the grid. Then the scheme $\Delta_1^e \phi$ defined by \ref{monotoneMC} is consistent.  Further, it approximates $\Delta_1$ with accuracy
\[\Delta_1^e \phi(x,y) = \Delta_1 \phi(x,y) + \bO\left((h \: n_\theta)^2 + d\theta\right),\]
when $\Abs{\nabla u(x,y)} \neq 0$.
\end{lemma}
\begin{remark}
Since $d\theta = \bO(\frac{1}{n_\theta})$, the ``optimal'' choice is $h = \bO(n_\theta^{-3/2})$, which results in accuracy of $\bO(h^{2/3})$.  However this also requires a dense stencil.  In practice, we use a fairly narrow stencil, and combine with a filtered scheme for more accuracy.
\end{remark}

The proof can be found in \cite{ObermanMC}. 

\subsection{Elliptic scheme for $-\Aff[u]$}

We now construct an elliptic scheme for $-\Aff[u]$ following the same procedure as in \autoref{subsec:elliptic1D}. This is accomplished by writing $\Aff[u]$ in terms of $|\nabla u|$ and $\Delta_1 u$ as
\begin{align*}
\Aff[u]	& = \Abs{\nabla u} k[u]^{1/3} 	 = (\Abs{\nabla u}^2 \Delta_1 u)^{1/3} = A(\Abs{\nabla u},\Delta_1 u)
\end{align*}
and using the elliptic schemes $\abs{\nabla u^h}^+$ and $\abs{\nabla u^h}^-$ presented in \autoref{sec:schemeEikonal} and $\Delta_1^e u$ described in \autoref{sec:schemeMC}. 
\begin{remark}
We chose to discretize $\Delta_1u$ with the median scheme \ref{monotoneMC}. However, other schemes could have been, for instance the morphological scheme \cite{MorphologicalMC}.
\end{remark}

\begin{definition}
Let $u:\R^2\to\R$. We define the scheme
\begin{align}\tag*{$(AC)^e$}\label{monotoneAC}
-\Aff^e[u]	= 
A^+\left(\Abs{\nabla u^h}^+,  -\Delta_1^e u\right) + A^-\left( -\Abs{\nabla u^h}^-, -\Delta_1^e u \right )
\end{align}
that approximates $-\Aff[u]$.
\end{definition}

\begin{remark}
The above approach can be generalized to obtain an elliptic scheme for the elliptic operator $-\Abs{\nabla u} G(k[u])$, where $G$ is nondecreasing and homogeneous of order  $\alpha \le 1$,  $G(t r) = t^\alpha G(r)$. (This PDE represents curves evolving with normal speed $G(k)$). In such cases, write
\begin{align*}
\Abs{\nabla u} G(k[u]) = \Abs{\nabla u}^{1-\alpha} G(\Abs{\nabla u} k[u]) = \Abs{\nabla u}^{1-\alpha} G(\Delta_1 u).	
\end{align*}
The elliptic scheme is then given by
\[\left(|\nabla u^h|^+\right)^{1-\alpha} G\left(\left(-\Delta_1^e u\right)^+\right) +\left(|\nabla u^h|^-\right)^{1-\alpha} G\left(\left(-\Delta_1^e u\right)^-\right).\]
\end{remark}

Now we prove the ellipticity and consistency of the scheme $-\Aff^e[u]$.

\begin{lemma}\label{lemma:elliptic}
The finite difference scheme $-\Aff^e[u]$, given by \ref{monotoneAC}, is elliptic.
\end{lemma}
\begin{proof}
The proof is similar to Lemma \ref{lemma:1Delliptic}.
\end{proof}

\begin{lemma}\label{lemma:consistencyAC}
Let $(x,y) \in \Omega$ be a reference point on the grid and $\phi$ be a smooth function that is defined in a neighborhood of the grid. Then the scheme $\Aff^e[\phi]$ defined by \ref{monotoneAC} is consistent with $\Aff[\phi]$ and has accuracy
\[
\Aff^e[\phi](x,y) - \Aff[\phi](x,y) = 
\begin{cases}
\bO\left(h + (h \: n_\theta)^2+ h \: d\theta +d\theta\right), & \text{if }\Aff[\phi](x,y) \neq 0 \text{ and } \Abs{\nabla u}(x,y) \neq 0,\\
\bO\left(((h \: n_\theta)^2 +  h \: d\theta + d\theta)^{1/3}\right), & \text{if }\Aff[\phi](x,y) = 0 \text{ and } \Abs{\nabla u}(x,y) \neq 0,\\
\bO\left(h^{2/3}\right)	& \text{if }\Abs{\nabla u}(x,y) = 0.
\end{cases}
\]
\end{lemma}
\begin{proof}
Suppose first that $\Abs{\nabla u}(x,y) \neq 0$. We have
\[\Abs{\nabla \phi(x,y)}^2 = \left(|\nabla \phi^h(x,y)|^\pm\right)^2 + \bO(h) \quad \text{and} \quad \Delta_1^e \phi(x,y) = \Delta_1 \phi(x,y) + \bO\left(d\theta + (h \: n_\theta)^2\right).\]
(See Lemma \ref{consistencyMC}).
Hence, when $\Aff[\phi](x,y) \neq 0$,
\[\Aff^e[\phi]^3 = \Aff[\phi]^3 + \bO\left(h + (h \: n_\theta)^2 + h \: d\theta +d\theta\right)\]
 and, by the Taylor expansion like in Lemma \ref{lemma:standardaccuracy}, we have $\Aff^e[\phi](x,y) = \Aff[\phi](x,y) + \bO\left(h + (h \: n_\theta)^2 + h \: d\theta +d\theta\right)$. Otherwise, when $\Aff[\phi](x,y) = 0$, we necessarily have $\Delta_1 \phi = 0$ and so
\[\Aff^e[\phi]^3 = \Aff[\phi]^3 + \bO\left((h \: n_\theta)^2 + h \: d\theta +d\theta\right).\]
Therefore, $\Aff^e[\phi](x,y) = \Aff[\phi](x,y) + \bO\left(((h \: n_\theta)^2 +  h \: d\theta + d\theta)^{1/3}\right)$.

When $\Abs{\nabla u}(x,y) = 0$,
\[\Abs{\nabla \phi(x,y)}^2 = \left(|\nabla \phi^h(x,y)|^\pm\right)^2 + \bO(h^2)\]
with $\Delta_1^e u$ being bounded. Hence, $\Aff^e[u] = \Aff[u] + \bO(h^{2/3})$.
\end{proof}
\begin{remark}\label{rmk:accuracyAC}
Since $d\theta = \bO(\frac{1}{n_\theta})$, the ``optimal'' choice is $h = \bO(n_\theta^{-3/2})$, which results in accuracy of $\bO(h^{2/9})$ in the worst case.  However this also requires a dense stencil.  In practice, we use a fairly narrow stencil, and combine with a filtered scheme for more accuracy.
\end{remark}

\subsection{Filtered scheme for $\Aff[u]$}\label{subsec:filter}

The filtered scheme, $\Aff^f[u]$, first introduced in \cite{FroeseObermanFiltered}, is defined to be a continuous interpolation between the accurate and the elliptic scheme, which equals the accurate scheme when the two schemes are within $\e$ of each other.

Here, we define filtered schemes by making use of the auxiliary function $S^\e:\R^2 \to \R$, defined to be a continuous function which for $(a,b) \in R^2$,  is equal to $a$ near the diagonal and $b$ off the diagonal. 
\begin{definition}
Define for $\epsilon > 0$, $A_\epsilon = \left\{(a,b) \in \R^2 \mid |a-b| < \epsilon\right\}$. Set $\rho = 10\epsilon$.   Define $d = {\dist}\left((a,b),A_\epsilon\right)$. 
\[
S^\e(a,b) = \begin{cases}
a							& \text{if } (a,b) \in A_\epsilon,\\
\frac{\rho-d}{\rho}a+\frac{d}{\rho}b	& \text{if } d \leq \rho,\\
b							& \text{otherwise},
\end{cases}
\]
Define the filtered scheme 
\bq\tag*{$(AC)^f$}\label{filterAC}
\Aff^f[u] = S^\e\left(\Aff^a[u],\Aff^e[u]\right)
\eq
where $\e = \e(h,d\theta)$.
\end{definition}

While theoretically, the only requirement on $\epsilon$ to ensure the convergence of the filtered schemes is that $\e \to 0$ as $h,d\theta\to0$, in practice $\epsilon$ must be chosen carefully. Intuitively, it should be large enough to permit the accurate scheme to be active where the solution is smooth (as shown in the next lemma), and small enough to force the monotone scheme to be active whenever needed for convergence (for instance, when the solution is singular)	.

\begin{lemma}
Suppose that the formal discretization errors of the schemes $\Aff^e$ and $\Aff^e$ are $\bO(r^{\beta_a})$ and $\bO(r^{\beta_e})$, respectively. Choose $\alpha$ such that $\beta_a > \beta_e > \alpha > 0$. Then if $\phi$ is smooth and $\epsilon = r^\alpha$, $\Aff^f[\phi] = \Aff^a[\phi]$.
\end{lemma}
\begin{proof}
If $\phi$ is smooth then 
\[
{F^a[\phi]-F^e[\phi]} = {\bO(r^{\beta_a})+\bO(r^{\beta_e})} = \bO(r^{\beta_e})  \leq \bO(\e)
\]
so using the definition of $\Aff^f$ it follows that $\Aff^f[\phi] = \Aff^a[\phi]$, as desired.
\end{proof}

\begin{remark}\label{rmk:choice_epsilon}
The lemma tells us that heuristically we could choose $\epsilon = \sqrt{\text{Acc}[\Aff^e]}$, where $\text{Acc}[\Aff^e]$ denotes the accuracy of $\Aff^e$. In the numerical results presented here, we defined $\epsilon$ based on the accuracy of the scheme away from the singularities of $\Aff[u]$.
\end{remark}

In practice, the filtered scheme allows us to neglect the error coming from the wide stencil, while in theory we still need to send $d\theta \to 0$ for convergence of the filtered scheme. In our numerical examples, we obtain the accuracy of the accurate scheme in most cases.

\subsection{Regularization and Forward Euler method}

As in the for the one dimensional model equation (see \autoref{subsec:Lipschitz1D}), in order to build a provably convergent scheme for the time dependent equation \eqref{ACPDE} we need to regularize the operator.

Write $\Aff[u] = A(\Abs{\nabla u},\Delta_1 u)$, where $A(p,q) = (p^2q)^{1/3}$, and regularize the cube root function as before. This leads to
\[
\Aff^\delta[u]  = A^\delta(\Abs{\nabla u},\Delta_1 u),
\]
where $A^\delta(p,q)  =	\sgn(q)\min( \abs{A(p,q)}, K \abs{p}, L \abs{q} )$.

\begin{remark}
The  regularized operator, $\Aff^\delta[u]$, is still a level set operator. To see this, it is enough to take $K = L = 1/\delta$:
\[\Aff^\delta[u] = \Abs{\nabla u} \sgn(k[u])\min\left(\Abs{k[u]}^{1/3},\frac{\Abs{k[u]}}{\delta},\frac{1}{\delta}\right),\]
The operator reduces to either a multiple of the mean curvature operator, or  the Eikonal equation and otherwise we obtain $\Aff[u]$.
\end{remark}

Similar results regarding ellipticity and consistency hold as for the one-dimensional model, which we present without proof.

\begin{lemma}
For $K=K(\delta)$, $L=L(\delta)$ such that $K\sqrt{L}\geq1$, define the finite difference scheme
\begin{align}\label{elliptic_scheme_regular}\tag*{$(AC)^{e,\delta}$}
-\Aff^{e,\delta}[u] = A^{\delta,+}\left(|\nabla u^h|^+,  -\Delta_1^e u \right) + A^{\delta,-}\left( -|\nabla u^h|^-, -\Delta_1^e u \right )
\end{align}
Then, $-\Aff^{e,\delta}[u]$ is elliptic and consistent with $-\Aff^\delta[u]$.
\end{lemma}

As with the one-dimensional case, $K$ and $L$ need to be chosen appropriately in order for $-\Aff^{e,\delta}[u]$ to be consistent with $-\Aff[u]$. Assuming $K = \bO(h^{-\alpha})$, $L = \bO(h^{-\beta})$ and $d\theta = \bO(1/n_\theta) = \bO(h^\gamma)$, full consistency is obtained when $\alpha \in (0,1)$, $\beta \in (0,\frac{2}{3})$ and $\gamma \in (\beta,\frac{2-\beta}{2})$. The optimal choices are $\alpha \in [1/9,7/9]$, $\beta = 4/9$, $\gamma = 2/3$ and, in such case, the accuracy is $\bO(h^{2/9})$. The proof is similar to Lemma \ref{lemma:accuracyregular1D}. As in the one dimensional model, no accuracy is lost with the regularization (see Remark \ref{rmk:accuracyAC}).

The Lipschitz constant of $\Aff^{e,\delta}[u]$ is given by
\bq\label{Lipconst}
C^h = \frac{K}{h} + \frac{2 L}{(h \: n_\theta)^2},
\eq
with the proof similar to Lemma \ref{lemma:Lipconst1D}. In practice, we will choose $K = c_K h^{-1/9}$ and $L = c_L h^{-4/9}$, which leads to
\[C^h = c_Kh^{-10/9} +2\frac{c_L}{{n_\theta}^2}h^{-22/9}.\]

We can finally define and prove the convergence of the discretizations for the time dependent equation \eqref{ACPDE}. The time derivative is discretized with an explicit forward Euler step, while $\Aff[u]$ is discretized using either the elliptic scheme $\Aff^{e,\delta}$ or the regularized filtered scheme $\Aff^{f,\delta}[u]$ (this is easily defined by replacing the elliptic scheme in $\Aff^f$ by its regularized version). This leads to monotone (resp. filtered) time discretization with solution map given by
\bq\label{discretizationtime}
u(\cdot,t+dt) = u(\cdot,t) + dt \Aff^{e,\delta}[u(\cdot,t)], \quad \left(\text{resp. } = u(\cdot,t) + dt \Aff^{f,\delta}[u(\cdot,t)\right).
\eq

\subsection{Convergence Theorems}

Having proved the ellipticity and consistency of the schemes, the uniform convergence follows as discussed in \autoref{sec:convergencetheory},  provided there exists unique viscosity solutions to the PDEs  \eqref{ACPDE} and $\Aff[u] = f$ along with the homogeneous Neumman boundary conditions, which is assured by the theory in \cite{giga2006surface}, as explained above. 

 We also need existence of solutions to the elliptic and filtered schemes.  These do not need to be unique to apply the convergence theory, but existence and uniqueness results for the schemes follow from the discrete comparison principle and from finite dimensional fixed point theorems, as in \cite{ObermanSINUM} and \cite{FroeseObermanFiltered}.  In practice, we can get discrete comparison from a general theorem by adding a small multiple of $u$ to the scheme, \cite{ObermanSINUM}, or we can do extra work to prove it directly. 

There are two convergence results.  The first is for the elliptic problem, where there is no need for regularization.  The second is for the parabolic problem, where we need to regularize and use an explicit Euler time step. 

\begin{theorem} Let $\Omega$ be a bounded convex domain with a $C^2$ boundary and $u$ denote the unique viscosity solution of $\Aff[u] = f$ in $\Omega$, along with homogeneous Neumann boundary conditions on $\partial \Omega$. For each $\epsilon = \epsilon(h,d\theta) >0$,  let $u^{e,\epsilon}$, (resp. $u^{f,\epsilon}$) be the uniformly bounded solution of $\Aff^e[u] = f$,  (resp. $\Aff^f[u] = f$). Then 
\[
u^{e,\epsilon} \to u \text{ and }  u^{f,\epsilon} \to u, \quad \text{ locally uniformly,  as } \epsilon \to 0.
\]
\end{theorem}
\begin{proof}
The assumptions on $\Omega$ guarantee that the PDE satisfies a comparison principle. Convergence for the elliptic discretization $-\Aff^e$ then follows from the Barles-Souganidis theorem. 
The modification of the proof for filtered schemes can be found in \cite{FroeseObermanFiltered}.
\end{proof}

\begin{theorem}\label{theorem:convergence_regular}
Let $\Omega$ be a bounded convex domain with a $C^2$ boundary. Assume that  $u(x,t)$ is the unique viscosity solution of $u_t = \Aff[u]$ in $\Omega \times [0,\infty)$, along with $u(x,0) = u_0(x)$ and homogeneous Neumman boundary conditions. Assume as well that $K$ and $L$ are picked so that $\Aff^{e,\delta}$ is consistent with $\Aff$. For each $\epsilon = \epsilon(h,d\theta) >0$,  let $u^{e,\e}$ (resp. $u^{f,\epsilon}$) be the uniformly bounded solution of the monotone (resp. filtered) time discretization \eqref{discretizationtime} with $dt \leq 1/{K^\dx}$ where $K^\dx$ is the Lipschitz constant of $\Aff^{e,\delta}[u]$ \eqref{Lipconst}. Then
\[u^{e,\e} \to u \quad \text{and} \quad u^{f,\e} \to u\]
locally uniformly, as $dt,\e  \to 0$.
\end{theorem}
\begin{proof}
The assumptions on $\Omega$ guarantee that the PDE satisfies a comparison principle. The elliptic scheme leads to a monotone time discretization, in other words, the solution map satisfies a comparison principle if $dt \leq 1/{K^\dx}$ where $K^\dx$ is the Lipschitz constant of the elliptic scheme by \cite{ObermanSINUM}. Then the Barles-Souganidis theorem \cite{BSnum} applies. For the time discretization of the filtered scheme, the convergence of filtered schemes in \cite{FroeseObermanFiltered} applies.
\end{proof}

\begin{remark}
Just like in the one-dimensional model, for fixed values of $K,L$, there is a unique viscosity solution, $u^\delta$, of the regularized PDE. Then, fixing $K,L$ and using the discretization above with $dt = \bO((h\: n_\theta)^2)$, the forward Euler method converges uniformly to $u^\delta$ as $h\to 0$. 
\end{remark}

\section{Numerical Results}\label{sec:results}

In this section, we start with a simple example to compare the affine invariant curvature motion and the regularized model. We present different examples on the evolution of a single curve in \autoref{sec:curves} under the affine curvature motion and compare it to the mean curvature motion. We test the accuracy of the schemes by computing solutions to the time-independent PDE in \autoref{subsec:accuracy} and to the time dependent equation in \autoref{sec:accuracytime}.   We mostly considered stationary problems because it was easier to generate exact solutions by applying the operator $\Aff$ to a given function $u$, and including $f = \Aff[u]$ as a source function.   For the time dependent problem, we took advantage of the exact solution from Lemma \ref{lemma:ellipse} to compare the accuracy of the solutions after a short time, $T= 0.1$. We test as well in \autoref{subsec:invariance} if the schemes satisfy numerically the morphology and affine invariance properties that \eqref{ACPDE} satisfies (see Theorem \ref{MoisanTheorem}).

\begin{definition}[Parameters for the discretization]
We used the regularized schemes with $K = 20 h^{-1/9}$ and $L = 20 h^{-4/9}$. For the elliptic discretization, we used two different stencils: the narrow elliptic scheme used $n_\theta = 3$ and the wider elliptic scheme used $n_\theta = 7$.  The filtered discretization used the wider elliptic scheme with 
\[
\epsilon(h,d\theta) = \sqrt{h} + d\theta/10.
\] 
For the forward Euler method with the elliptic and filtered schemes, we used a constant time step of $dt = 1/C_h$ where $C_h$ is given by \eqref{Lipconst}. For the forward Euler method with the standard finite difference scheme, we used the same time step as the filtered scheme, except when computing the steady state solutions in Example \ref{ex:static} where we used $dt = h^2/2$ (this choice of time step proved to be enough in practice). The stopping criteria was simply that the $l^\infty$ norm of the residual $\Abs{\Aff^{f,\delta}[u^n]-f}$ (and similarly for $\Aff^{e,\delta}, \Aff^a$) was below $tol = 10^{-5}$.
\end{definition}

\begin{example}\label{ex:Lipconst_regular}[Comparing the regularized to the unregularized operators]
We compare the evolution of an ellipse by the affine invariant curvature motion and by the regularized model. The ellipse should remain an ellipse of fixed eccentricity. The results were obtained by numerically solving \eqref{ACPDE} and $u_t = \Aff^\delta[u]$
respectively, with the initial condition
\[u_0(x,y) = \min\left\{\left(\frac{x}{2}\right)^2+y^2-1,1\right\}\]
and homogeneous Neumann boundary conditions on $\Omega = [-4,4]^2$.
The narrow elliptic scheme was used for the spatial discretization. The difference between the level sets of the solution of the two equations is visually indistinguishable. Measured in the $l^\infty$ norm ranges from $10^{-6}$ for early times steps and $10^{-7}$ for the later time steps. The level sets of the numerical solution obtained by solving the regularized PDE are depicted in Figure~\ref{fig:ACvsMC}.
\end{example}

\begin{remark}
While the regularized scheme with the more stringent time step is needed for the convergence proof, in practise we found, as reported in the previous example, using the (unregularized) elliptic discretization with the time step $dt = h^2/2$ gave nearly identical results to within two or more significant digits.   For the remaining examples, we present results using the regularized schemes (with $K = 20h^{-1/9}$ and $L = 20h^{-4/9}$).
\end{remark}

\subsection{Numerical examples showing curve evolution}\label{sec:curves}

In this section we present some numerical examples illustrating the geometric properties of the PDEs.  

\begin{example}[Ellipse]\label{ex:ellipse}
We compare the evolution of an ellipse by the affine invariant curvature motion and by mean curvature. For the former, the ellipse should remain an ellipse of fixed eccentricity. For the latter, the ellipse asymptotically approaches a circle instead. The results were obtained by numerically solving \eqref{ACPDE} and \eqref{MCevolution}, respectively, with the initial condition
\[u_0(x,y) = \min\left\{\left(\frac{x}{2}\right)^2+y^2-1,1\right\}\]
and homogeneous Neumann boundary conditions. We took $[-4,4]^2$ as the computational domain on a $128\times 128$ grid. 
As for the scheme used, we chose the narrow elliptic schemes for both. In Figure \ref{fig:ACvsMC}, we plot the zero level sets obtained at $t \in \{0,0.1,0.3,0.5,0.7,0.9\}$.
\end{example}


\begin{figure}[htp]
\centering
\begin{tabular}{cc}
\includegraphics[width=0.25\textwidth]{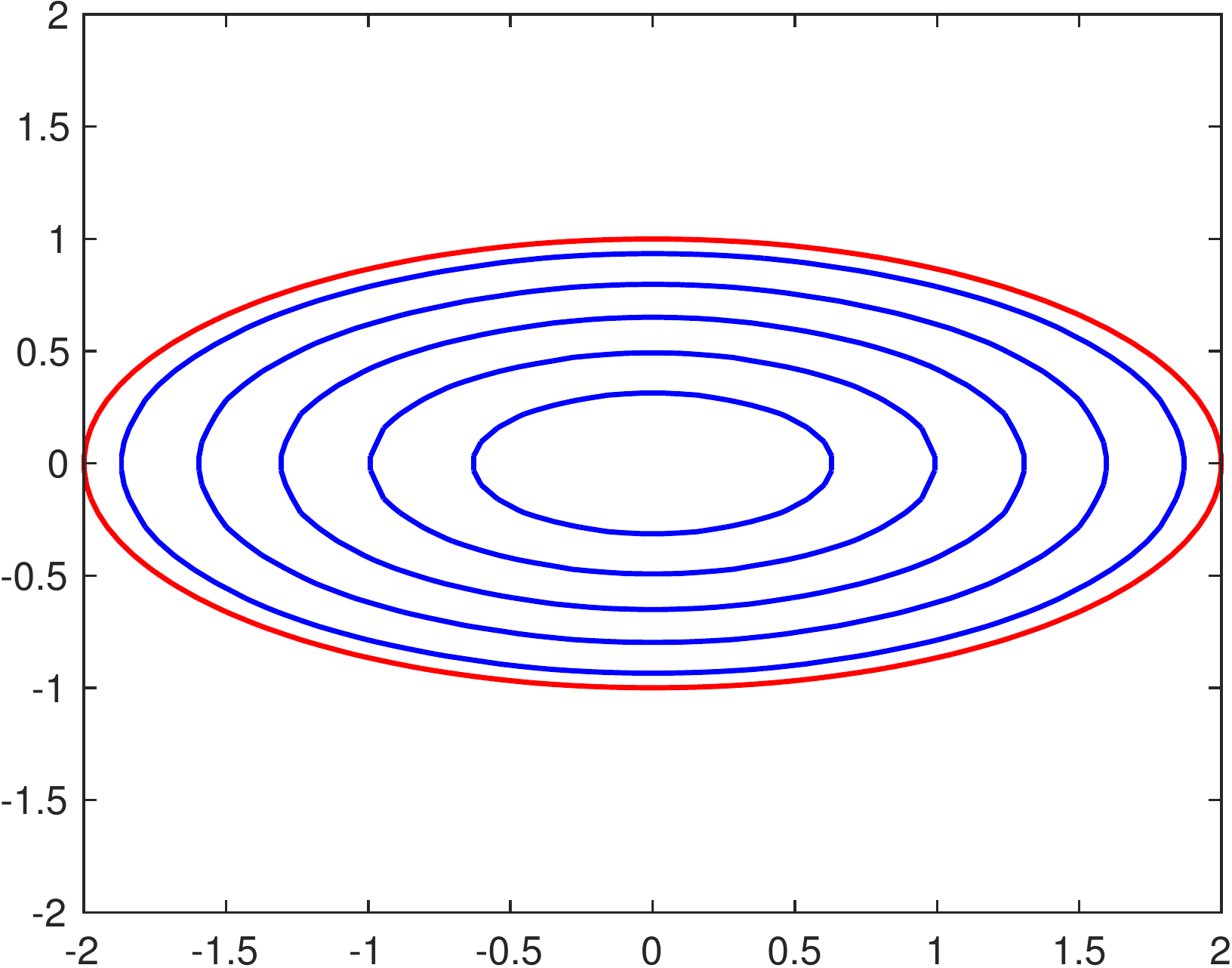} & \includegraphics[width=0.25\textwidth]{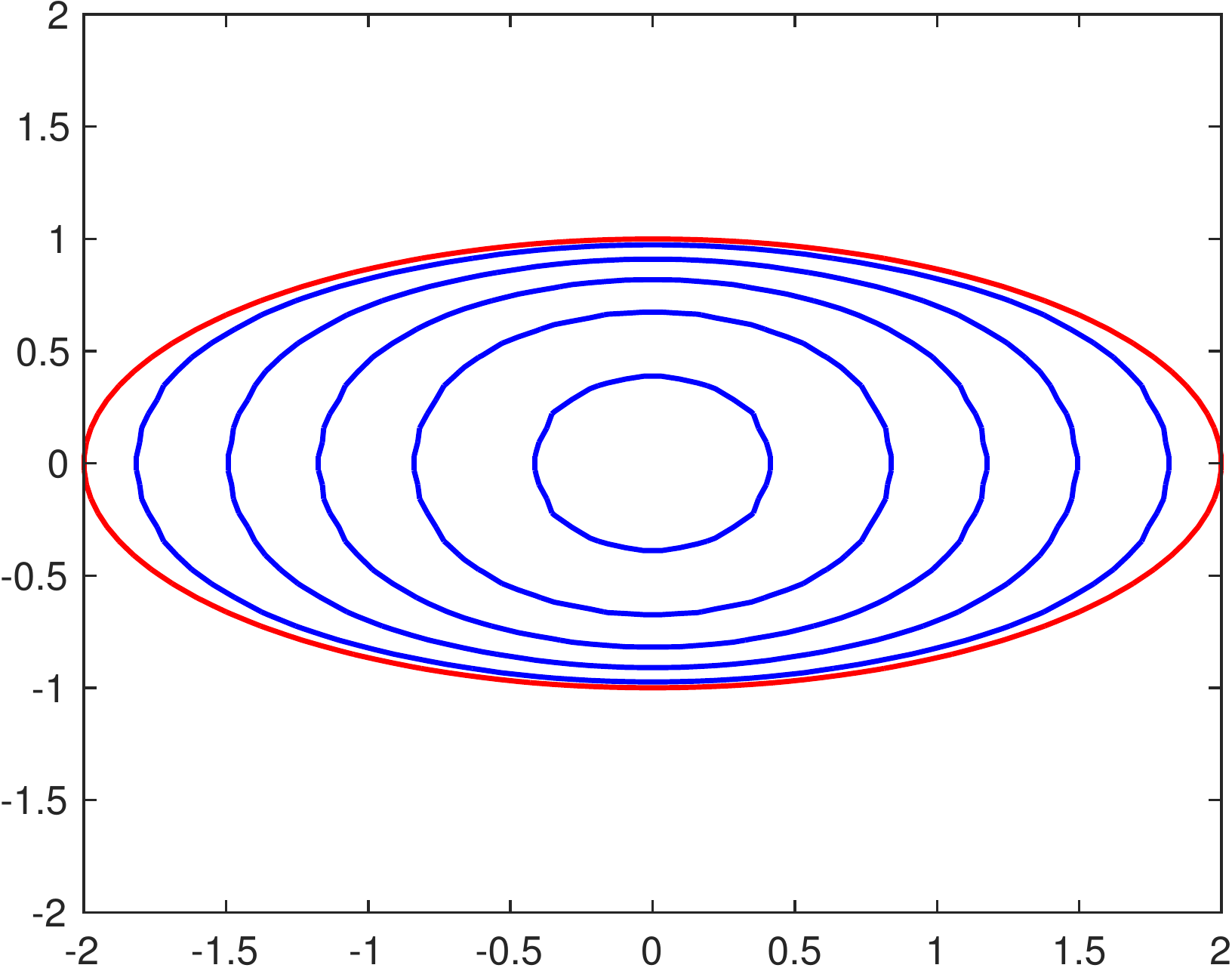}
\end{tabular}
\centering
\begin{tabular}{cc}
\includegraphics[width=0.25\textwidth]{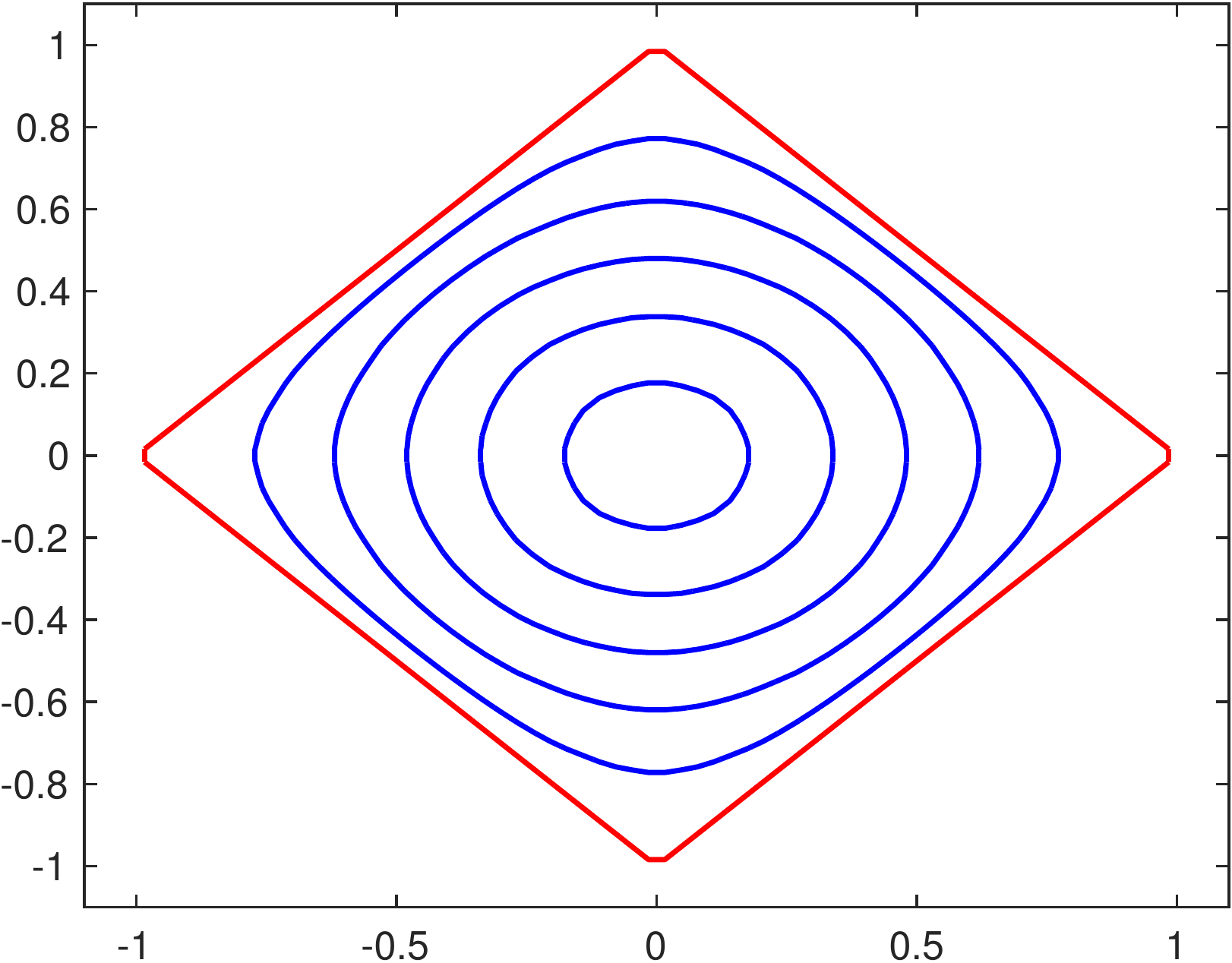} & \includegraphics[width=0.25\textwidth]{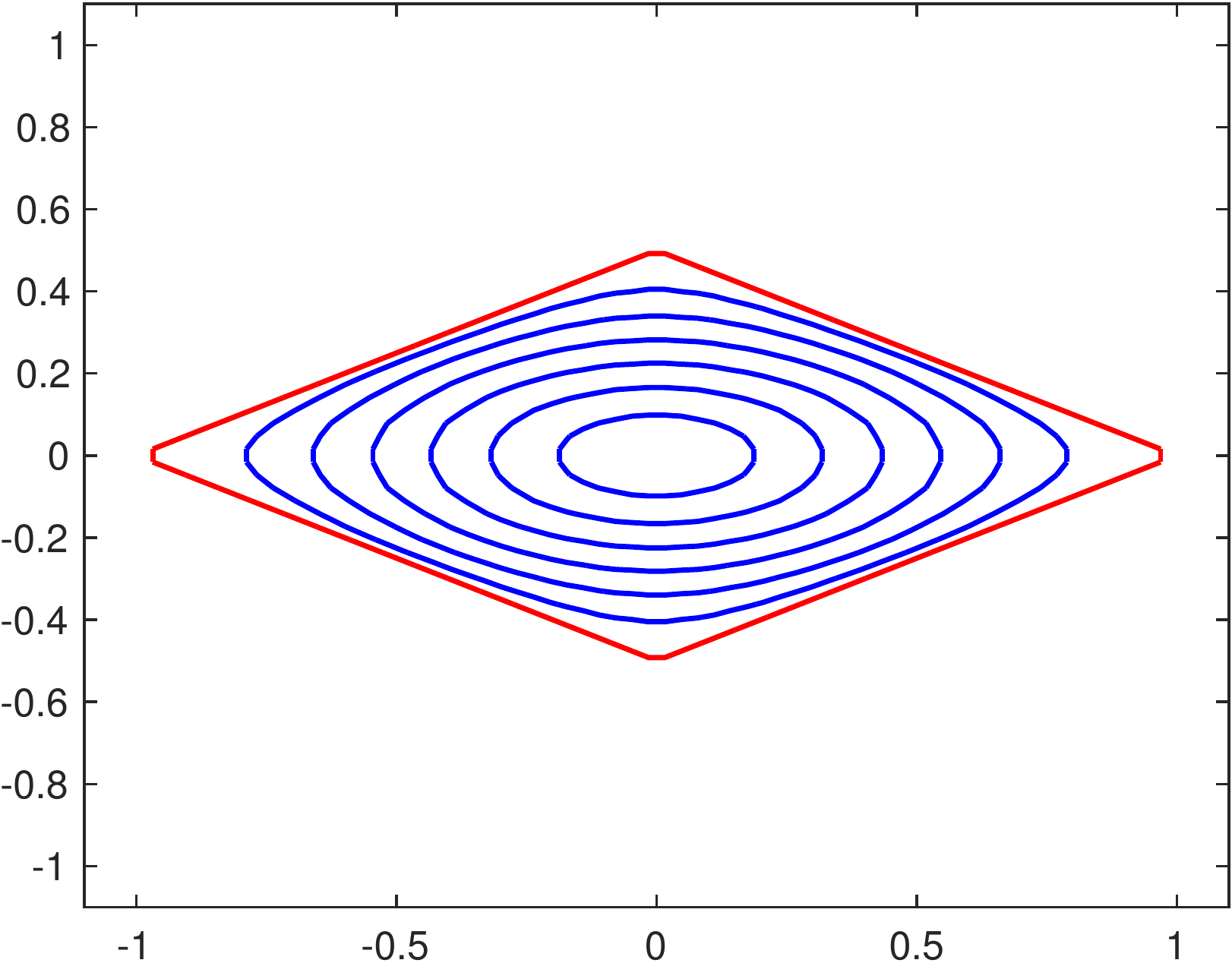}
\end{tabular}
\caption{(Top:) Evolution of
 an ellipse by (left) affine curvature, (right) mean curvature.
(Bottom:) Evolution by affine curvature of (left) a diamond and (right) a flatter diamond.}
\label{fig:ACvsMC}
\label{fig:exDiamond}
\end{figure}


\begin{figure}[htp]
\centering
\begin{tabular}{cccc}
\includegraphics[width=0.2\textwidth]{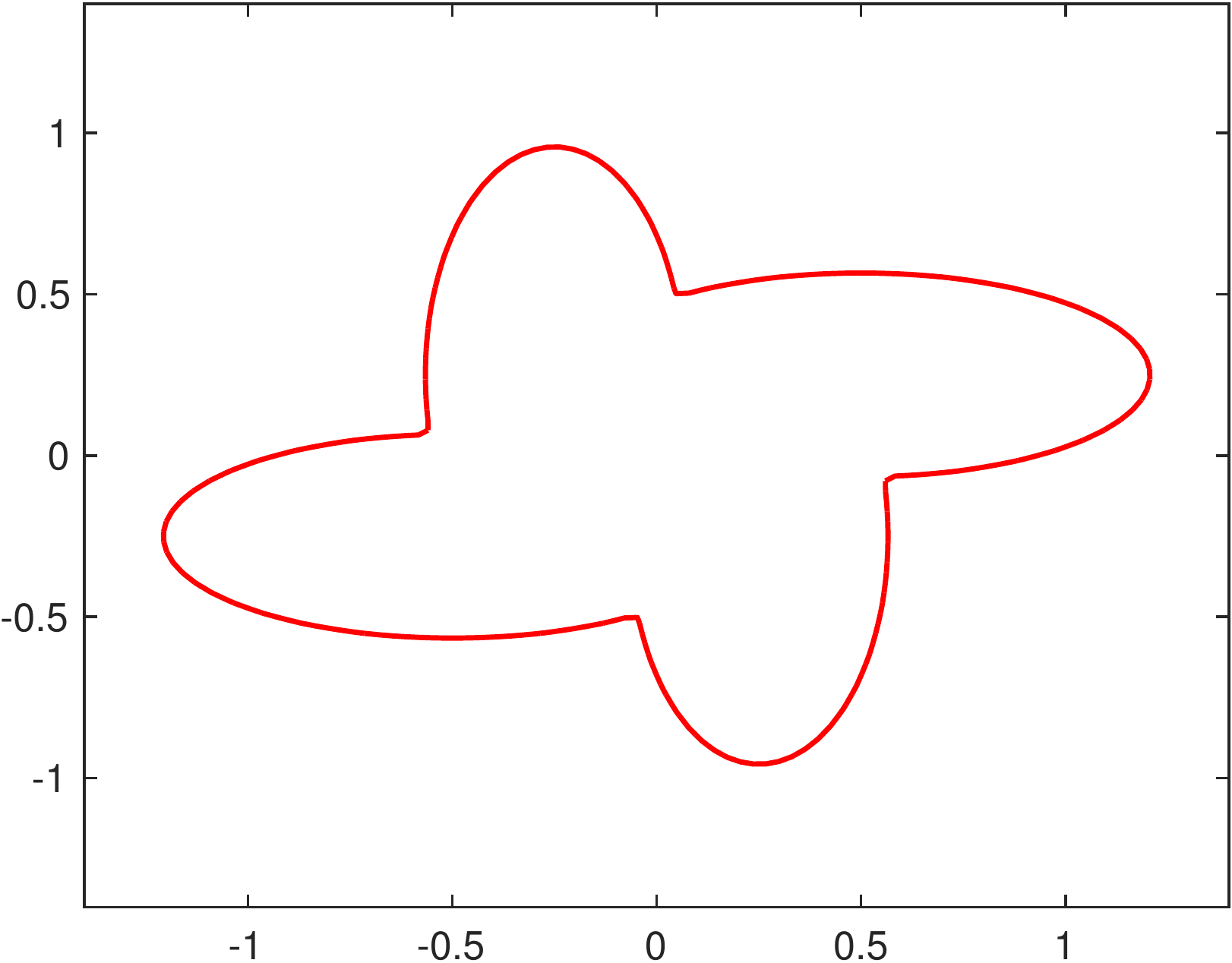} & \includegraphics[width=0.2\textwidth]{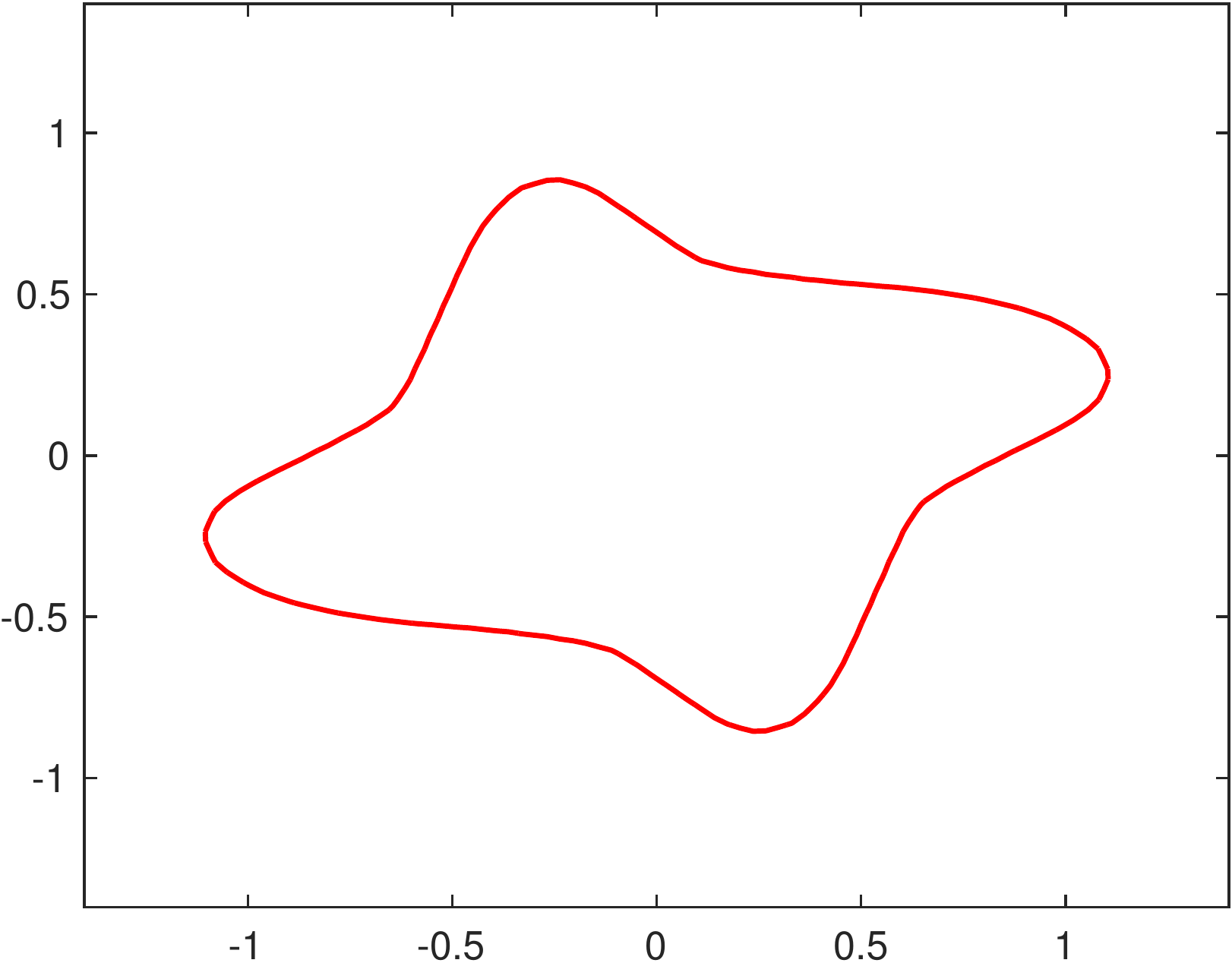} & \includegraphics[width=0.2\textwidth]{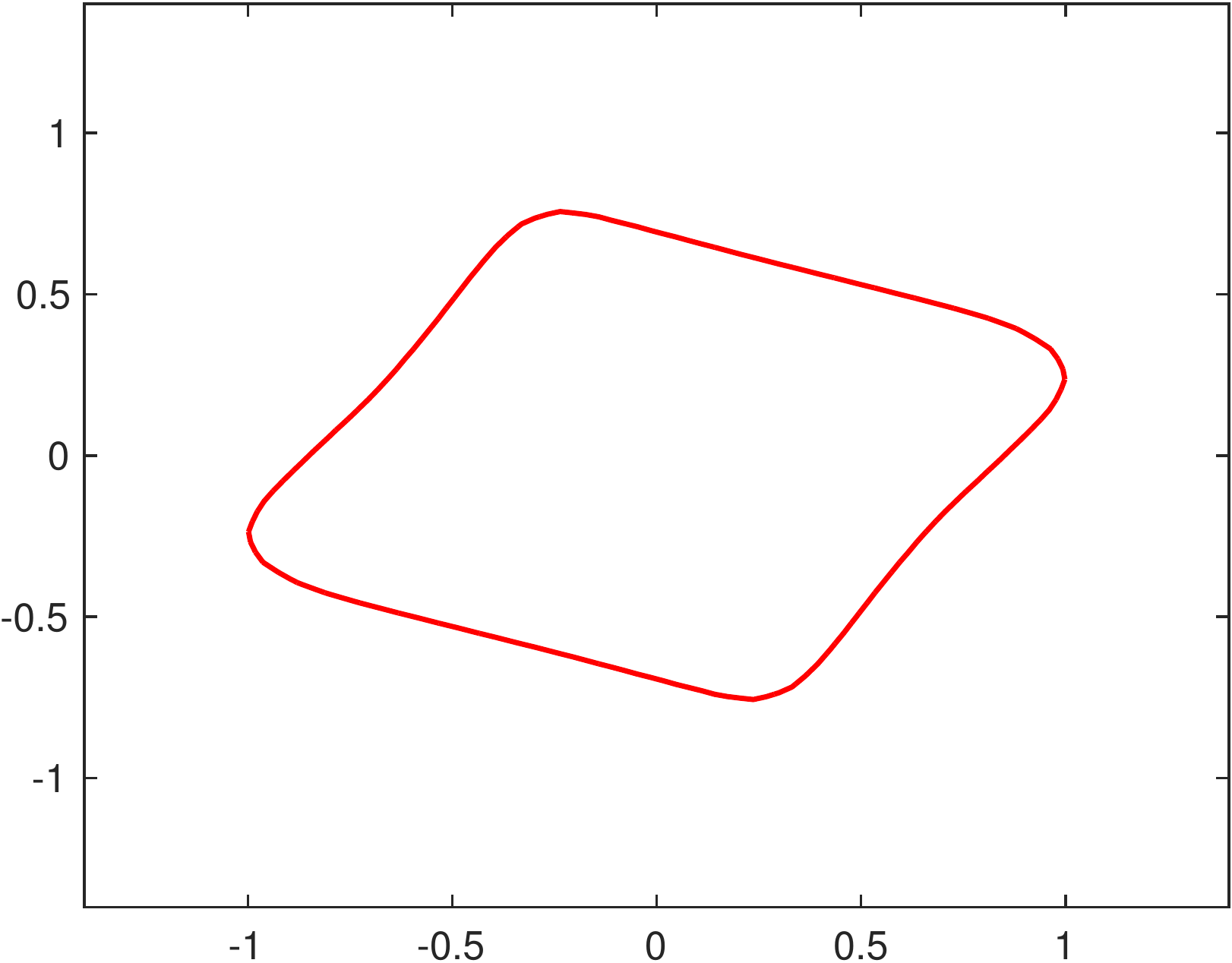} & \includegraphics[width=0.2\textwidth]{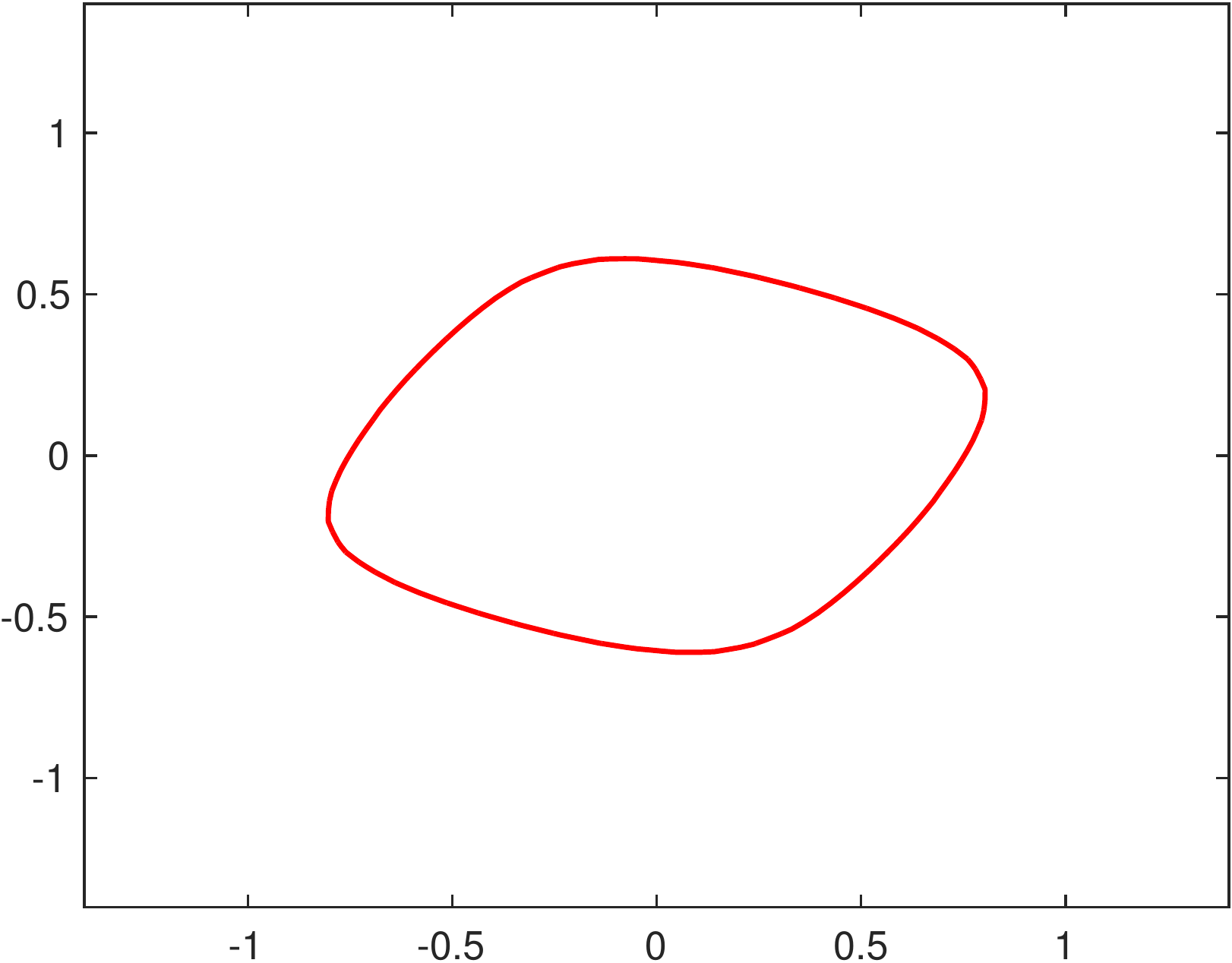}\\
\includegraphics[width=0.2\textwidth]{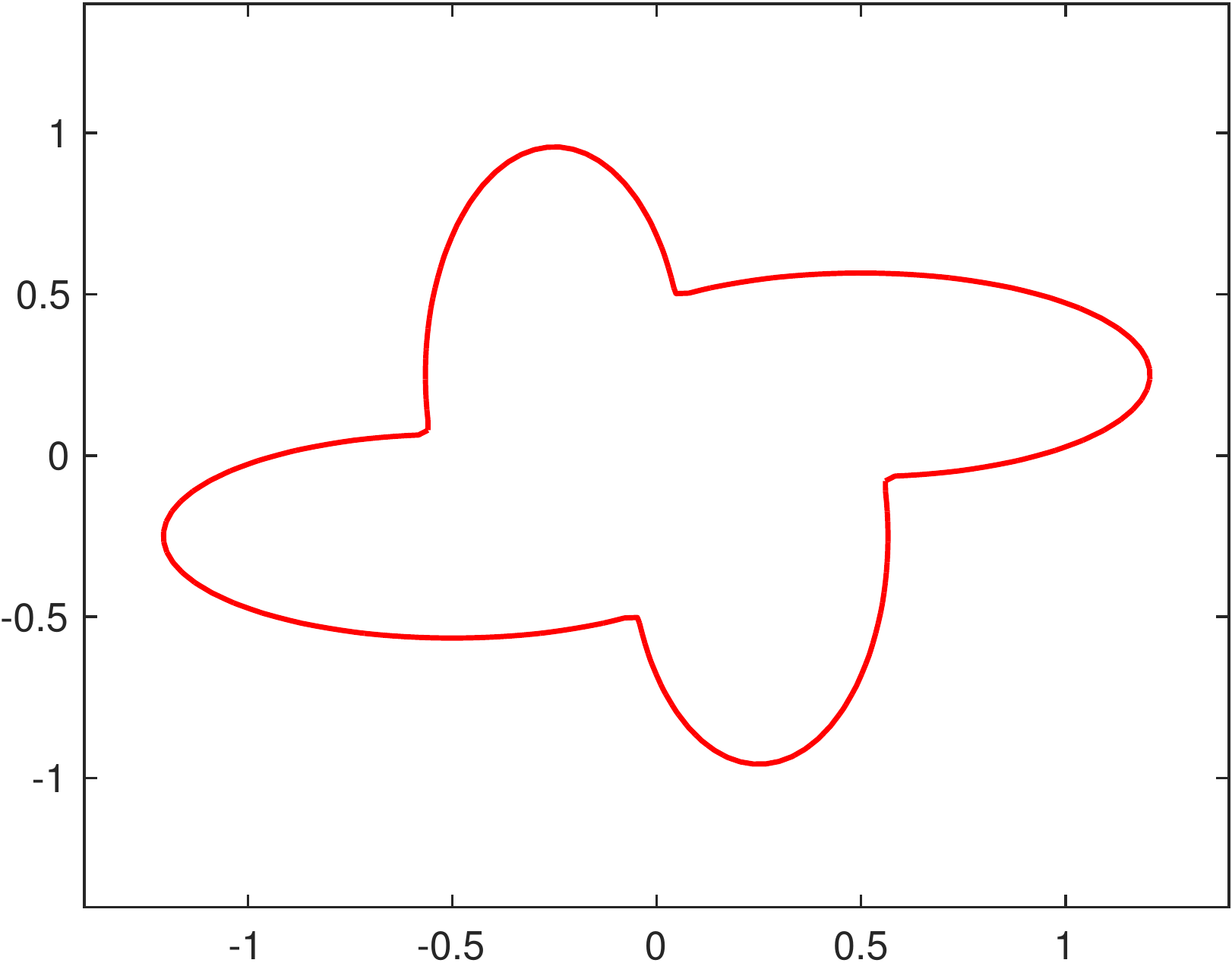} & \includegraphics[width=0.2\textwidth]{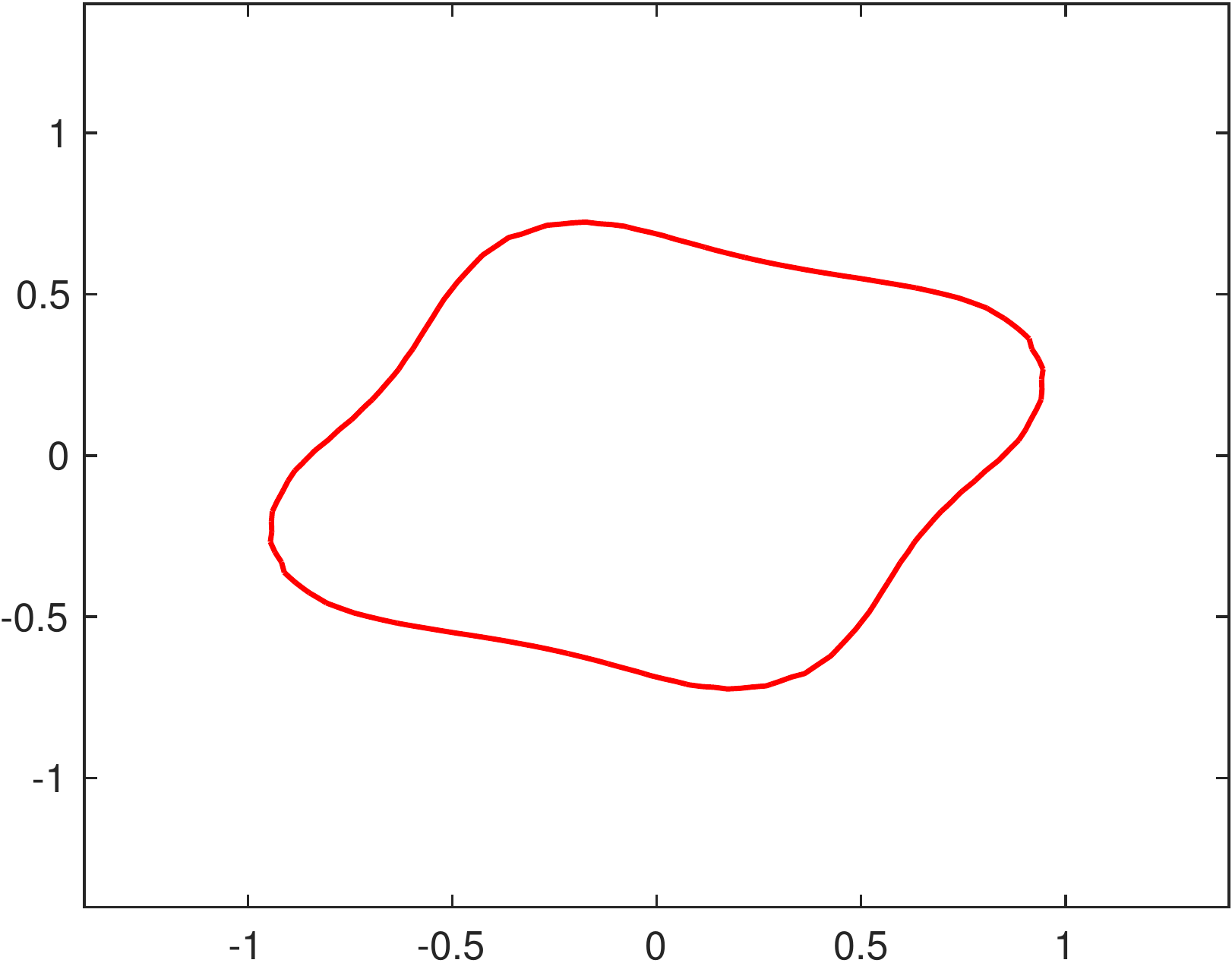} & \includegraphics[width=0.2\textwidth]{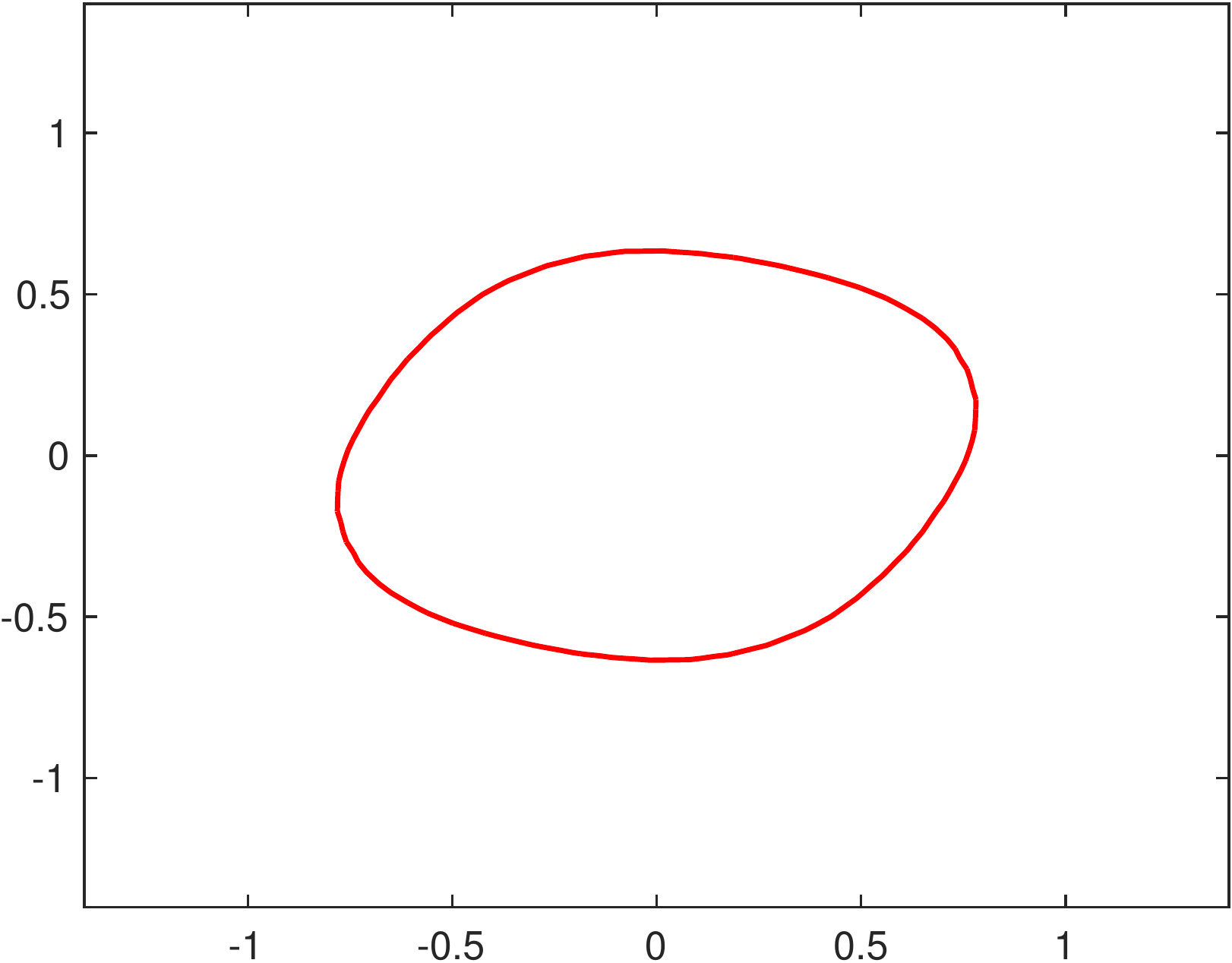} & \includegraphics[width=0.2\textwidth]{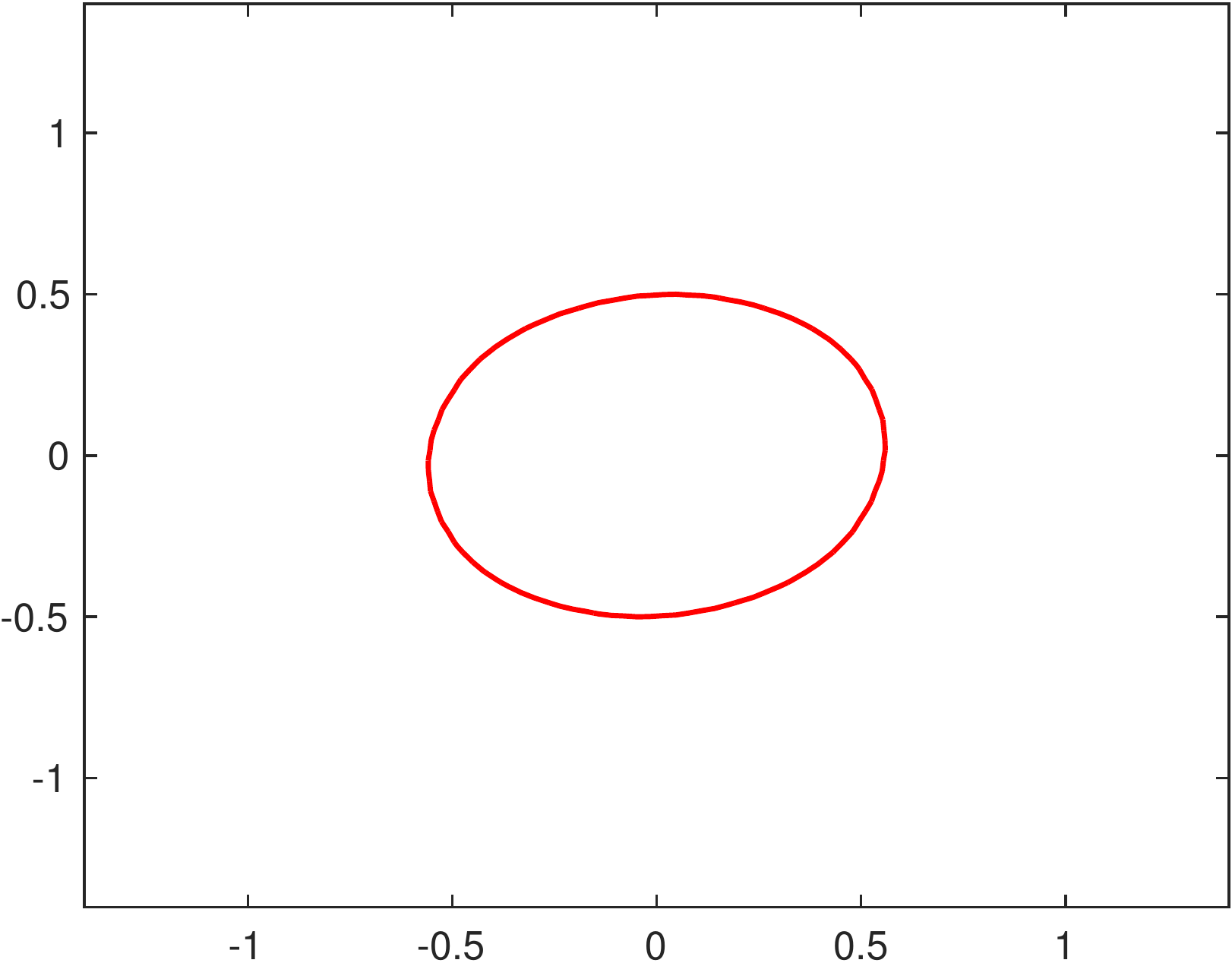}\\
\end{tabular}
\caption{Evolution of a fan-shape like curve under affine curvature motion (top) and mean curvature (bottom) at time $t~\in~\{0,0.05,0.1,.2\}$ (see Example~\ref{ex:clover} for more details).}
\label{fig:ex5}
\end{figure}

\begin{example}[Diamond]\label{ex:diamond}
We also compute the solution of \eqref{ACPDE} with initial condition
\[(a) \ u_0(x,y) = \min\left\{|x|+|y|-1,1\right\}, \qquad (b) \ u_0(x,y) = \min\left\{|x|+2|y|-1,1\right\}\]
and homogeneous Neumann boundary conditions. The exact solutions are not known. However, we do know that smooth convex curves evolving under affine curvature converge to ellipses until collapsing to a point and that is the behaviour we observed here (see Figure \ref{fig:exDiamond}). We took $[-2,2]^2$ as the computational domain on a $128 \times 128$ grid. 
As for the scheme, we chose again the narrow elliptic scheme.
In Figure \ref{fig:exDiamond}, we plot the zero level sets of the numerical solution from time $t = 0$ to $t = 0.5$ in increments of $0.1$ for example $(a)$ and from $t=0$ to $t=0.3$ in increments of $0.05$ for example $(b)$.
\end{example}

\begin{example}[Fan-shape curve]\label{ex:clover}
We also compute the solution of \eqref{ACPDE} and \eqref{MCevolution} with initial condition
\[u_0(x,y) = \min\left\{c_+^{(1)}(x,y),c_-^{(1)}(x,y),c_+^{(2)}(x,y),c_-^{(2)}(x,y),1\right\},\]
where
\[c_\pm^{(1)}(x,y) = \left(x\pm\frac{1}{2}\right)^2+5\left(y\pm\frac{1}{4}\right)^2-\frac{1}{2} \quad \text{and} \quad c_\pm^{(2)}(x,y) = 5\left(x\pm\frac{1}{4}\right)^2+\left(y\mp\frac{1}{4}\right)^2-\frac{1}{2}\]
and homogeneous Neumann boundary conditions. The exact solution is not known. As in the previous example, we took $[-2,2]^2$ as the computational domain on a $128 \times 128$ grid
 together with the narrow elliptic scheme.

Under the both curvature motions, the curve should initially become convex. At later times, under the affine curvature motion, the curve should evolve to an ellipse as opposed to a circle in the mean curvature case. In Figure \ref{fig:ex5}, we plot the zero level set of the numerical solutions at time $t \in \{0,0.05,0.1,.2\}$ and observe the exact behaviour described.
\end{example}

\subsection{Accuracy of stationary solutions}\label{subsec:accuracy}

We test the accuracy for the following Dirichlet problem 
\[\begin{cases}
\Aff[u] = f,	& \text{in } \Omega,\\
u = g,	& \text{on } \partial \Omega.
\end{cases}\]

Solutions are obtained by computing the steady state solution of 
$u_t = \Aff[u] - f$, with  $u(\cdot,t) = g$ on $\partial\Omega$ and $u(x,0) = u_0(x)$. 
We set $dt = 1/C_h$ where $C_h$ is given by \eqref{Lipconst} for the elliptic and filtered schemes and $dt = h^2/2$ for the standard finite difference scheme scheme.
These examples also demonstrate stability of the time dependent problem for elliptic and filtered scheme, as well as convergence to the steady solution, since we used the time dependent problem to obtain the solution. 

\begin{remark}
The unregularized schemes were also used. For these we set $dt = h^2/2$ for all examples, except for the filtered scheme in example (d) where we set $dt = h^2/8$. The results obtained were virtually the same.
\end{remark}

We set $u_0(x)$ to be the exact solution in a layer of seven grid points adjacent to the boundary. As a result, each discretization is initialized at the same set of grid points and therefore we can make a fair comparison of their accuracy. On the coarsest grid, we set $u_0(x) = 0$ on the interior grid points. 
To speed up calculations, on finer grids we set $u_0(x)$ to be interpolated solutions from the coarser grids at interior grid points. 

\begin{remark}
When considering Dirichlet boundary conditions we lose the comparison principle. However, we can always cap off the solution and impose homogeneous boundary conditions. 
\end{remark}

\begin{example}\label{ex:static}
We consider the following exact solutions
\begin{align*}
(a) & \quad u(x,y) = x^2+y^2, \quad f(x) = 2 \left(x^2 + y^2\right)^{\frac{1}{3}},\\
(b) & \quad u(x,y) = e^{x^2+y^2}, \quad f(x,y) = 2 \left(e^{3 (x^2 + y^2)}(x^2 + y^2)\right)^{\frac{1}{3}},\\
(c) & \quad u(x,y) = \left(x^2 + y^2\right)^{\frac{1}{3}}, \quad f(x) = \frac{4}{3},\\
(d) & \quad u(x,y) = \frac{\sin(2\pi x)\sin(2\pi y)}{4}, \quad f(x) = \frac{\pi^{\frac{4}{3}}}{2} \left(-\left(2+\cos(4\pi x)+\cos(4\pi y)\right) \sin(2\pi x)\sin(2\pi y)\right)^{\frac{1}{3}},
\end{align*}
with $\Omega = [-1,1]^2$. The solutions in (a),(b) and (d) are smooth, but the functions $f$ are only Holder continuous, $C^{0,2/3}$,  with singularities at the origin for (a) and (b), and at several points in example (d).   The solution in (c) is only $C^{0,2/3}$ with a singularity at the origin,  but in this case the function $f$ is constant. 
\end{example}

The results are presented in Table~\ref{table:static_error}. 
In Example~\ref{ex:static}$(a)$ the solution is a quadratic polynomial.  The accurate scheme $\Aff^a$ gives essentially machine precision, which is not surprising, since this scheme is second order accurate.  The filtered scheme $\Aff^f$ gives nearly the same accuracy (with a small discrepancy which could be eliminated by tuning the parameter).  On the other hand, both the narrow and wider elliptic scheme are much less accurate.   In contrast, in Example~\ref{ex:static}(b) the solution is smooth but not quadratic.  In this case we see that the accurate scheme appears to be converging to $\bO(h)$.  The elliptic schemes are less accurate and the filtered scheme is in between. 
In Example~\ref{ex:static}$(c)$ the solution is only Holder continuous.  In this case, the elliptic schemes have almost constant error near $0.01$ over the range of parameters used.   Despite the singular solution, the accurate scheme gives accuracy $\bO(h)$, and the filtered scheme does just as well.  
Example~\ref{ex:static}$(d)$ shows that the standard finite difference scheme does  not converge  as discussed in \autoref{sec:breakaccurate} (see Figure \ref{fig:exBreak}).   The narrow elliptic scheme has almost constant error $0.03$ and the wider elliptic scheme has error $0.1$ for the smallest grid, decreasing by a factor of two as the grid is refined.  The filtered scheme has the best accuracy, achieving an error of $0.001$ at the finest grid.

When comparing the different schemes, we have to account for the width of the stencil since for the elliptic schemes the wider schemes also have a larger spatial discretization error.  In general, the accuracy improved with the use of the wider stencil. Moreover, the filtered scheme performed as expected by providing better accuracy than the elliptic scheme and almost as good accuracy as the accurate scheme.   The final example shows that the standard finite difference scheme may not converge.  This may be due to the fact that, despite the solution being smooth,  there were multiple points where $f$ was singular.  Geometrically, this solution has several points where curves shrunk to zero. 

We also considered for Example~\ref{ex:static} the elliptic schemes with $n_\theta = 1$ and $2$. These provided poor accuracy with the directional resolution error easily dominating the spatial resolution error.  The errors do not decrease to zero as we decrease the grid size. This property is consistent with the theoretical results since convergence is only proved as both $h,d\theta \to 0$, which is indeed observed in the numerical results.  On the other hand, using the narrow and wider schemes, the accuracy of the elliptic scheme was good enough for the filtered scheme to give accuracy comparable to the accurate scheme in many examples.  In principle we still need to send $d\theta \to 0$, but in practice, the rate at which it needs to go to zero is much slower than $h$ when the filtered scheme is used. 

Finally, we point that the choice of $\epsilon$ is not an easy one and it has hard to pick an $\epsilon$ that yields optimal results in each example presented. By choosing $\epsilon$ larger, it is possible to achieve with the filtered schemes the accuracy of the standard schemes in Examples \ref{ex:static} (a), (b), (c). However, such choice is too permissive for Example \ref{ex:static} (d). As pointed out in \autoref{subsec:filter}, $\epsilon$ needs to be chosen small enough in order for the monotone scheme to be active to ensure convergence. (This is comparable to the CFL condition in time dependent equations: methods are convergent as $dt,h \to 0$, with $dt$ satisfying the CFL condition.)

\begin{table}[htp]
\centering
\footnotesize
\begin{tabular}{ccccc}
\hline
& \multicolumn{4}{l}{Errors and order, Example \ref{ex:static} $(a)$}\\ \cline{2-5}
N & Narrow Elliptic ($n_\theta = 3$) & Wide Elliptic ($n_\theta = 7$) & Standard & Filter \\ \hline
32 & \num{3.565e-02} & \num{3.978e-02} & \num{2.922e-07} & \num{5.104e-07} \\
64 & \num{2.555e-02} & \num{2.696e-02} & \num{3.019e-07} & \num{2.827e-07} \\
128 & \num{1.623e-02} & \num{1.678e-02} & \num{1.982e-07} & \num{1.984e-07} \\
256 & \num{1.050e-02} & \num{1.055e-02} & \num{9.293e-08} & \num{5.678e-05} \\
\hline \hline
& \multicolumn{4}{l}{Errors and order, Example \ref{ex:static} $(b)$}\\ \cline{2-5}
N & Narrow Elliptic $(n_\theta = 3)$ & Wide Elliptic ($n_\theta = 7$) & Standard & Filter \\ \hline
32 & \num{5.462e-02} & \num{7.666e-02} & \num{1.922e-03} & \num{1.985e-03} \\
64 & \num{5.207e-02} & \num{5.872e-02} & \num{9.875e-04} & \num{8.904e-03} \\
128 & \num{4.105e-02} & \num{3.725e-02} & \num{3.385e-04} & \num{7.262e-03} \\
256 & \num{3.173e-02} & \num{2.240e-02} & \num{9.798e-05} & \num{8.065e-03} \\
\hline \hline
& \multicolumn{4}{l}{Errors and order, Example \ref{ex:static} $(c)$}\\ \cline{2-5}
N & Narrow Elliptic $(n_\theta = 3)$ & Wide Elliptic ($n_\theta = 7$) & Standard & Filter \\ \hline
32 & \num{1.387e-02} & \num{2.386e-02} & \num{1.129e-02} & \num{1.129e-02} \\
64 & \num{1.310e-02} & \num{7.683e-03} & \num{4.625e-03} & \num{4.625e-03} \\
128 & \num{9.302e-03} & \num{8.202e-03} & \num{1.859e-03} & \num{1.872e-03} \\
256 & \num{6.445e-03} & \num{7.156e-03} & \num{7.426e-04} & \num{7.570e-04} \\
\hline \hline
& \multicolumn{4}{l}{Errors and order, Example \ref{ex:static} $(d)$}\\ \cline{2-5}
N & Narrow Elliptic $(n_\theta = 3)$ & Wide Elliptic ($n_\theta = 7$) & Standard & Filter \\ \hline
32 & \num{3.964e-02} & \num{9.813e-02} & - & \num{1.926e-02} \\
64 & \num{3.788e-02} & \num{4.679e-02} & - & \num{8.309e-03} \\
128 & \num{3.688e-02} & \num{2.367e-02} & - & \num{2.482e-03} \\
256 & \num{3.003e-02} & \num{1.798e-02} & - & \num{9.697e-04} \\
\hline \hline \\
\end{tabular}
\caption{Accuracy in the $l^\infty$ norm and order of convergence of the schemes for Example \ref{ex:static} with regularized schemes.}
\label{table:static_error}
\end{table}

\subsection{Accuracy for the time dependent problem}\label{sec:accuracytime}

Recall here that for general boundary conditions, the time dependent PDE \eqref{ACPDE} requires Neumann boundary conditions in order for uniqueness of viscosity solutions to hold.  We have already established (in the previous section) the stability of the numerical method. In the following examples, we test the accuracy of solutions, comparing two different wide stencil discretizations, along with regularized filtered discretization and the (generally unstable) standard finite difference method.  Consequently we test as well the accuracy using Dirichlet boundary conditions coming from the exact solution of Lemma \ref{lemma:ellipse}.

\begin{example}[Neumann BC]\label{ex:ellipseNeumann}
We consider the exact solution
\[u(x,y,t) = \min\left\{t + \frac{3}{4}\left(\frac{b}{a}x^2+\frac{a}{b}y^2\right)^{2/3}-1,0\right\}.\]
with $a = 2$ and $b = 1$ (see Lemma \ref{lemma:ellipse}).
By taking the minimum with $0$, we are imposing homogeneous Neumann boundary conditions, thus avoiding having to deal with boundary of the computation domain. We set $\Omega = [-3,3]^2$.

We display the numerical error in the $l^\infty$ norm at time $T = 0.1$ in Table \ref{table:ellipseTimeDependent}.
\end{example}

\begin{example}[Dirichlet BC]\label{ex:ellipseDirichlet}
We consider the exact solution
\[u(x,y,t) = t + \frac{3}{4}\left(\frac{b}{a}x^2+\frac{a}{b}y^2\right)^{2/3}.\]
with $a = 2$ and $b = 1$ (see Lemma \ref{lemma:ellipse}). This is the same example as in Example \ref{ex:ellipseNeumann}, but we consider Dirichlet boundary conditions instead. Therefore, we prescribe the exact solution at a seven point layer at the boundary for all time $t$. This way all schemes are initialized at the same set of grid points and thus we can compare their accuracy.

The error in the $l^\infty$ norm at time $T = 0.1$ is presented in Table \ref{table:ellipseTimeDependent}.
\end{example}

When using Neumman boundary conditions, we observed slow convergence, with errors near $0.01$, slowly decreasing as the grid size improved.  The accuracy improved as we went from the narrow to the wider elliptic scheme, and further improved as we went to the accurate and then the filtered scheme. In the case of Dirichlet boundary conditions, the accuracy is better overall and the error decrease is slightly faster. The difference is explained by the cap off done in the Neumman boundary conditions that introduces an additional error in the solution. However, this error does not propagate to the whole domain as the level sets of the solution shrink to its interior and so when we look at the error away from the cap off, we recover results very similar to the Dirichlet boundary conditions.

\begin{table}[htp]
\centering
\footnotesize
\begin{tabular}{ccccc}
\hline
& \multicolumn{4}{l}{Errors and order, Example \ref{ex:ellipseNeumann}}\\ \cline{2-5}
N & Narrow Elliptic ($n_\theta = 3$) & Wide Elliptic ($n_\theta = 7$) & Standard &  Filter \\
\hline
32 & \num{4.845e-02} & \num{6.691e-02} & \num{4.894e-02} & \num{4.888e-02} \\
64 & \num{4.432e-02} & \num{4.607e-02} & \num{2.977e-02} & \num{2.975e-02} \\
128 & \num{3.544e-02} & \num{2.823e-02} & \num{2.457e-02} & \num{2.438e-02} \\
256 & \num{2.971e-02} & \num{2.080e-02} & \num{1.747e-02} & \num{1.724e-02} \\
512 & \num{2.764e-02} & \num{1.652e-02} & \num{1.205e-02} & \num{1.182e-02} \\
\hline \hline \\
& \multicolumn{4}{l}{Errors and order, Example \ref{ex:ellipseDirichlet}}\\ \cline{2-5}
N & Narrow Elliptic ($n_\theta = 3$) & Wide Elliptic ($n_\theta = 7$) & Standard &  Filter \\
\hline
32 & \num{2.182e-02} & \num{1.449e-02} & \num{1.985e-02} & \num{1.985e-02} \\
64 & \num{1.435e-02} & \num{1.160e-02} & \num{1.279e-02} & \num{1.279e-02} \\
128 & \num{9.580e-03} & \num{7.517e-03} & \num{5.566e-03} & \num{5.567e-03} \\
256 & \num{6.404e-03} & \num{4.854e-03} & \num{2.442e-03} & \num{2.409e-03} \\
512 & \num{6.090e-03} & \num{4.288e-03} & \num{1.036e-03} & \num{1.002e-03} \\
\hline \hline \\
\end{tabular}
\caption{Error in the $l^\infty$ norm of the whole computational domain at time $t=0.1$ for the time dependent Example  \ref{ex:ellipseNeumann} and \ref{ex:ellipseDirichlet}.}
\label{table:ellipseTimeDependent}
\end{table}

\subsection{Numerical study of the morphology and affine invariance properties}\label{subsec:invariance}

In this section, we test if our proposed schemes satisfy numerically the morphology and affine invariance properties that \eqref{ACPDE} satisfies (see Theorem \ref{MoisanTheorem}).

\begin{example}\label{ex:morphology}
In this example, we test numerically if the schemes presented here satisfy the morphology property of the affine curvature evolution (\eqref{morphology} in Theorem \ref{MoisanTheorem}). We consider two examples: $(a)$ $g_1(x) = e^x$, $(b)$ $g_2(x) = x^3$.
We take
\[u_0(x,y) = \min\left\{\left(\frac{x}{2}\right)^2+y^2-1,0\right\}\]
and compare $\Phi_t(g_v\circ u_0)$ with $g_v \circ \Phi_t(u_0)$ at $t=1$ for $v = 1,2$. We took $[-3,3]^2$ as the computational domain with homogeneous Neumann boundary conditions. The results are displayed in Table \ref{table:Morphology}. The difference in the $l^\infty$ norm is one order of magnitude smaller than the observed accuracy for the schemes in \autoref{subsec:accuracy}. Based on these examples, the morphology property seems to hold numerically.


\begin{table}[htp]
\centering
\footnotesize
\begin{tabular}{ccccc}
\hline
& \multicolumn{4}{l}{Difference in the $l^\infty$ norm, Example \ref{ex:morphology} (a)}\\ \cline{2-5}
N & Narrow Elliptic $(n_\theta = 3)$ & Wide Elliptic ($n_\theta = 7$) & Standard &  Filter \\
\hline
32 & \num{8.943e-03} & \num{1.037e-02} & \num{3.419e-03} & \num{3.446e-03} \\
64 & \num{5.709e-03} & \num{6.111e-03} & \num{1.954e-03} & \num{1.970e-03} \\
128 & \num{4.061e-03} & \num{3.257e-03} & \num{1.061e-03} & \num{1.109e-03} \\
256 & \num{3.115e-03} & \num{1.871e-03} & \num{5.907e-04} & \num{6.109e-04} \\
512 & \num{2.604e-03} & \num{1.161e-03} & \num{3.207e-04} & \num{3.397e-04} \\
\\ \hline
& \multicolumn{4}{l}{Difference in the $l^\infty$ norm, Example \ref{ex:morphology} (b)}\\ \cline{2-5}
N & Narrow Elliptic $(n_\theta = 3)$ & Wide Elliptic ($n_\theta = 7$) & Standard &  Filter \\
\hline
32 & \num{3.450e-02} & \num{7.023e-02} & \num{6.947e-03} & \num{6.947e-03} \\
64 & \num{1.730e-02} & \num{2.842e-02} & \num{1.881e-03} & \num{1.877e-03} \\
128 & \num{1.032e-02} & \num{9.355e-03} & \num{6.254e-04} & \num{6.333e-04} \\
256 & \num{6.894e-03} & \num{4.307e-03} & \num{2.283e-04} & \num{2.307e-04} \\
512 & \num{5.372e-03} & \num{2.316e-03} & \num{8.343e-05} & \num{8.506e-05} \\
\end{tabular}
\caption{Difference in the $l^\infty$ norm between $\Phi_t(g_v \circ u_0)$ and $g_v\circ \Phi_t(u_0)$ for $v=1,2$ for Example \ref{ex:morphology}.}
\label{table:Morphology}
\end{table}
\end{example}

\begin{example}\label{ex:invariance}
In this example we do a qualitative test of the affine invariance property, which in practice is what one needs for applications in image analysis. In order to do so we plot the level sets of the affine invariant motion by curvature \eqref{ACPDE} with
\[u(x,y) = \left(\frac{x}{2}\right)^2+y^2-1\]
and $u \circ \phi$ as the initial solutions. For the affine transformations $\phi(\x) = A\x$, we consider
\[(a) \text{ (rotation by $\pi/4$) } A = \begin{bmatrix}\frac{\sqrt{2}}{2} & \frac{\sqrt{2}}{2}\\ -\frac{\sqrt{2}}{2} & \frac{\sqrt{2}}{2}\end{bmatrix}, \quad (b) \: A = \begin{bmatrix}1 & 1\\ -1 & 1\end{bmatrix}.\]
We take $[-5,5]^2$ as the computational domain on a $256 \times 256$ grid. 

In Figure \ref{fig:ExLevelSet} we plot the zero level set of $\Phi_t(u\circ \phi)$ and $\left(\Phi_{t(\det \phi)^{2/3}}(u)\right) \circ \phi$ at $t=1$. As expected, the standard finite difference scheme and the filtered scheme provided the best results, being indistinguishable to the naked eye. For long time, the elliptic scheme did not provide as good results, a consequence of its lower accuracy (see~\autoref{subsec:accuracy}).  (For shorter times the curves were very close). 


\begin{figure}[htp]
\centering
\begin{tabular}{ccc}
\includegraphics[width=0.3\textwidth]{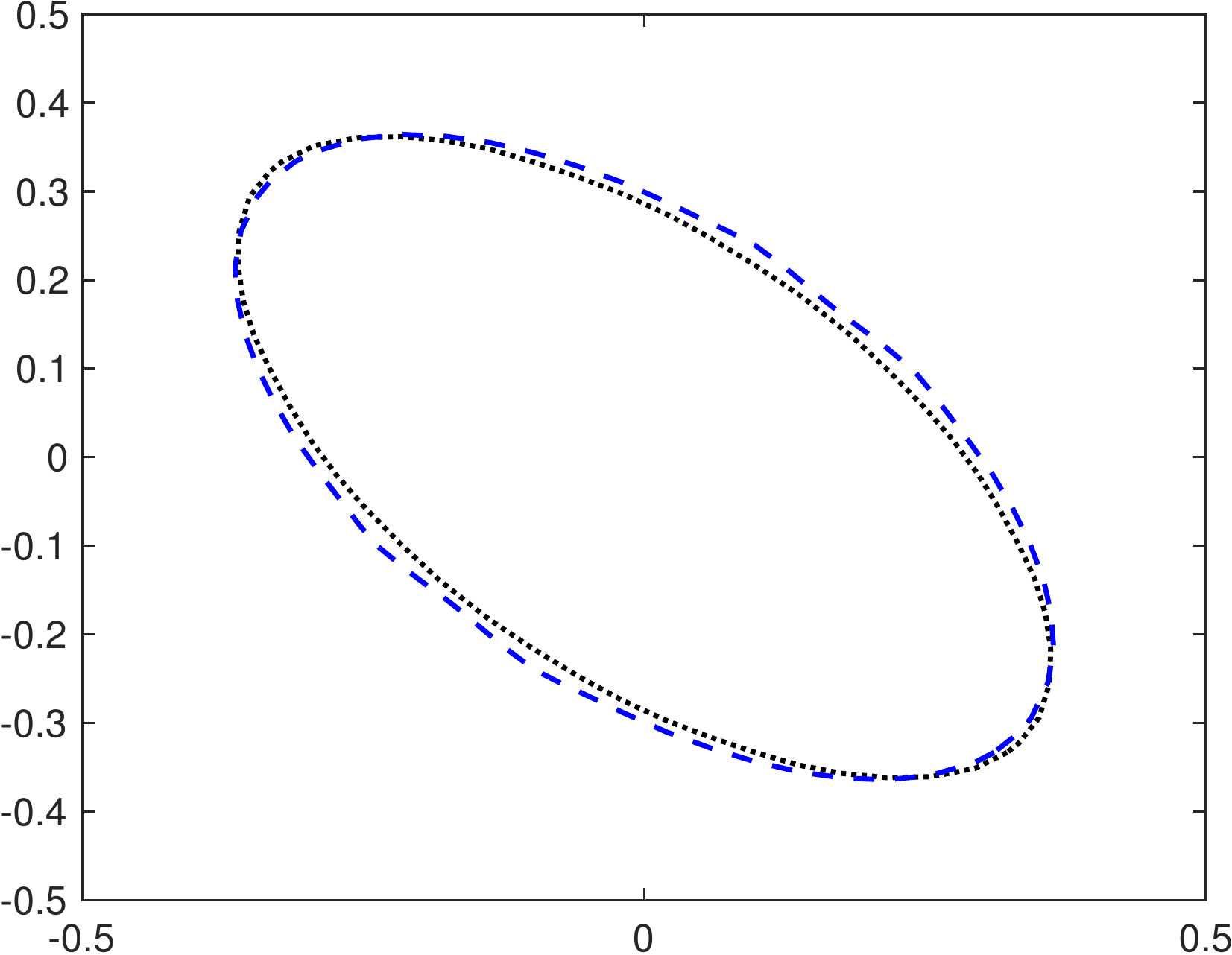} 
& \includegraphics[width=0.3\textwidth]{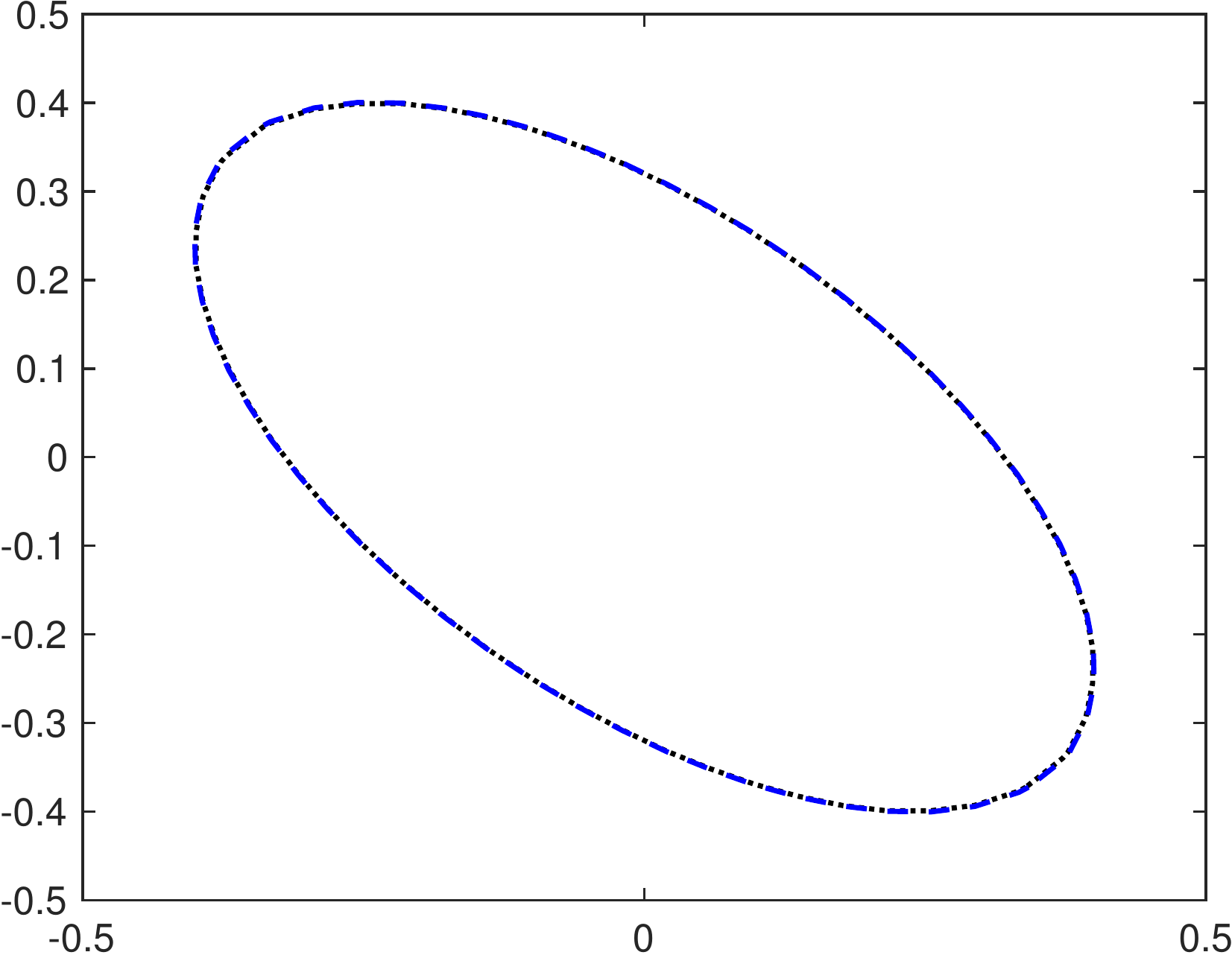}\\
\includegraphics[width=0.3\textwidth]{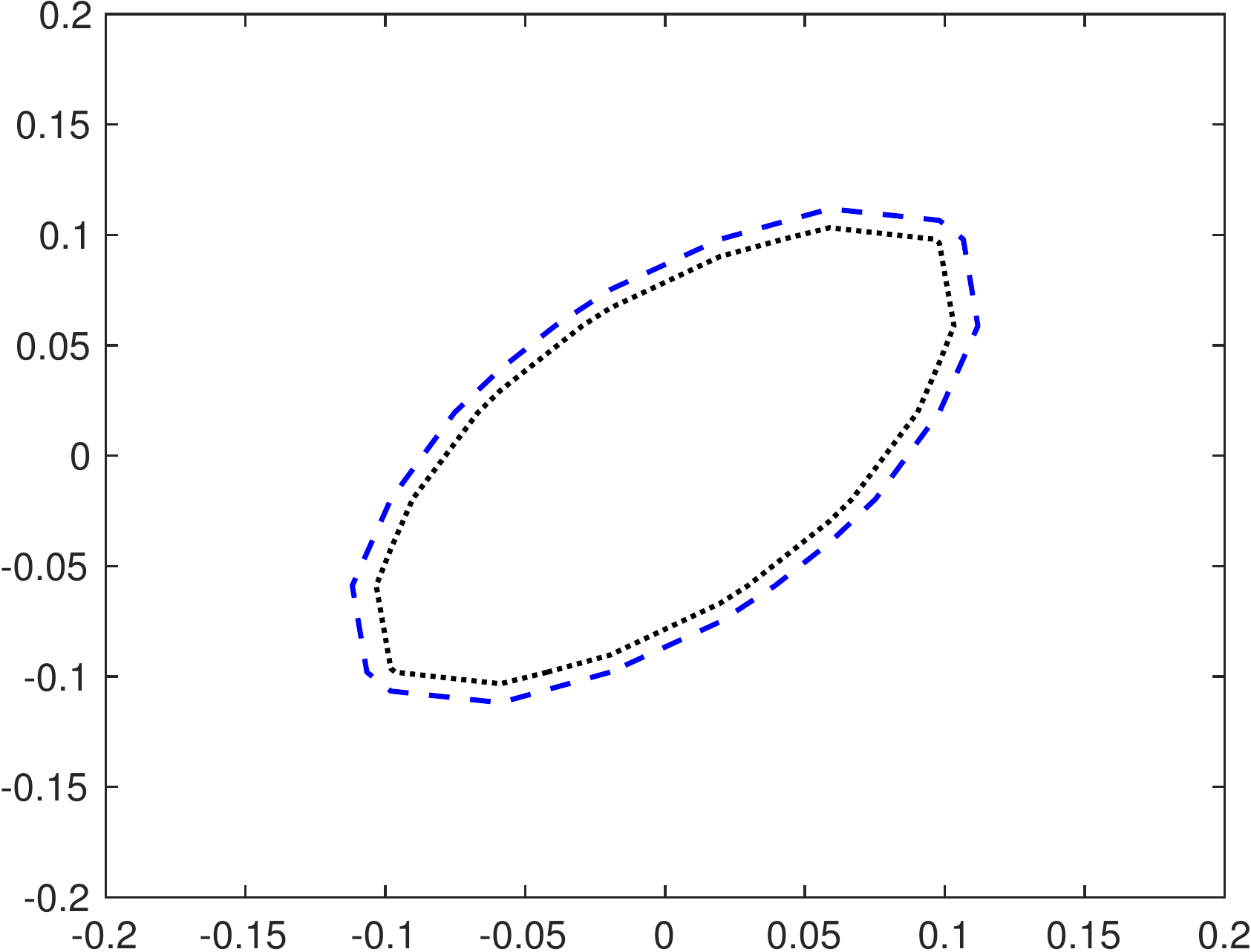} 
& \includegraphics[width=0.3\textwidth]{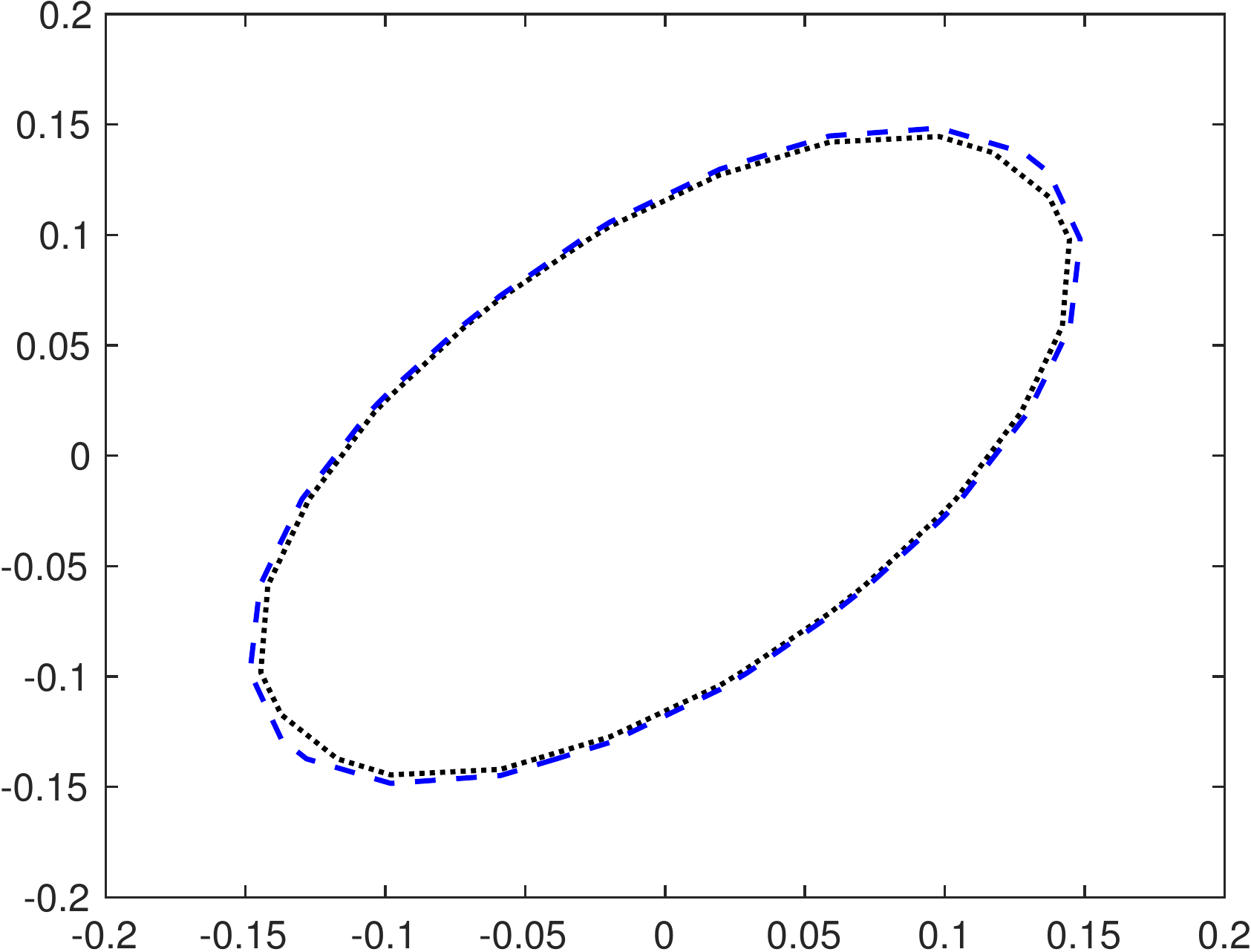}\\
\end{tabular}
\caption{Plot of the zero level sets in Example \ref{ex:invariance} of $\Phi_t(u\circ \phi)$ and $\left(\Phi_{t(\det \phi)^{2/3}}(u)\right) \circ \phi$ for regular elliptic scheme (left), 
regular filtered scheme (right) at time $t=1$ with $\phi$ given by $(a)$ (top) and at time $t = 0.7$ with $\phi$ given by $(b)$ (bottom).}
\label{fig:ExLevelSet}
\end{figure}

We point out that when the affine transformation is a rotation by a multiple of $\pi/2$ or a reflection over a line $L$ that makes an angle multiple of $\pi/4$ with the x-axis, we obtain essentially machine precision (The difference in the $l^\infty$ norm was of the order $10^{-10}$.). This is expected since our stencil is invariant under these transformations.
\end{example}

\section{Conclusions}

We presented a convergent finite difference discretization of the PDE for motion of level sets by affine curvature in two dimensions.
Computational examples demonstrated that the standard finite difference method is unstable, which motivates the need for a convergent method. 

The foundation of the scheme used an existing wide stencil  discretization of the mean curvature operator, combined with an elliptic discretization of the positive and negative eikonal operators, $\pm \abs{\grad u}$.   However, explicit time discretizations require Lipschitz continuous operators, which the affine curvature operator fails to be. Thus, we approximated it by a Lipschitz continuous regularization.   In theory, the explicit Euler discretization is stable using a time step $dt \leq \bO(h^{22/9})$ , with the constant determined by the width of the stencil.   In practice, we achieved numerically equivalent results using $dt = h^2/2$  and without the regularization, although there is no proof of stability with the less restrictive time step. 

A careful choice of the regularization parameters allowed for the regularized elliptic scheme to maintain the same order of accuracy as the unregularized scheme, while being provably convergent. The lower accuracy of both schemes, which results from the singularity of the operator, is overcomed by the use of the filtered schemes, which essentially attain the accuracy of the standard finite difference schemes, while being provably convergent.

Simulations demonstrated the geometric properties of the PDE were nearly preserved by the numerical solutions, including affine invariance, morphological properties, and the accurate representation of the shrinking ellipses.   Simulations validated the convergence of the elliptic scheme, and the improved accuracy of the filtered scheme. 

\bibliographystyle{amsalpha}
\bibliography{AffineCurvature}

\end{document}

%% file: Preamble.tex
\usepackage{url}
\usepackage{verbatim}
\usepackage{showtags}

\usepackage{amssymb}
\usepackage{alltt}
\usepackage{graphicx}
\graphicspath{{figures/}}
\usepackage{siunitx}
\usepackage{amsthm}
\usepackage{enumerate}
\usepackage[colorlinks, linkcolor=blue,  citecolor=blue, urlcolor=blue]{hyperref}%

\newtheorem{theorem}{Theorem}[section]

\newtheorem{lemma}[theorem]{Lemma}

\newtheorem{definition}[theorem]{Definition}
\theoremstyle{remark}
\newtheorem{remark}{Remark}[section]
\theoremstyle{definition}
\newtheorem{example}{Example}[section]

\DeclareMathOperator{\dist}{dist}
\newcommand{\norm}[1]{\Vert#1\Vert}

\newcommand{\abs}[1]{\vert#1\vert}
\newcommand{\Abs}[1]{\left\vert#1\right\vert}

\newcommand{\bq}{\begin{equation}}
\newcommand{\eq}{\end{equation}}
\newcommand{\R}{\mathbb{R}}

\newcommand{\Rn}{\R^n}
\newcommand{\e}{\epsilon}
\newcommand{\bO}{\mathcal{O}}

\newcommand{\curve}{\mathcal{C}}

\newcommand{\T}{\mathbf{t}}
\newcommand{\N}{\mathbf{n}}

\newcommand{\x}{\mathbf{x}}

\newcommand{\td}{\text{d}}
\newcommand{\Aff}{F}
\newcommand{\divergence}{\textup{div}}

\newcommand{\median}{\text{median}}
\newcommand{\grad}{\nabla}
\newcommand{\tr}{\text{tr}}
\newcommand{\vi}{\mathbf{v_i}}

\newcommand{\st}{{r}}

\newcommand{\dx}{h}
\newcommand{\Sk}{{F^\dx}}
\newcommand{\Grd}{G^\dx}
\newcommand{\Zb}{\mathbb{Z}}

\DeclareMathOperator{\sgn}{sgn}